\pdfoutput=1
\RequirePackage{ifpdf}
\ifpdf 
\documentclass[pdftex]{sigma}
\else
\documentclass{sigma}
\fi

\usepackage{enumitem}
\usepackage{multirow}

\newfont{\fra}{eufm10 scaled 1095}
\newfont{\Bb}{msbm10 scaled 1095}
\newfont{\Bbg}{msbm10 scaled 1280}
\newcommand\CC{{\mbox{\Bb C}}}
\newcommand\RR{{\mbox{\Bb R}}}
\newcommand\NN{{\mbox{\Bb N}}}
\newcommand\ZZ{{\mbox{\Bb Z}}}
\newcommand\SSS{{\mbox{\Bb S}}}

\newcommand\Z{{\mathbb Z}}
\newcommand\R{{\mathbb R}}

\newcommand\fg{{\frak{g}}}

\newcommand\bm{\textbf{{\rm \textbf{m}}}}
\newcommand\bs{\textbf{{\rm \textbf{s}}}}
\newcommand\bF{\textbf{{\rm \textbf{F}}}}
\newcommand\cC{{\mathcal C}}
\newcommand\cV{{\mathcal V}}
\newcommand\cW{{\mathcal W}}
\newcommand\cH{{\mathcal H}}
\newcommand\cM{{\mathcal M}}
\newcommand\cA{{\mathcal A}}
\newcommand\cF{{\mathcal F}}
\newcommand\cO{{\mathcal O}}
\newcommand\cB{{\mathcal B}}
\newcommand\ph{\varphi}

\newcommand\osc{\frak o \frak s \frak c}

\newcommand{\Aut}{\mathop{{\rm Aut}}}
\newcommand{\GL}{\mathop{{\rm GL}}}
\newcommand{\SL}{\mathop{{\rm SL}}}
\newcommand{\SO}{\mathop{{\rm SO}}}
\newcommand{\Osc}{{\rm Osc}}

\newcommand{\Id}{\mathop{{\rm id}}}

\newcommand{\tr}{\mathop{{\rm tr}}}

\newcommand{\sgn}{\mathop{{\rm sgn}}}

\renewcommand{\Re}{\mathop{{\rm Re}}}
\renewcommand{\Im}{\mathop{{\rm Im}}}
\newcommand{\diag}{\mathop{{\rm diag}}}

\newcommand{\rem}{\mathop{{\rm rem}}}
\newcommand{\Span}{\mathop{{\rm span}}}
\newcommand{\cSpan}{{\overline{\rm span}}}

\newcommand{\lcm}{\mathop{{\rm lcm}}}
\newcommand{\ord}{{\rm ord}}

\newcommand\la{\langle}
\newcommand\ra{\rangle}

\newcommand{\benur}{\begin{enumerate}[label=$(\roman*)$]\itemsep=0pt}

\numberwithin{equation}{section}

\newtheorem{Theorem}{Theorem}[section]
\newtheorem{Corollary}[Theorem]{Corollary}
\newtheorem{Lemma}[Theorem]{Lemma}
\newtheorem{Proposition}[Theorem]{Proposition}
 { \theoremstyle{definition}
\newtheorem{Definition}[Theorem]{Definition}

\newtheorem{Example}[Theorem]{Example}
\newtheorem{Remark}[Theorem]{Remark} }

\begin{document}

\newcommand{\arXivNumber}{1912.00050}

\renewcommand{\PaperNumber}{051}

\FirstPageHeading

\ShortArticleName{Spectra of Compact Quotients of the Oscillator Group}

\ArticleName{Spectra of Compact Quotients of the Oscillator Group}

\Author{Mathias FISCHER and Ines KATH}
\AuthorNameForHeading{M. Fischer and I. Kath}
\Address{Institut f\"ur Mathematik und Informatik der Universit\"at Greifswald, \\
Walther-Rathenau-Str.~47, D-17489 Greifswald, Germany}
\Email{\href{mailto:mathias.fischer@uni-greifswald.de}{mathias.fischer@uni-greifswald.de}, \href{mailto:ines.kath@uni-greifswald.de}{ines.kath@uni-greifswald.de}}

\ArticleDates{Received September 28, 2020, in final form April 24, 2021; Published online May 13, 2021}

\Abstract{This paper is a contribution to harmonic analysis of compact solvmanifolds. We~consider the four-dimensional oscillator group ${\rm Osc}_1$, which is a semi-direct product of~the three-dimensional Heisenberg group and the real line. We~classify the lattices of~${\rm Osc}_1$ up to inner automorphisms of~${\rm Osc}_1$. For every lattice $L$ in~${\rm Osc}_1$, we~compute the decomposition of the right regular representation of~${\rm Osc}_1$ on~$L^2(L\backslash{\rm Osc}_1)$ into irreducible unitary representations. This decomposition allows the explicit computation of the spectrum of the wave operator on the compact locally-symmetric Lorentzian manifold $L\backslash {\rm Osc}_1$.}


\Classification{22E40; 22E27; 53C50}

{\small \tableofcontents}

\section{Introduction}
Let $G$ be a Lie group and $L$ be a cocompact discrete subgroup in~$G$. Then we can consider the right regular representation of~$G$ on~$L^2(L\backslash G)$, which is unitary. This representation can be decomposed into irreducible unitary representations. It~is a classical result that this decomposition is a discrete sum and each summand appears with finite multiplicity. We~will call the set of~irreducible representations that appear in this decomposition together with their multiplicities the {\it spectrum} of~$L\backslash G$. It~allows to compute the spectrum of differential operators on~$L\backslash G$ that come from bi-invariant operators on~$G$.

One of the basic questions in harmonic analysis is to determine the spectrum of~$L\backslash G$. For semisimple $G$, this question is very deep and has been studied for decades. For arithmetic $L$, it~is related to the Langlands program. In~any case, only for very few irreducible representations, their multiplicities can be explicitly determined. For nilpotent $G$, the situation is quite different. Here explicit computations are possible and there are closed formulas for the multiplicity of~every irreducible subrepresentation~\cite{CG, H71, R71}. Compared to nilmanifolds, solvmanifolds are less systematically studied, even in low dimensions. The first comprehensive texts on the subject were a book by Brezin~\cite{Br} and a long article by Howe~\cite{Ho77}, both published in the 70s. To date, there has been surprisingly little further research except for exponential solvable Lie groups~\cite{FL}. We~think that solvable Lie groups deserve much more interest. They play an important role in pseudo-Riemannian geometry. There are large families of solvable Lie groups that admit a bi-invariant pseudo-Riemannian metric. In~particular, quotients of these groups by discrete subgroups are locally-symmetric spaces.

Indeed, our motivation to write this paper comes from Lorentzian geometry. We~are interested in the spectrum of the wave operator on certain compact Lorentzian locally-symmetric spaces. These spaces arise as quotients of an oscillator group by a discrete subgroup. Oscillator groups are solvable and admit a bi-invariant Lorentzian metric. The wave operator can be considered as minus the quadratic Casimir operator, which is bi-invariant on the group. Hence its spectrum can be easily computed if the decomposition of the right regular representation is known, see Section~\ref{rmwave}. We~will see that there exist lattices such that the spectrum is discrete as well as~lattices for which the quotient has a non-discrete spectrum.

For solvable Lie groups, we~prefer to speak of lattices instead of discrete cocompact subgroups. By~definition, a lattice in a Lie group $G$ is a discrete subgroup $L$ with the property that the quotient space $L\backslash G$ has finite invariant measure. Lattices in a solvable Lie group are uniform, that is, they are the same as discrete cocompact subgroups.

The above mentioned book by Brezin contains an informal discussion of ideas how to attack the problem of decomposing the regular representation in the solvable case. There are also some explicit calculations for a series of four-dimensional examples. These examples are semi-direct products of the Heisenberg group $H$ by the real line in which selected examples of lattices are considered. Inexplicably, the book has only very few citations and Brezin's ideas were not picked up and developed except in~\cite{Hu}. Here we will continue the study of the mentioned four-dimensional groups. We~will concentrate on the case where $\RR$ acts on~$H/Z(H)$ by rotations, which gives exactly the four-dimensional oscillator group. We~will treat this example more systematically than it was done in~\cite{Br}. In~particular, we~consider quotients of this group by arbitrary lattices and not only by ``suitable'' ones as done in~\cite{Br}, see Remark~\ref{RBr} for more information. We~will use a somewhat different approach than proposed by Brezin and are able to completely determine the spectrum for all compact quotients.

Let us introduce the four-dimensional oscillator group and its Lie algebra in more detail. The four-dimensional oscillator algebra $\osc_1$ is a Lie algebra spanned by a basis $X$,~$Y$,~$Z$,~$T$, where $Z$ spans the centre and the remaining basis elements satisfy the relations $[X,Y]=Z$, $[T,X]=Y$, $[T,Y]=-X$. This Lie algebra is strongly related to the one-dimensional quantum harmonic oscillator. Indeed, the Lie algebra spanned by the differential operators $P:={\rm d}/{\rm d}x$, \mbox{$Q:=x$}, $H=\big(P^2+Q^2\big)/2$ and the identity $I$ is isomorphic to $\osc_1$. In~differential geometry, the oscillator algebra $\osc_1$ is also called warped Heisenberg algebra. It~appears in the classification of~(con\-nected) isometry groups of compact Lorentzian manifolds as the Lie algebra of one of such groups~\cite{Z98}.

The oscillator group is the simply-connected Lie group $\Osc_1$ that is associated with the Lie algebra $\osc_1$. It~is a semi-direct product of the 3-dimensional Heisenberg group $H$ by the real line $\RR$, where $\RR$ acts by rotation on the quotient of~$H$ by its centre. In~particular, it~is solvable. As~mentioned above, it~admits a bi-invariant Lorentzian metric. Moreover, this group contains lattices. So, each lattice $L$ in~$\Osc_1$ gives rise to a compact Lorentzian manifold $L\backslash \Osc_1$.

The group $\Osc_1$ is not exponential. However, it~is of type I. Hence, irreducible unitary representations can be determined by the orbit method, which was originally developed by~Kirillov for nilpotent Lie groups and generalised to type I solvable Lie groups by Auslander and Kostant. The polynomials $z$ and $x^2+y^2+2zt$ (in the obvious notation) generate the algebra of invariant polynomials on~$\osc_1^*$. The level sets $\big\{z=c,\, x^2+y^2+2zt=cd\big\}\cong\RR^2$ with fixed $c\not=0$ and $d\in\RR$ are the generic coadjoint orbits. The set $\{ z=0\}$ is stratified into two-dimensional cylindrical orbits $\big\{x^2+y^2=a,\, z=0\big\}$, $a>0$, and orbits consisting of one point $\{x=y=z=0,\, t=d\}$. The~latter orbits correspond to one-dimensional representations~$\cC_d$. Each cylindric \mbox{orbit} corresponds to a set ${\mathcal S}_a^\tau$, $\tau\in S^1$, of infinite-dimensional representations and each generic orbit corresponds to a an infinite-dimensional representation~$\cF_{c,d}$, cf.~\cite{Ki04}.

\looseness=-1 The study of lattices in oscillator groups was initiated by Medina and Revoy~\cite{MR}. Note, how\-ever, that the classification results in~\cite{MR} are not correct due to a wrong description of~the automorphism group of an oscillator group. Lattices of~$\Osc_1$ (as subgroups) were classified up to automorphisms of~$\Osc_1$ by the first author~\cite{F}. Since outer automorphisms can change the spectrum of the quotient, here we need a classification only up to inner automorphisms of~$\Osc_1$. We~derive such a classification in Section~\ref{S4}. First we classify the abstract groups (up to isomorphism) that appear as lattices of~$\Osc_1$. We~will call these groups {\it discrete oscillator groups}. These groups are semi-direct products of a discrete Heisenberg group by $\ZZ$. A discrete Heisenberg group is generated by a central element $\gamma$ and elements $\alpha,\beta$ satisfying the relation~$[\alpha,\beta]=\gamma^r$ for some $r\in\NN_{>0}$. It~is characterised by $r$ and will be denoted by $H_1^r(\ZZ)$. If~$L=H_1^r(\ZZ)\rtimes \ZZ$ is a~discrete oscillator group and if $\delta$ is a generator of the $\ZZ$-factor, then the action of~$\delta$ on the quotient of~$H_1^r(\ZZ)$ by its centre has order $q\in\{1,2,3,4,6\}$. For $q\not=1$, $L$ is almost determined by $r$ and $q$: for fixed $r$, $q$, there are only one or two non-isomorphic discrete oscillator groups. If~$q=1$, then the situation is slightly more complicated, see Section~\ref{S41}. Each discrete oscillator group has different realisations as a lattice in~$\Osc_1$. More exactly, given a discrete oscillator group~$L$ we consider the set $\cM_L$ of lattices of~$\Osc_1$ isomorphic to $L$, where we identify subgroups that differ only by an inner automorphisms of~$\Osc_1$. We~want to parametrise $\cM_L$. First we observe, that every lattice arises from a normalised and unshifted one by outer automorphisms of~$\Osc_1$ that can be viewed as rescaling and shifting. Here, ``normalised'' means that the covolume of~the lattice obtained by projection of~$L\cap H$ to $H/Z(H)\cong \RR^2$ equals one. The explanation of~the notion ``unshifted'' is more involved, see Section~\ref{S42}. Basically, it~means that the lattice is adapted to the fixed splitting of~$\Osc_1$ into the normal subgroup $H$ and a complement $\RR$. Hence, we~get a description of~$\cM_L$ by the subset $\cM_{L,0}\subset \cM_L$ of normalised unshifted lattices modulo inner automorphisms and by two continuous parameters corresponding to rescaling and shifting. For $q\in\{3,4,6\}$, the set $\cM_{L,0}$ is discrete and can be parametrised by the generator $\lambda\in 2\pi/q +2\pi\ZZ$ of the projection of~$L$ to the $\RR$-factor of~$\Osc_1=H\rtimes \RR$. For $q\in\{1,2\}$, $\cM_{L,0}$ is parametrised by a continuous parameter $(\mu,\nu)$ in a fundamental domain~$\cF$ of the $\SL(2,\RR)$-action on the upper halfplane and some further discrete parameters besides $\lambda$. The final classification result for lattices in~$\Osc_1$ up to inner automorphisms is formulated in Theorems~\ref{MMM} and~\ref{th:classUptoInner}. In~particular, we~determine a set of representatives of unshifted normalised lattices, which we will call {\it standard lattices}.

\looseness=-1 For the computation of~the spectrum of~a quotient $L\backslash \Osc_1$, we~will concentrate on normalised and unshifted lattices. The spectrum for arbitrary lattices can easily be derived from these. The first step is to consider only lattices that are generated by a lattice in the Heisenberg group $H$ and an element of~the centre of~$\Osc_1$. Such lattices will be called {\it straight}. For a~straight lattice~$L$, we~derive a Fourier decomposition for functions in~$L^2(L\backslash \Osc_1)$. For the calculations, we~use another model of~the oscillator group, i.e., a group $G$ that is isomorphic to~$\Osc_1$. Lattices in~$G$ will usually be denoted by $\Gamma$. If~$\Gamma\subset G$ corresponds to a straight lattice in~$\Osc_1$ under this isomorphism, then a (continuous) function in~$L^2(\Gamma\backslash G)$ is periodic in three of~the four variables and it is not hard to describe its Fourier decomposition explicitly. From this decomposition, we~obtain the spectrum. Having solved the problem for straight lattices, we~can turn to standard lattices. We~show that each standard lattice $L$ is generated by a straight lattice $L'\subset L$ and an additional element $l\in L$. We~can identify $L^2(L\backslash \Osc_1)$ with the space of~functions in~$L^2(L'\backslash \Osc_1)$ that are invariant under $l$. The isotypic components of~$L^2(L'\backslash \Osc_1)$ are invariant by $l$. Hence, it~remains to determine the invariants of~the action of~$l$ on each of~the isotypic components. This is done in Section~\ref{S8}, where in practice we again use $G$ instead of~$\Osc_1$ and pass from the lattice $L\subset \Osc_1$ to the corresponding lattice $\Gamma\subset G$.
Since, up to an inner automorphism of~$\Osc_1$, every normalised and unshifted lattice $L$ is equal to a standard lattice, we~obtain the spectrum of~$L\backslash \Osc_1$ from our computations for the standard lattices. The representation~$\cC_d$ appears in~$L^2(L\backslash \Osc_1)$ for $d=n/\lambda$, $n\in\ZZ$. The discrete parameter $\tau$ for which the representation~${\mathcal S}_a^\tau$ appears depends only on~$\lambda$ if $q\in\{2,3,4,6\}$. For $q=1$, it~depends also on the above mentioned additional discrete parameters of~$L$. The parameter $a$ depends on~$\lambda$ and, for $q\in\{1,2\}$, in addition on~$(\mu,\nu)\in\cF$. Finally, the representation~$\cF_{c,d}$ appears for $(c,d)=\big(rm,n/(q\lambda)\big)$, $m\in\ZZ_{\not=0}$, $n\in\ZZ$, if $q\in\{2,3,4,6\}$. For $q=1$, the parameter $d$ of~the representation depends besides on~$\lambda$ on one of~the additional discrete parameters of~$L$. For each irreducible unitary representation of~type $\cC_d$ or $\cF_{c,d}$, we~compute the multiplicity with which it appears in~$L^2(L\backslash \Osc_1)$. Furthermore, for each summand of~type ${\mathcal S}_a^\tau$ we describe $a$ and $\tau$ explicitly in terms of~the parameters of~the lattice. However, the exact determination of~the multiplicity of~${\mathcal S}_a^\tau$ for a \emph{given} real number $a\in\RR$ remains a number theoretic problem, see Remark~\ref{fR}.

Finally, given an arbitrary lattice $L'$, then we can determine an (outer) automorphism $F\in\Aut(\Osc_1)$ that transforms $L'$ into a normalised and unshifted lattice $L$. This allows the computation of the spectrum of~$L'$ from that of~$L$ since we can control the action of~$F$ on~the irre\-ducible representations of~$\Osc_1$. We~summarise the results in Section~\ref{S9} at the end of~the paper.

\medskip

\noindent\textbf{Notation} \\[1ex]
$\NN:=\{0,1,2,\dots\};$\\[.5ex]
$\RR_{>0}:=\{x\in\RR\mid x>0\};$ $\NN_{>0}$, $\ZZ_{\not=0}$ are defined analogously; \\[.5ex]
$\ZZ_m:=\ZZ/{m\ZZ};$\\[.5ex]
$\rem_q(n)\in\{0,\dots,q-1\}$ \quad remainder after dividing $n\in\ZZ$ by $q\in\NN_{>0};$\\[.5ex]
$\sgn$ \quad sign function on~$\RR;$\\[.5ex]
$\cSpan$ \quad closed span in a Hilbert space.

\section{The four-dimensional oscillator group}
The 4-dimensional oscillator group is a semi-direct product of the 3-dimensional Heisenberg group $H$ by the real line. Usually, the Heisenberg group $H$ is defined as the set $H=\CC\times\RR$ with multiplication given by
\begin{gather*}
(\xi_1,z_1)\cdot(\xi_2,z_2)=\bigg(\xi_1+\xi_2,z_1+z_2+{\frac12}\omega(\xi_1,\xi_2)\bigg),
\end{gather*}
where $\omega(\xi_1,\xi_2):=\Im\big(\overline{\xi_1}\xi_2\big)$. Hence in explicit terms, the oscillator group is understood as the set $\Osc_1=H\times\RR$ with multiplication defined by
\begin{gather*}
(\xi_1,z_1,t_1)\cdot(\xi_2,z_2,t_2)=\bigg(\xi_1+{\rm e}^{{\rm i}t_1}\xi_2,z_1+z_2+{\frac12}\omega\big(\xi_1,{\rm e}^{{\rm i}t_1}\xi_2\big),t_1+t_2\bigg),
\end{gather*}
where $\omega(\xi_1,\xi_2):=\Im\big(\overline{\xi_1}\xi_2\big)$.

In the following we want to describe the automorphisms of~$\Osc_1$, see~\cite{F} for a proof.
For $\eta\in\CC$, let $F_\eta\colon \Osc_1\rightarrow \Osc_1$ be the conjugation by $(\eta,0,0)$. Furthermore, we~define an automorphism~$F_u$ of~$\Osc_1$ for $u\in\RR$ by
\begin{gather}\label{EFu}
F_u\colon \ (\xi, z, t)\longmapsto (\xi,z+ut,t).
\end{gather}
Finally, consider an $\RR$-linear isomorphism $S$ of~$\CC$ such that $S({\rm i}\xi)=\mu {\rm i}S(\xi)$ for an element $\mu\in\{1,-1\}$ and for all $\xi\in\CC$. Then
$\mu=\sgn(\det S)$ and also
\begin{gather}\label{ES}
F_{S}\colon \ (\xi,z,t)\longmapsto (S\xi,\det(S) z,\mu t)
\end{gather}
is an automorphism of~$\Osc_1$. Each automorphism $F$ of~$\Osc_1$ is of the form
\begin{gather*}
F=F_u\circ F_{\eta}\circ F_{S}
\end{gather*}
for suitable $u\in\RR$, $\eta\in \CC$ and $S\in\GL(2,\RR)$ as considered above. Besides $F_\eta$ also $F_S$ is an inner automorphism if $S\in\SO(2,\RR)$.
\section[Irreducible unitary representations of~$\Osc_1$]
{Irreducible unitary representations of~$\boldsymbol{\Osc_1}$} \label{irrrep}

The irreducible unitary representations of the oscillator group were determined for the first time by Streater~\cite{St67}. Actually, in~\cite{St67}, not the simply-connected group $\Osc_1$ was studied but a~quotient of this group which is a semidirect product of the Heisenberg group and the circle and can be considered as a matrix group.

The group $\Osc_1$ is not exponential, i.e., the exponential map $\exp\colon \osc_1\to\Osc_1$ is not a~dif\-feo\-morphism. More exactly, $\exp$ is not surjective. However, $\Osc_1$
has the following weaker property which ensures that one can describe its irreducible unitary representations.
A locally compact group $G$ is said to be of type I if every factor representation of~$G$ is isotypic. It~is known that the universal cover of any connected solvable algebraic group is of type I, see~\cite[p.~440]{Pu69}. In~particular, $\Osc_1$ is of type I. The irreducible unitary representations of solvable Lie groups of type I can be determined by applying a generalised version of Kirillov's orbit method, see~\cite{AK71,Ki99,Ki04} for a description of this method. An explicit computation of the irreducible unitary representations of~$\osc_1$ can be found in \cite[Section~4.3]{Ki04}, where the oscillator Lie algebra is called diamond Lie algebra.
The difference between the representations of the non-simply-connected group considered in~\cite{St67} and the representations of~$\Osc_1$ lies in an additional parameter $\tau\in\RR/\ZZ$ for the irreducible representations corresponding to cylindrical orbits, see $(ii)$ below.

Let us recall the results of \cite[Section~4.3]{Ki04}. Note that the case $c<0$ in item $(iii)$ below is not discussed in~\cite{Ki04}. We~consider coadjoint orbits of~$\Osc_1$ in~$\osc_1^*=\CC\oplus\RR\oplus\RR$, which are represented by
\benur
\item $(0,0,d)\in\osc_1^*$, $d\in\RR$,
\item $(a,0,0)$, $a\in\RR$, $a>0$,
\item $(0,c,d)$, $c,d\in\RR$, $c\not=0$.
\end{enumerate}
The orbit of~$(0,0,d)$ is a point, the orbit of~$(a,0,0)$ is diffeomorphic to $S^1\times \RR$, and the orbit of~$(0,c,d)$ for $c\not=0$ is diffeomorphic to $\RR^2$. We~will describe the representations corresponding to~$(ii)$ and~$(iii)$ only on the Lie algebra level. Let~$X$, $Y$, $Z$, $T$ be the standard basis of~$\osc_1=\CC\oplus\RR\oplus \RR$,
The non-vanishing brackets of basis elements are $[X,Y]=Z$, $[T,X]=Y$ and $[T,Y]=-X$.
\benur
\item $\cC_d:=\big(\sigma_{(0,0,d)},\CC\big)$, $\sigma_{(0,0,d)} (\xi,z,t)={\rm e}^{2\pi {\rm i} d t}$.
\item ${\mathcal S}_a^\tau:=\big(\sigma:=\sigma_{(a,0,0)}^\tau, L^2\big(S^1\big)\big)$, $\tau\in \RR/\ZZ\cong [0,1)$,
\begin{gather*}
\sigma_*(Z)(\ph)=0,
\\
\sigma_*(X+{\rm i}Y)(\ph)=2\pi {\rm i} a {\rm e}^{-{\rm i}t} \ph,
\\
\sigma_*(X-{\rm i}Y)(\ph)=2\pi {\rm i} a {\rm e}^{{\rm i}t} \ph,
\\
\sigma_*(T)(\ph)=\ph'+{\rm i}\tau \ph
\end{gather*}
for $\ph=\ph(t)\in C^{\infty}\big(S^1\big)\subset L^2\big(S^1\big)$. The orthonormal system $\phi_n:={\rm e}^{{\rm i}nt}$, $n\in\ZZ$ satisfies
\begin{gather*}
\sigma_*(X+{\rm i}Y)(\phi_n) =2\pi {\rm i} a\, \phi_{n-1},
\\
\sigma_*(X-{\rm i}Y)(\phi_n) =2\pi {\rm i} a\, \phi_{n+1},
\\
 \sigma_*(T)(\phi_n)={\rm i}(n+\tau)\,\phi_n.
 \end{gather*}
\item For $c>0$, we~consider the Hilbert space
\begin{gather*}
\cF_c(\CC):=\bigg\{\ph\colon \CC\to\CC \text{ holomorphic }\bigg| \int_{\mathbb C} |\ph(\xi)|^2{\rm e}^{-\pi c |\xi|^2} c\, {\rm d}\xi <\infty\bigg\}
\end{gather*}
with scalar product
\begin{gather}
\la \ph_1,\ph_2\ra=\int _{{\mathbb C}} \ph_1(\xi)\overline {\ph_2(\xi)} {\rm e}^{-\pi c |\xi|^2} c\, {\rm d}\xi\label{ipF}
\end{gather}
for $\ph_1,\ph_2\in \cF_c(\CC)$.
Then the representation~$\sigma:=\sigma_{(0,c,d)}$ on~$\cF_c(\CC)$ is given by
\begin{gather*}
\sigma_*(Z)(\ph)=2\pi {\rm i}c \ph,
\\
\sigma_*(X+{\rm i}Y)(\ph)=2\pi c\xi \ph,
\\
\sigma_*(X-{\rm i}Y)(\ph)=-2\partial \ph ,
\\
\sigma_*(T)(\ph)=2\pi {\rm i} d \ph -{\rm i}\xi \partial\ph.
\end{gather*}
The functions $\psi_n:=\frac{(\sqrt{\pi c} \xi)^n}{\sqrt{n!}}$, $n\ge0$, constitute a complete orthonormal system of~$\cF_c(\CC)$ and we have
\begin{gather*}
\sigma_*(Z)(\psi_n)= 2\pi {\rm i}c \psi_n,\qquad
\sigma_*(T)(\psi_n)= (2\pi d-n)\,{\rm i} \psi_n
\end{gather*}
and, for $A_+:=\sigma_*(X+{\rm i}Y)$ and $A_-:=\sigma_*(X-{\rm i}Y)$,
\begin{gather}
A_+(\psi_n)=2\sqrt{\pi c(n+1)} \psi_{n+1},\qquad n\ge0, \label{A+}\\
A_-(\psi_0)=0,\qquad A_-(\psi_n)=-2\sqrt{\pi c n} \psi_{n-1}, \qquad n\ge1.\label{A-}
\end{gather}

Furthermore, for $c>0$, we~consider
\begin{gather*}
\cF_{-c}(\CC):=\bigg\{\ph\colon \CC\to\CC \text{ anti-holomorphic }\bigg|\, \int_{\mathbb C} |\ph(\xi)|^2{\rm e}^{-\pi c |\xi|^2} c\, {\rm d}\xi <\infty\bigg\}
\end{gather*}
with scalar product given by (\ref{ipF}) now for $\ph_1,\ph_2\in \cF_{-c}(\CC)$.
The representation~$\sigma:=\sigma_{(0,-c,d)}$ on~$\cF_{-c}(\CC)$ is given by
\begin{gather*}
\sigma_*(Z)(\ph)=-2\pi {\rm i}c \ph,
\\[.5ex]
\sigma_*(X+{\rm i}Y)(\ph)=-2\bar \partial \ph,
\\[.5ex]
\sigma_*(X-{\rm i}Y)(\ph)= 2\pi c\bar\xi \ph,
\\[.5ex]
\sigma_*(T)(\ph)=2\pi {\rm i} d \ph + {\rm i}\bar\xi \bar\partial\ph.
\end{gather*}
Here, the functions \[\psi_n:=\frac{(\sqrt{\pi c} \bar\xi)^n}{\sqrt{n!}}, \qquad n\ge0, \] constitute a complete orthonormal system and we have
\begin{gather*}
\sigma_*(Z)(\psi_n)= -2\pi {\rm i}c \psi_n,\qquad
\sigma_*(T)(\psi_n)= (2\pi d+n)\,{\rm i} \psi_n.
\end{gather*}
Now, equations (\ref{A+}) and (\ref{A-}) hold for
\[A_+:=\sigma_*(X-{\rm i}Y)$ and $A_-:=\sigma_*(X+{\rm i}Y).\]
We will use the notation~\[\cF_{c,d}:=\big(\sigma_{(0,c,d)},\cF_c(\CC)\big),\qquad \cF_{-c,d}:=\big(\sigma_{(0,-c,d)},\cF_{-c}(\CC)\big).\]
\end{enumerate}

Every irreducible unitary representation of the oscillator group is isomorphic to one of these representations.

Let us consider the Casimir operator for these representations. Let~$X$, $Y$, $Z$, $T$ be the standard basis of~$\osc_1$ as before. Let~$\langle\cdot,\cdot\rangle$ denote the $\operatorname{ad}$-invariant Lorentzian metric on~$\osc_1$ given by
$\langle X,X\rangle=\langle Y,Y\rangle=\langle Z,T\rangle=1$.
The remaining scalar products of basis elements vanish.
The~Casi\-mir operator of a representation~$\sigma$ with respect to $\langle\cdot,\cdot\rangle$ equals
\begin{gather*}
\Delta_\sigma=(\sigma_*(X))^2+(\sigma_*(Y))^2+2(\sigma_*(Z))(\sigma_*(T)).
\end{gather*}
A straight forward computation yields $\Delta_{\cC_d}=0$, $\Delta_{{\mathcal S}_a^\tau}=-4\pi a^2$,
$\Delta_{\cF_{c,d}}=-2\pi c(4\pi d+1)$ for $c>0$
 and $\Delta_{\cF_{c,d}}=-2\pi c(4\pi d-1)$ for $c<0$.

Let $S$ be an $\RR$-linear isomorphism of~$\CC$ such that $S({\rm i}\xi)=\mu {\rm i}S(\xi)$ for $\mu\in\{1,-1\}$ and for all~$\xi\in\CC$. Recall that this implies
$\mu=\sgn(\det S)$. Let~$F_S$ be defined as in~\eqref{ES}.

Let $F$ be an automorphism of~$\Osc_1$. For a representation~$(\sigma,V)$ of~$\Osc_1$ we define a representation~$F^*(\sigma,V):=(\sigma\circ F,V)$. One computes
\begin{gather}
F_u^*\cC_d =\cC_d,\qquad
F_u^*{\mathcal S}^\tau_a ={\mathcal S}^\tau_a,\qquad
F_u^*{\mathcal F}_{c,d} ={\mathcal F}_{c,d+uc}, \nonumber
\\
F_S^*\cC_d =\cC_{\mu d},\qquad
F_S^*{\mathcal S}^\tau_a ={\mathcal S}^{\mu\tau}_{|\hspace{-1pt}\det S|^{1/2}a},\qquad
F_S^*{\mathcal F}_{c,d} ={\mathcal F}_{\hspace{-1pt}\det(S) c,\mu d}.
\label{Fiso}
\end{gather}

\section[Lattices in~$\Osc_1$]{Lattices in~$\boldsymbol{\Osc_1}$}\label{S4}
Oscillator groups can be defined in every even dimension as a generalisation of the classical four-dimensional one. The study of lattices in such groups was initiated by Medina and Revoy~\cite{MR}. However, their classification of lattices in~$\Osc_1$ is not correct, see Remark~\ref{RMR}. Finally, lattices of~$\Osc_1$ were classified up to automorphisms of~$\Osc_1$ by the first author~\cite{F}. Since outer automorphisms can change the spectrum of the quotient, here we need a classification up to {\it inner} automorphisms of~$\Osc_1$. Such a classification is achieved in this section. In~order to formulate it, we~introduce the new concept of normalised and unshifted lattices. This concept will also play an important role in Section~\ref{S8}, where we compute the spectrum of quotients by standard lattices.

\subsection{Classification of discrete oscillator groups}
\label{S41}
Consider the discrete Heisenberg group
\begin{gather}\label{dHeis}
H_1^r(\ZZ):=\big\la \alpha, \beta, \gamma \mid \alpha \beta \alpha^{-1} \beta^{-1} =\gamma^r,\, \alpha \gamma=\gamma \alpha,\, \beta\gamma=\gamma\beta\big\ra,
\end{gather}
where $r\in\NN_{>0}$. It~is well known that $H_1^r(\ZZ)$ and $H_1^{r'}(\ZZ)$ are isomorphic if and only if $r=r'$.
Let~$S$ be a homomorphism from the infinite cyclic group $\ZZ \delta$ generated by $\delta$ into the automorphism group of~$H_1^r(\ZZ)$.
A {\it discrete oscillator group} is a semidirect product $H_1^r(\ZZ)\rtimes_S \ZZ \delta$, where the map $\overline{S(\delta)}$ induced by $S(\delta)$ on~$H_1^r(\ZZ)/Z(H_1^r(\ZZ))$ is conjugate to a rotation in~$\GL(2,\RR)$. In~the following we just write $S$ instead of~$S(\delta)$ and $\bar S$ instead of~$\overline{ S(\delta)}$. Using the projections
of~$\alpha$ and $\beta$ to $H_1^r(\ZZ)/Z(H_1^r(\ZZ))$ as a basis, we~identify this quotient with $\ZZ^2$ and $\bar S$ with an element of~$\SL(2,\ZZ)$. Since the map $\bar S\in \SL(2,\ZZ)$ is conjugate over $\GL(2,\RR)$ to a rotation, it~is of finite order. This implies
\begin{Lemma}\label{lmgen}
Given a discrete oscillator group $H_1^r(\ZZ)\rtimes_S \ZZ \delta$, we~find $($new$)$ generators $\alpha$,~$\beta$,~$\gamma$ of~$H_1^r(\ZZ)$ such that the relations in~\eqref{dHeis} are satisfied and $\bar S$ equals one of the matrices
\begin{gather}\label{Sq}
S_1=I_2,\quad\
S_2=-I_2,\quad\
S_3= \begin{pmatrix} 0&-1\\1&-1 \end{pmatrix}\!,\quad\
S_4=\begin{pmatrix} 0&-1\\1&0 \end{pmatrix}\!,\quad\
S_6= \begin{pmatrix} 1&-1\\1&0 \end{pmatrix}\!.
\end{gather}
In particular, each discrete oscillator group is isomorphic to a group $H_1^r(\ZZ)\rtimes_S \ZZ \delta$ for which $\bar S$ is one of the maps in~\eqref{Sq}.
\end{Lemma}

\begin{proof}
It is well known that each matrix of finite order in~$\SL(2,\ZZ)$ is conjugate over $\GL(2,\ZZ)$ to one of the matrices $S_1$, $S_2$, $S_3$, $S_4$, $S_6$, see, e.g.,~\cite{N}. The map $T\in\GL(2,\ZZ)$ that is used for this conjugation can be extended to an isomorphism of discrete oscillator groups which preserves~$\delta$.
\end{proof}
\begin{Lemma}
Let $H_1^r(\ZZ)$ be a discrete Heisenberg group given as in~\eqref{dHeis} and consider a discrete oscillator group $L:=H_1^r(\ZZ)\rtimes_S \ZZ \delta$ with $S(\alpha)=\bar S(\alpha)\gamma^k$, $S(\beta)=\bar S(\beta)\gamma^l$ and $S(\gamma)=\gamma$.

If $\ord(\bar S)=:q\in\{2,3,4,6\}$, then the centre $Z(L)$ of~$L$ equals $\la \gamma, \delta^q\ra$. If~$\bar S=I_2$, then $Z(L)=\la \gamma, \alpha^a\beta^b\delta^d\ra$,
where $a=-l/r_0$, $b=k/r_0$, $d=r/r_0$ for $r_0:=\gcd(k,l,r)$. In~this case $[L,L]$ is generated by $\gamma^{r_0}$.
\end{Lemma}

\begin{proof}
First take $q\in\{2,3,4,6\}$. We~may assume that $\bar S$ is one of the maps $S_2$, $S_3$, $S_4$, $S_6$ and it is easy to check that not only $\bar S^q=I_2$ but also $S^q=\Id$, which proves the claim.

Now consider $\bar S=I_2$ and suppose that $l:=\alpha^a\beta^b\delta^d$, $d\not=0$, is in~$Z(L)$. Then
\begin{gather*}
\alpha=l\alpha l^{-1}=\alpha \gamma^{-br+dk}, \qquad
\beta=l\beta l^{-1}=\beta\gamma^{ar+dl},
\end{gather*}
thus $br=dk=n_1 \lcm(k,r)$ and $dl=-ar=n_2 \lcm(l,r)$ for suitable $n_1,n_2\in\ZZ$. This implies
\begin{gather*}
d= n_1\frac{\lcm(k,r)}k=n_2\frac{\lcm(l,r)}l
 = n_1 \frac r{\gcd(k,r)}= n_2\frac r{\gcd(l,r)}.
\end{gather*}
Hence $d$ is a multiple of lcm$\big( \frac r{\gcd(k,r)},\frac r{\gcd(l,r)}\big)=r/{r_0}$, from which the assertion easily follows.
\end{proof}

For each $r\in \NN_{>0}$, we~define discrete oscillator groups $L^i_r$, $i\in\{1,2,3,4,6\}$ and $L^{j,+}_r$, $j\in\{2,3,4\}$ by the data in the following table:
\begin{center}
{\renewcommand{\arraystretch}{1.5}
\begin{tabular}{l|c|c} \hline	
\multicolumn{1}{c|}{Group}		& $S(\alpha)$ 			& $S(\beta)$ 		\\
\hline
\ $L^1_r$ & $\alpha$ & $\beta$ \\ \hline
\ $L^2_r$ & $\alpha^{-1}$ & $\beta^{-1}$ \\ \hline
\ $L^{2,+}_r$,\quad $r$ even & $\alpha^{-1}\gamma$ & $\beta^{-1}\gamma^{-1}$ \\ \hline
\ $L^3_r$ & $\beta\gamma^r$ & $\beta^{-1}\alpha^{-1}$ \\ \hline
\ $L^{3,+}_r,\quad r\equiv 0\ {\rm mod}\ 3$ & $\beta\gamma^{r-1}$ & $\beta^{-1}\alpha^{-1}\gamma$ \\ \hline
\ $L^4_r$ & $\beta$ & $\alpha^{-1}$ \\ \hline
\ $L^{4,+}_r$,\quad $r$ even & $\beta\gamma$ & $\alpha^{-1}$ \\ \hline
\ $L^6_r$ & $\alpha\beta$ & $\alpha^{-1}$ \\ \hline
\end{tabular}}
\end{center}
Note that, compared to~\cite{F}, we~use a slightly different system of representatives to assure that~$L^{2}_r$,~$L^{2,+}_r$ and~$L^{3}_r$ are subgroups of~$L^{4}_r$, $L^{4,+}_r$ and $L^{6}_r$, respectively.

Clearly, if $r$ is odd, then $L^{2,+}_r\cong L^{2}_r$, $L^{4,+}_r\cong L^{4}_r$ and if $r\not\equiv 0$ mod $3$, then $L^{3,+}_r\cong L^{3}_r$.

We recall the classification of discrete oscillator groups from~\cite{F}. We~will present also the main ideas of a direct proof since the proof in~\cite{F} already uses the classification of lattices in~$\Osc_1$.

\begin{Proposition}[\cite{F}]\label{TF}
The groups $L^i_r$ and $L^{j,+}_r$ for $i\in\{1,2,3,4,6\}$, $j\in\{2,3,4\}$ and $r\in\NN_{>0}$ are pairwise non-isomorphic discrete oscillator groups.
Conversely, each discrete oscillator group is isomorphic to one of these groups.

More exactly, let $L:=H_1^r(\ZZ)\rtimes_S \ZZ \delta$ be a discrete oscillator group with $S(\alpha)=\bar S(\alpha)\gamma^k$, $S(\beta)=\bar S(\beta)\gamma^l$, where we already assume that $\bar S$ is one of the maps listed in~\eqref{Sq}$:$
\benur
\item if $\bar S=I_2$, then $L$ is isomorphic to $L^1_{r_0}$ for $r_0=\gcd(r,k,l)$,
\item for $\bar S=-I_2$, we~have
$L\cong \begin{cases}
L^2_r, & \text{if $r$ is odd or if $k$ and $l$ are even},
\\[0.5ex]
L^{2,+}_r, & \text{else},
\end{cases}$
\item for $\bar S=S_3$, we~have
$L\cong \begin{cases} L^3_r, & \text{if $3\nmid r$ or if $k\equiv l\ {\rm mod}\ 3$},
\\[0.5ex]
L^{3,+}_r, & \text{else},
\end{cases}$
\item for $\bar S=S_4$, we~have
$L\cong \begin{cases} L^4_r, & \text{if $r$ is odd or if $k\equiv l\ {\rm mod}\ 2$},
\\[0.5ex]
L^{4,+}_r, & \text{else},
\end{cases}$
\item if $\bar S=S_6$, then $L\cong L^6_{r}$.
\end{enumerate}
\end{Proposition}

\begin{proof}
Assume first that $\bar S=I_2$, i.e., $S(\alpha)=\alpha\gamma^k$ and $S(\beta)=\beta\gamma^l$. We~put again~$a=-l/r_0$, $b=k/r_0$, $d=r/r_0$. Then $\gcd(a,b,d)=1$. Hence we can find elements $(m_1,m_2,m)$, $(n_1,n_2,n)$ $\in\ZZ^3$ such that
\begin{gather}\label{Edet}
\det \begin{pmatrix} m_1&m_2&m\\ n_1&n_2&n\\ a&b&d \end{pmatrix} =1.
\end{gather}
We define a map from the set of generators $\{\alpha,\beta,\gamma,\delta\}$ of~$L_{r_0}^1$ into $L$ by
\begin{gather*}
\alpha\longmapsto \alpha^{m_1}\beta^{m_2}\delta^m,\qquad
\beta\longmapsto \alpha^{n_1}\beta^{n_2}\delta^n,\qquad
\gamma\longmapsto\gamma,\qquad
\delta\longmapsto \alpha^a\beta^b\delta^d \in Z(L).
\end{gather*}
This map extends to a homomorphism from $L_{r_0}^1$ to $L$ since
$[\alpha^{m_1}\beta^{m_2}\delta^m,\alpha^{n_1}\beta^{n_2}\delta^n]=\gamma^N$ with
\begin{gather*}
N=\det\begin{pmatrix} m_1&m_2&m\\ n_1&n_2&n\\ -l&k&r \end{pmatrix} =r_0
\end{gather*}
by~\eqref{Edet}. Obviously, this homomorphism is bijective.

Now suppose $\bar S=S_2$. Let~$(k,l)$ and $(k',l')$ be pairs of integers and let $S$ and $S'$ be automorphisms of~$H^1_r(\ZZ)$ defined by $S(\alpha)=S_2(\alpha)\gamma^k$, $S(\beta)=S_2(\beta)\gamma^l$ and $S'(\alpha)=S_2(\alpha)\gamma^{k'}$, $S(\beta)=S_2(\beta)\gamma^{l'}$, respectively.
We define maps $T_1,\dots,T_4$ from the set of generators of $H^1_r(\ZZ)\rtimes_S \delta \ZZ$ into $H^1_r(\ZZ)\rtimes_{S'} \delta \ZZ$ by
\begin{gather*}
T_1\colon\quad T_1(\alpha)=\alpha \gamma,\qquad T_1(\beta)=\beta,\qquad \phantom{{}^{-1}}T_1(\gamma)=\gamma,\qquad T_1(\delta)=\delta,
\\
T_2\colon\quad T_2(\alpha)=\beta,\qquad \phantom{\alpha}T_2(\beta)=\alpha^{-1},\qquad T_2(\gamma)=\gamma,\qquad T_2(\delta)=\delta,
\\
T_3\colon\quad T_3(\alpha)=\alpha\beta,\qquad T_3(\beta)=\beta,\qquad \phantom{{}^{-1}}T_3(\gamma)=\gamma,\qquad T_3(\delta)=\delta\beta,
\\
T_4\colon\quad T_4(\alpha)=\alpha,\qquad \phantom{\alpha}T_4(\beta)=\beta,\qquad \phantom{{}^{-1}}T_4(\gamma)=\gamma,\qquad T_4(\delta)=\delta\beta.
\end{gather*}
Then $T_1$ extends to an isomorphism if and only if $(k',l')=(k-2,l)$, $T_2$ extends to an isomorphism if and only if $(k',l')=(-l,k)$, $T_3$ extends if and only if $(k',l')=(k-l,l)$ and $T_4$ if and only if $(k',l')=(k+r,l)$. By~composing these isomorphisms we see that $L$ is isomorphic to $L^2_r$ or~to~$L^{2,+}_r$ under the conditions indicated in $(ii)$. It~remains to show that $L_r^2$ and $L_r^{2,+}$ are not isomorphic for a fixed even $r$. This follows from the fact that the abelianisation~$L/[L,L]$ of~$L$ is isomorphic to $\ZZ_2\times\ZZ_{2r}\times\ZZ$ if $L=L^{2,+}_{r}$ and to $\ZZ_2\times\ZZ_2\times\ZZ_r\times\ZZ$ if $L=L^2_{r}$.

The assertions $(iii)$--$(v)$ are proven in an analogous way.
\end{proof}

\begin{Corollary}\label{co:iso}
Let $L:=H_1^r(\ZZ)\rtimes_S \ZZ \delta$ be a discrete oscillator group.
Consider the abelianisation~$L^{ab}=L/[L,L]$.
The lattice $L$ is isomorphic to $L^1_{r_0}$ if and only if $L^{ab}\cong\ZZ_{r_0}\times \ZZ^3$ and it is isomorphic to $L^6_r$ if and only if $L^{ab}\cong \ZZ_r\times\ZZ$.

Furthermore, $\ord\big(\bar S\big)=2$ holds if and only if $L^{ab}\cong\ZZ_2\times\ZZ_2\times\ZZ_r\times\ZZ$ or $L^{ab}\cong\ZZ_2\times\ZZ_{2r}\times\ZZ$. If~$r$ is odd, the latter two groups are isomorphic and $L\cong L^2_r$. If~$r$ is even, then
\begin{gather*}
L\cong \begin{cases}
L^{2,+}_{r}, & \text{if}\quad L^{ab}\cong\ZZ_2\times\ZZ_{2r}\times\ZZ,
\\
L^2_{r}, & \text{if}\quad L^{ab}\cong\ZZ_2\times\ZZ_2\times\ZZ_r\times\ZZ.
\end{cases}
\end{gather*}
We have $\ord\big(\bar S\big)=3$ if and only if
$L^{ab}\cong\ZZ_{3r}\times\ZZ$ or
$L^{ab}\cong\ZZ_3\times\ZZ_r\times\ZZ$. If~$3\nmid r$, then these groups are isomorphic and $L\cong L^3_r$. If~$3\mid r$, then
\begin{gather*}
L\cong\begin{cases}
L^{3,+}_{r}, & \text{if}\quad L^{ab}\cong\ZZ_{3r}\times\ZZ,
\\
L^3_{r}, & \text{if}\quad L^{ab}\cong\ZZ_3\times\ZZ_r\times\ZZ.
\end{cases}
\end{gather*}
Finally, $\ord\big(\bar S\big)=4$ holds if and only if
$L^{ab}\cong\ZZ_{2r}\times\ZZ$ or $L^{ab}\cong\ZZ_2\times\ZZ_r\times\ZZ$. As~above, these groups are isomorphic for odd $r$ and
$L\cong L^4_r$. If~$r$ is even, then
\begin{gather*}
L\cong \begin{cases}
L^{4,+}_{r}, & \text{if}\quad L^{ab}\cong\ZZ_{2r}\times\ZZ,
\\
L^4_{r}, & \text{if}\quad L^{ab}\cong\ZZ_2\times\ZZ_r\times\ZZ.
\end{cases}
\end{gather*}
\end{Corollary}

\subsection[Classification of lattices up to inner automorphisms of~$\Osc_1$]
{Classification of lattices up to inner automorphisms of~$\boldsymbol{\Osc_1}$}
\label{S42}
Let $L$ be a lattice in~$\Osc_1$. The intersection of~$L$ with the Heisenberg group is a lattice in the Heisenberg group (see~\cite[Corollary~3.5]{Ra72}). In~particular, it~is a discrete Heisenberg group $H_1^r(\ZZ)$ for some unique $r\in\NN_{>0}$. Hence $L$ considered as an abstract group is isomorphic to exactly one of the discrete oscillator groups listed in Proposition~\ref{TF}.
The projection of~$L\subset\Osc_1=H\times\RR$ to the second factor is generated by a unique real number $\lambda$, where
\begin{gather*}
\lambda\in \pi\cdot \NN_{>0}\qquad \text{or}\qquad \lambda= \lambda_0+2\pi\ZZ,\qquad \lambda_0\in \{\pi/3, \pi/2, 2\pi/3\}.
\end{gather*}
We define
\begin{gather*}
\ord(\lambda):=\{\text{smallest positive integer $q$ such that } q\lambda\in 2\pi\ZZ\}
\in\{1,2,3,4,6\}.
\end{gather*}
The order of~$\lambda$ coincides with the order of~$\bar S$ in the abstract discrete oscillator group~$L$. In~particular, $L$ is isomorphic to $L^q_r$ or $L^{q,+}_r$ as an abstract group. We~will say that $L$ is of type $L^q_r$ or~$L^{q,+}_r$, respectively. Furthermore, we~define
\begin{gather}\label{Drlambda}
r(L):=r,\qquad\lambda(L):=\lambda.
\end{gather}

For fixed $r$ and $q$, there are different lattices in~$\Osc_1$ (up to automorphisms of~$\Osc_1$) of type~$L^q_r$ and $L^{q,+}_r$. The aim of this subsection is to give a classification of all lattices in~$\Osc_1$ up to inner automorphisms of~$\Osc_1$. The problem will be reduced to the classification of lattices with the additional property to be normalised and unshifted, which we will explain in the following.
\begin{Definition} Let $L\subset \Osc_1$ be a lattice in~$\Osc_1$. Let~$\bar L_0$ denote the projection of~$L_0:=L\cap H$ to $H/Z(H)\cong \RR^2$. The lattice $L$ is called normalised if $\bar L_0$ is a normalised lattice in~$\RR^2$, i.e., if~$\bar L_0\subset \RR^2$ has covolume one with respect to the standard metric of~$\RR^2$.
\end{Definition}
\begin{Remark}\quad
\benur
\item A lattice $L$ in~$\Osc_1$ is normalised if and only if the commutator subgroup $[L\cap H,L\cap H]$ is generated by $(0,1,0)\in\Osc_1$.
\item
If $L$ is a normalised lattice in~$\Osc_1$, so is $F_S(L)$, $F_\eta(L)$ and $F_u(L)$ for all $S\in\GL(2,\RR)$ with $\det S=\pm 1$, for all $\eta\in \RR^2$ and $u\in\RR$.
\end{enumerate}
\end{Remark}

\begin{Definition}\label{def:adapted}
Generators $\alpha,\beta,\gamma,\delta$ of~$L$ will be called adapted if
the projection of~$\delta\in \Osc_1=H\times \RR$ to the second factor equals $\lambda:=\lambda(L)$,
if $\alpha$, $\beta$, $\gamma$ of~$L\cap H$ satisfy the condition in Lemma~\ref{lmgen} and
if the projection of~$\alpha$ and $\beta$ to $H/Z(H)\cong\RR^2$ is a positively oriented basis.
\end{Definition}

Every lattice has adapted generators. Indeed, first we choose a suitable~$\delta$. Now we choose $\alpha$, $\beta$, $\gamma$ according to Lemma~\ref{lmgen}. If~$\ord(\lambda)\in\{3,4,6\}$, then the basis $\bar \alpha$, $\bar\beta$ of~$\RR^2$ is positively oriented. If~$\ord(\lambda)\in\{1,2\}$ and $\bar \alpha$, $\bar \beta$ is negatively oriented, then we replace $\beta$ by $\beta^{-1}$ and $\gamma$ by~$\gamma^{-1}$ to obtain adapted generators.

Clearly, if $L$ is normalised, then $\gamma=(0,1/r(L),0)$.
\begin{Definition}\label{de:inv}
Let $L$ be a normalised lattice. We~will associate with $L$ numbers $s_0\in\NN_{>0}$ and $\bs_L \in \RR/ \frac1{s_0 r}\ZZ $, where $r=r(L)$. We~choose adapted generators
\begin{gather}\label{eq:generators}
\alpha=(\bar\alpha,z_\alpha,0),\qquad
\beta=(\bar\beta,z_\beta,0),\qquad
\gamma=(0,1/r,0),\qquad
\delta=(x_\delta\bar\alpha+y_\delta\bar\beta,z_\delta,\lambda)
\end{gather}
and define $a,b\in \frac12 \ZZ$ and $v,w\in\RR$ by
\begin{gather}\label{abvw}
(a,b)=\begin{cases} \big(1,\frac12\big),& 2\nmid r, \quad q=3,
\\
\big(\frac12,0\big),& 2\nmid r,\quad q=6,
\\
(0,0),& \text{else},
\end{cases}
\qquad
{v\choose w}=r\begin{pmatrix}x_\delta\\y_\delta\end{pmatrix}
-r(S_q-I_2)\begin{pmatrix}-z_\beta\\z_\alpha\end{pmatrix}\!.
\end{gather}
\begin{enumerate}\itemsep=0pt
\item If $q=1$ and $L$ is of type $L^1_{r_0}$, then we set $s_0=r/r_0=r/\gcd(v,w,r)$ and
\begin{gather*}
\bs_L=z_\delta-\frac{1}{2}vw-\frac{1}{r}\,\omega
\left(\!\begin{pmatrix}v\\w\end{pmatrix}\!,\begin{pmatrix}-z_\beta\\z_\alpha\end{pmatrix}\!\right)
\ {\rm mod}\ \frac{1}{s_0r}\ZZ.
\end{gather*}
\item For $q\in\{2,3,4,6\}$ denote by $q_0\in\{2,3\}$ the smallest prime factor of~$q$ and by $\tilde r=\rem_{q_0}(r)\in\{0,1,\dots,q_0-1\}$ the remainder after dividing $r$ by $q_0$.
\begin{enumerate}\itemsep=0pt
\item[$(a)$] If $L$ is of type $L^{q,+}_r$ then we set
\begin{gather*}
s_0=\begin{cases}
1,& \text{if}\quad r\in 2+4\ZZ \quad\text{and}\quad q=4,
\\
q_0,& \text{else},
\end{cases}
\end{gather*}
and
\begin{gather*}
\bs_L=z_\delta -\frac{1}{2r}\,\omega\left(\!\begin{pmatrix}v\\w\end{pmatrix}\!, (S_q+I_2)\begin{pmatrix}-z_\beta\\z_\alpha\end{pmatrix}\!\right)
-\frac{1}{2}\,\omega\left(S_q\begin{pmatrix}-z_\beta\\z_\alpha\end{pmatrix}\!, \begin{pmatrix}-z_\beta\\z_\alpha\end{pmatrix}\!\right)
\\ \hphantom{\bs_L=}
{}-z_0\ {\rm mod}\ \frac{1}{s_0r}\ZZ,
\end{gather*}
where $z_0$ is the real number uniquely defined by
\begin{gather*}
\bigg(\frac{v}{r}\bar\alpha+\frac{w}{r}\bar\beta,z_0,\lambda\bigg)^q=(0,0,q\lambda).
\end{gather*}
\item[$(b)$] If $L$ is of type $L^q_r$ then we set
\begin{gather*}
s_0=\begin{cases}
q_0,& \text{if}\quad r\in 2+4\ZZ\quad\text{and}\quad q=4,
\\
1,& \text{else},
\end{cases}
\end{gather*}
and
\begin{gather*}
\bs_L=z_\delta
 -\frac{1}{2r}\,\omega\left(\!\begin{pmatrix}v\\w\end{pmatrix}\!, (S_q+I_2)\begin{pmatrix}-z_\beta\\z_\alpha\end{pmatrix}\!\right)
-\frac{1}{2}\,\omega\left(S_q\begin{pmatrix}-z_\beta\\z_\alpha\end{pmatrix}\!, \begin{pmatrix}-z_\beta\\z_\alpha\end{pmatrix}\!\right)-z_0
\\[0.5ex] \hphantom{\bs_L=}
{}-\frac{1}{2}(v-a)(w-b)\tilde r^2
-\frac{\tilde r}{2r}\omega\left(\!\begin{pmatrix}v\\w\end{pmatrix}\!, \begin{pmatrix}a\\b\end{pmatrix}\!\right)\ {\rm mod}\ \frac{1}{s_0r}\ZZ,
\end{gather*}
where $z_0\in\RR$ is uniquely defined by
\begin{gather*}
\bigg(\bigg(\frac vr-\tilde rv+\tilde ra\bigg)\bar\alpha+\bigg(\frac{w}{r}-\tilde rw+\tilde rb\bigg)\bar\beta,z_0,\lambda\bigg)^q=(0,0,q\lambda).
\end{gather*}
\end{enumerate}
\end{enumerate}
\end{Definition}

\begin{Remark}\label{rm4}\quad
\benur
\item The number $\bs_L$ is well-defined, see Proposition~\ref{scorrect}.
 \item 
 Note that $v\in\ZZ+a$, $w\in\ZZ+b$. Indeed, if $S_q=(s_{ij})_{i,j=1,2}$, then we can define inte\-gers~$k$,~$l$ by $\delta \alpha\delta^{-1}=\alpha^{s_{11}}\beta^{s_{21}}\gamma^k$ and $\delta \beta\delta^{-1}=\alpha^{s_{12}}\beta^{s_{22}}\gamma^l$. By~\eqref{eq:generators}, we~have
\begin{gather}\label{klvw}
 {v\choose w}=\frac r2 S_q \begin{pmatrix}\phantom{-}s_{12}\cdot s_{22}\\[0.5ex] -s_{11}\cdot s_{21}\end{pmatrix} -S_q{-l\choose k},
\end{gather}
which proves the claim.
\item
One easily computes that $z_0=0$ for $q=2$ and
\begin{gather*}
z_0= -\frac{c}{2qr^2}\, \omega\left(\!{v \choose w},S_q{v \choose w}\!\right)
\end{gather*}
for $L=L^{q,+}_r$, $q\in\{3,4,6\}$, and
\begin{gather*}
z_0=-\frac{c}{2qr^2}\, \omega\left(\!{v-r\tilde rv+r\tilde ra\choose w-r\tilde rw+r\tilde rb},S_q{v-r\tilde rv+r\tilde ra\choose w-r\tilde rw+r\tilde rb}\!\right)
\end{gather*}
for $L=L^{q}_r$, $q\in\{3,4,6\}$, where $c=1$ if $q=3$, $c=2$ if $q=4$ and $c=6$ if $q=6$.
\item If $L$ is a normalised lattice, then $\bs_{F_u(L)}=\bs_L+u\cdot\lambda(L)\ {\rm mod}\ \frac1{s_0r}\ZZ$, where $s_0$ is defined as in Definition~\ref{de:inv}.
\end{enumerate}
\end{Remark}
\begin{Definition} A normalised lattice $L\subset\Osc_1$ is called unshifted if $\bs_L=0$.
\end{Definition}
\begin{Proposition}\label{scorrect}
For a given lattice $L$, the number $\bs_L$ is independent of the chosen adapted generators. It~is invariant under inner automorphisms of~$\Osc_1$, i.e., $\bs_{F(L)}=\bs_L$ for every inner automorphism $F$ of~$\Osc_1$.
\end{Proposition}
\begin{proof}
We show that, for an inner automorphism $F$, the number $\bs_L$ defined by generators $\alpha$, $\beta$, $\gamma$, $\delta$ of~$L$ coincides with the number $\bs_{F(L)}$ defined by the generators $\alpha':=F(\alpha),\dots$, \mbox{$\delta':=F(\delta)$}. This is obvious for $F_S$, $S\in\SO(2,\RR)$.
To show the claim for $F_\eta$, $\eta\in\RR^2$, consider $\eta=\eta_1\bar\alpha+\eta_2\bar\beta$.
Then $\alpha'=\eta\alpha\eta^{-1}=(\bar\alpha,z_\alpha-\eta_2,0)$, $\beta'=(\bar\beta,z_\beta+\eta_1,0)$,
$\gamma'=\gamma$ and $\delta'=(x_\delta'\bar \alpha+y_\delta'\bar \beta,z_\delta',\lambda)$, where
\begin{gather*}
\begin{pmatrix}x_\delta'\\y_\delta'\end{pmatrix}
= {x_\delta \choose y_\delta} +(I_2-S_q){\eta_1 \choose \eta_2}\end{gather*}
and
\begin{gather*}
z_\delta'=z_\delta+\dfrac{1}{2}\omega\left(\!{x_\delta\choose y_\delta},
(S_q+I_2){\eta_1\choose \eta_2}\!\right)
+\dfrac{1}{2}\omega\left( S_q{\eta_1\choose \eta_2}, {\eta_1\choose \eta_2}\! \right)\!.
\end{gather*}
We also denote the remaining data associated with $\alpha',\dots,\delta'$ by $z_\alpha'$, $z_\beta'$, $v'$, $w'$, $z_0'$. Then
\begin{gather*}
z_\alpha'=z_\alpha-\eta_2,\qquad
z_\beta'=z_\beta+\eta_1,\qquad
v'=v,\qquad w'=w,\qquad z_0'=z_0
\end{gather*}
and an easy calculation shows that $\bs_{F_\eta(L)}=\bs_L$.

Now, we~want to show that $\bs_L$ does not depend on the choice of the generators $\alpha$, $\beta$ and~$\delta$. We~will do that only for $q=1$ and $q=4$, the proof of the remaining cases follows the same strategy. We~may assume $z_\alpha=z_\beta=0$ since we already proved invariance under inner automorphisms $F_\eta$. Thus $\alpha=(\bar\alpha,0,0)$, $\beta=(\bar\beta,0,0)$, $\gamma=\big(0,\frac{1}{r},0\big)$ and $\delta=\big(\frac{v}{r}\bar\alpha+\frac{w}{r}\bar\beta,z_\delta,\lambda\big)$. We~will change the set of generators and denote the new data by $z_\alpha'$, $z_\beta'$, $z'_\delta$, $v'$, $w'$, $z_0'$ and $\bs_L'$. In~each case we will show that $\bs_L=\bs_L'$ holds.

Let us first consider $q=1$. Then $a=b=0$, thus $v,w\in\ZZ$. In~order to show the independence of the choice of~$\delta$ it suffices to change $\delta$ into $\delta'=\alpha\delta$, $\delta'=\beta\delta$ and $\delta'=\gamma\delta$. For the set of generators $\{\alpha,\beta,\gamma,\delta'=\alpha\delta\}$ we obtain
\begin{gather}\label{ad}
 z_\alpha'= z_\beta'=0,\qquad z_\delta'= z_\delta+\frac w{2r},\qquad v'=v+r,\qquad w'=w.
\end{gather}
Then
\begin{gather*}
\bs_L'-\bs_L =z_\delta'-\frac{1}{2}v'w'-z_\delta+\frac{1}{2}vw=\frac w{2r}-\frac12 rw=\frac w{r_0} \frac12 \big(r-r^3\big) \frac1{s_0r}\equiv 0\ {\rm mod}\ \frac{1}{s_0r}\ZZ
\end{gather*}
since $r_0|w$. Analogously, $\bs_L'=\bs_L$ for the generators $\{\alpha,\beta,\gamma,\delta'=\beta\delta\}$. Obviously, this is also true for $\{\alpha,\beta,\gamma,\delta'=\gamma\delta\}$. It~remains to check that we can replace $\alpha$, $\beta$ by $\alpha'=\alpha^{s_{11}}\beta^{s_{21}}\gamma^k$ and~$\beta'=\alpha^{s_{12}}\beta^{s_{22}}\gamma^l$ for arbitrary $k$, $l$ and $S=(s_{ij})_{i,j=1,2}\in\SL(2,\ZZ)$. To do so, it~suffices to consider the case $S=I_2, (k,l)\in\{(1,0),(0,1)\}$ and the case, where $k=l=0$ and $S$ is chosen from a set of generators of~$\SL(2,\ZZ)$. For $\{\alpha'=\alpha\gamma,\beta,\gamma,\delta\}$ we have
\begin{gather*}
z_\alpha'= 1/r,\qquad z_\beta'=0,\qquad z_\delta'= z_\delta,\qquad v'=v,\qquad w'=w,
\end{gather*}
thus
\begin{gather*}
\bs_L'-\bs_L = -\frac v{r^2}=-\frac v{r_0}\, \frac 1{s_0r} \equiv 0\ {\rm mod}\ \frac{1}{s_0r}\ZZ.
\end{gather*}
Analogously, $\bs_L'=\bs_L$ holds for $\{\alpha,\beta'=\beta\gamma,\gamma,\delta\}$. For $\{\alpha,\beta'=\alpha\beta,\gamma,\delta\}$
we have
\begin{gather*}
z_\alpha'= 0,\qquad z_\beta'=1/2,\qquad z_\delta'= z_\delta,\qquad v'=v-w,\qquad w'=w,
\end{gather*}
hence
\begin{gather*}
\bs_L'-\bs_L = \frac {w^2}{2}-\frac w{2r}= \frac{w\big(wr^2-r\big)}{2r_0}\,\frac 1{s_0r} \equiv 0\ {\rm mod}\ \frac{1}{s_0r}\ZZ.
\end{gather*}
Finally, for $\big\{\alpha'=\beta,\,\beta'=\alpha^{-1},\,\gamma,\,\delta\big\}$ we get
\begin{gather*}
 z_\alpha'= z_\beta'=0,\qquad z_\delta'= z_\delta,\qquad v'=w,\qquad w'=-v,
 \end{gather*}
which yields
\begin{gather*}
 \bs_L'-\bs_L = vw \equiv 0\ {\rm mod}\ \frac{1}{s_0r}\ZZ.
 \end{gather*}

For $q=4$, we~have $a,b=0$, hence $v,w\in\ZZ$. First assume that $r$ is even and $L$ is of type~$L_r^{4,+}$. Then $v+w$ is odd by Proposition~\ref{TF} and~\eqref{klvw}. Furthermore, $z_0=-\frac1{4r^2}\big(v^2+w^2\big)$. We~proceed as in the case of~$q=1$. For $\{\alpha,\beta,\gamma,\delta'=\alpha\delta\}$ we have the identities in~\eqref{ad} and $z_0'=-\frac1{4r^2}((v+r)^2+w^2)$, which gives
\begin{gather*}
\bs_L'-\bs_L=z_\delta'-z_0'-z_\delta+z_0=\frac1{2r} (w+v)+\frac14\ {\rm mod}\ \frac{1}{s_0r}\ZZ.
\end{gather*}
If $4|r$, then $s_0=2$ and $\bs_L'=\bs_L$ follows. Otherwise we have $s_0=1$ and the assertion follows from the fact that $v+w$ is odd. In~the same way we obtain~$\bs_L'=\bs_L$ for $\delta'=\beta\delta$. Again, $\bs_L'=\bs_L$ is obvious for $\delta'=\gamma\delta$. For $\{\alpha'=\alpha\gamma,\beta,\gamma,\delta\}$ we have
\begin{gather}\label{eo2}
z_\alpha'= 1/r,\qquad z_\beta'=0,\qquad z_\delta'= z_\delta,\qquad v'=v+1,\qquad w'=w+1,
\end{gather}
and $z_0'=-\frac1{4r^2}\big(v'^2+w'^2\big)$. This yields
\begin{gather*}
\bs_L'=z_\delta-\frac 1{2r^2}(v+w+1)-z_0'=z_\delta-z_0,
\end{gather*}
thus $\bs_L'=\bs_L$ already in~$\RR$. The same is true for $\{\alpha,\beta'=\beta\gamma,\gamma,\delta\}$. It~remains to consider generators $\big\{\alpha'=\alpha^{s_{11}}\beta^{s_{21}}, \beta'=\alpha^{s_{12}}\beta^{s_{22}}, \gamma, \delta \big\}$. Since these generators are supposed to be adapted, $S=(s_{ij})_{i,j=1,2}$ has to stabilise $S_4$. The stabiliser of~$S_4$ is generated by $S_4$, hence it suffices to check $\bs_L'$ for $\big\{\alpha'=\beta,\, \beta'=\alpha^{-1},\gamma,\delta\big\}$. For these generators we get
\begin{gather}\label{eo3}
 z_\alpha'= z_\beta'=0,\qquad z_\delta'= z_\delta,\qquad v'=w,\qquad w'=-v,\qquad z_0'=z_0,
\end{gather}
thus $\bs_L'=\bs_L$. Now assume that $L$ is of type $L_r^{4}$ for $r\in\NN_{>0}$. If~$r$ is even, then $v+w$ is even by Proposition~\ref{TF} and~\eqref{klvw}. We~compute $z_0=-\frac{(1-r\tilde r)^2}{4r^2}\big(v^2+w^2\big)$ and
\begin{gather*}
\bs_L=z_\delta-z_0+\frac{1}{2}vw\tilde r^2\ {\rm mod}\ \frac{1}{s_0r}\ZZ.
\end{gather*}
For $\{\alpha,\beta,\gamma,\delta'=\alpha\delta\}$, we~obtain the identities~\eqref{ad} and $z_0'=-\frac{(1-r\tilde r)^2}{4r^2}\big((v+r)^2+w^2\big)$, hence
\begin{gather*}
 \bs_L'-\bs_L=\frac w{2r}+\frac{(1-\tilde r r)^2}{4r}(2v+r)-\frac12 rw\tilde r\ {\rm mod}\ \frac{1}{s_0r}\ZZ.
 \end{gather*}
If $r$ is odd, then $4|(1-r)^2$. Moreover, $\tilde r=s_0=1$. This gives
\begin{gather*}
\frac w{2r}+\frac{(1-\tilde r r)^2}{4r}(2v+r)-\frac12 rw\tilde r \equiv \frac wr\, \frac{1-r^2}2 \equiv 0\ {\rm mod}\ \frac{1}{s_0r}\ZZ.
\end{gather*}
If $4|r$, then $\tilde r=0$ and $s_0=1$. Consequently,
\begin{gather}\label{eo}
\frac w{2r}+\frac{(1-\tilde r r)^2}{4r}(2v+r)-\frac12 rw\tilde r = \frac {v+w}2\, \frac 1r +\frac14 \equiv 0\ {\rm mod}\ \frac{1}{s_0r}\ZZ,
\end{gather}
since $v+w$ is even and $4|r$ implies that $1/4\equiv 0$ mod $\frac1r \ZZ$. If~$r\in 2+4\NN$, then $\tilde r=0$ and $s_0=2$. Hence equation~\eqref{eo} also holds since $v+w$ is even and $1/4\equiv 0$ mod~$\frac 1{2r}\ZZ$. We~obtain~$\bs_L'=\bs_L$. Similarly, this can be shown for $\{\alpha,\beta,\gamma,\delta'=\beta\delta\}$. Again, $\bs_L'=\bs_L$ is obvious for $\{\alpha,\beta,\gamma,\mbox{$\delta'=\gamma\delta$}\}$. For $\{\alpha'=\alpha\gamma,\beta,\gamma,\delta\}$, again the identities~\eqref{eo2} hold, whereas now $z_0'=-\frac{(1-\tilde r r)^2}{4r^2}\big(v'^2+w'^2\big)$. We~obtain
\begin{gather*}
\bs_L'=z_\delta-\frac1{2r^2}(v+w+1) -z_0'-\frac12(v+1)(w+1)\tilde r\ {\rm mod}\ \frac{1}{s_0r}\ZZ,
\end{gather*}
which implies
\begin{gather*}
\bs_L'-\bs_L=-\frac{\tilde r}r(v+w+1) \equiv 0\ {\rm mod}\ \frac{1}{s_0r}\ZZ.
\end{gather*}
Analogously, $\bs_L'=\bs_L$ for $\{\alpha,\beta'=\gamma\beta,\gamma,\delta\}$.
Finally, for the choice of generators $\{\alpha'=\beta,\, \beta'=\alpha^{-1},\gamma,\delta\}$ we get the identities~\eqref{eo3}, thus
\begin{gather*}
\bs_L'-\bs_L=vw\tilde r \equiv 0\ {\rm mod}\ \frac{1}{s_0r}\ZZ.
\end{gather*}
This finishes the proof for $q=4$.
\end{proof}

\begin{Theorem} \label{MMM}
Let $\cM$ denote the set of all isomorphism classes of lattices in~$\Osc_1$ with respect to inner automorphisms of~$\Osc_1$, $\cM_1\subset \cM$ the subset of isomorphism classes of normalised lattices and $\cM_0\subset\cM_1$ the subset of isomorphism classes of normalised unshifted lattices. Then we have bijections
\begin{gather*}
\cM_1\times\RR_{>0} \longrightarrow \cM,\qquad (L,a)\longmapsto F_S(L),\qquad S=a I_2,
\end{gather*}
and
\begin{gather}\label{M0M1}
\cM_0 \times \RR/\ZZ \longrightarrow \cM_1,\qquad (L,s) \longmapsto F_u(L),\qquad u=\frac {s}{s_0r\lambda},
\end{gather}
where $\lambda=\lambda(L)$, $r=r(L)$ and $s_0$ is associated with $L$ according to Definition~$\ref{de:inv}$.
\end{Theorem}
It remains to determine $\cM_0$, which we want to do before proving Theorem~\ref{MMM}.
\begin{Lemma}\label{lmlattice} Take $\lambda\in \pi\, \NN_{>0}$ or $\lambda= \lambda_0+2\pi\ZZ,\, \lambda_0\in \{\pi/3, \pi/2, 2\pi/3\}$ and choose $\mu+{\rm i}\nu\in\CC$ and $\xi_0=x_0+iy_0$ such that
\benur
\item $\mu+{\rm i}\nu\in \CC$, $\nu>0$, and $x_0,y_0\in \frac1r \ZZ$ for $\lambda \in \pi\,\NN_{>0}$,
\item $\mu+{\rm i}\nu = {\rm i}$ and $x_0,y_0\in \frac1r \ZZ$ for $\lambda \in \pi/2+2\pi\ZZ$,
\item $\mu+{\rm i}\nu ={\rm e}^{{\rm i}\pi/3}$ and $x_0\in \frac12+\frac1r \ZZ,\, y_0\in \frac1r \ZZ$ for $\lambda \in \pi/3+2\pi\ZZ$,
\item $\mu+{\rm i}\nu ={\rm e}^{{\rm i}\pi/3}$ and $x_0\in \frac1r \ZZ,\, y_0\in\frac12+ \frac1r \ZZ$ for $\lambda \in 2\pi/3+2\pi\ZZ$.
\end{enumerate}
Then the elements
\begin{gather*}
 l_1:=\bigg(\frac{1}{\sqrt\nu},0,0\bigg),\qquad
 l_2:=\bigg({-}\frac{\mu}{\sqrt\nu}+{\rm i}\sqrt\nu,0,0\bigg),\qquad
 l_3:=\bigg(0,\frac{1}{r},0\bigg),
 \\
l_4:=\bigg(\frac{1}{\sqrt\nu}x_0-\frac{\mu}{\sqrt\nu}y_0+{\rm i}\sqrt\nu y_0,z,\lambda \bigg),
\end{gather*}
for every $z\in\RR$ generate a lattice in~$\Osc_1$.
\end{Lemma}
\begin{Definition}
We denote the lattice obtained in Lemma~\ref{lmlattice} by
\begin{gather*}
L_r(\lambda,\mu,\nu,\xi_0,z)\subset \Osc_1.
\end{gather*}
\end{Definition}

The group $\SL(2,\ZZ)$ acts on the Poincar\'e half plane $\textbf{H}$. The set $\cF=\cF_+\cup\cF_-$ for
 \begin{gather}
\cF_+:= \bigg\{(\mu,\nu)\in\textbf{H}\,\bigg|\,\mu\in\bigg[0,\frac{1}{2}\bigg],\,
\mu^2+\nu^2\geq 1\bigg\},\nonumber
\\
\cF_-:= \bigg\{(\mu,\nu)\in\textbf{H}\,\bigg|\,\mu\in\bigg({-}\frac{1}{2},0\bigg),\,
\mu^2+\nu^2>1\bigg\} \label{F+}
\end{gather}
is a fundamental domain of this action.

For $M\in\SL(2,\ZZ)$, we~denote by $\la M \ra$ the subgroup of~$\SL(2,\ZZ)$ that is generated by $M$.

\begin{Theorem}\label{th:classUptoInner}
A complete list of normalised unshifted lattices of the oscillator group $\Osc_1$ up to inner automorphisms of~$\Osc_1$ is given by
\begin{enumerate}\itemsep=0pt
\item[$1.$] $L^1_{r_0}(r,\lambda,\mu,\nu,\iota_1,\iota_2):= L_r\big(\lambda,\mu,\nu,\frac{\iota_1}{r}+{\rm i}\frac{\iota_2}{r},z_0\big)$, where $z_0=0$ if $r\iota_1\iota_2$ even and $z_0=1/2$ else, for parameters
\begin{itemize}
\item[$\circ$] $r\in r_0\,\NN_{>0}$,
\item[$\circ$] $\lambda\in 2\pi\,\NN_{>0}$,
\item[$\circ$] $(\mu,\nu)\in\cF$,
\item[$\circ$] $\iota=(\iota_1,\iota_2)^\top\in\SSS\backslash\ZZ^2_r$ with $\gcd(\iota_1,\iota_2,r)=r_0$, where
\begin{gather*}
\SSS=\begin{cases}
\langle-I_2\rangle\simeq \ZZ_2, & \text{if}\quad \mu+{\rm i}\nu\notin\big\{{\rm i},{\rm e}^{\pi {\rm i}/3}\big\},
\\[0.5ex]
\langle S_4\rangle\simeq \ZZ_4, & \text{if}\quad \mu+{\rm i}\nu={\rm i},
\\[0.5ex]
\langle S_6\rangle\simeq \ZZ_6, & \text{if}\quad \mu+{\rm i}\nu={\rm e}^{\pi {\rm i}/3},
\end{cases}
\end{gather*}
and $S\in\SSS$ acts on~$\iota\in\ZZ^2_r$
by
$S\cdot\iota=S(\iota);$
\end{itemize}
\item[$2.$] $L^2_r(\lambda,\mu,\nu):=L_r(\lambda,\mu,\nu,0,0)$ for
\begin{itemize}
\item[$\circ$] $\lambda\in\pi+2\pi\NN$,
\item[$\circ$] $(\mu,\nu)\in\cF;$
\end{itemize}
\item[$3.$] $L^{2,+}_r(\lambda,\mu,\nu,\iota_1,\iota_2):= L_r\big(\lambda,\mu,\nu,\frac{\iota_1}{r}+{\rm i}\frac{\iota_2}{r},0\big)$, where $r$ is even, for
\begin{itemize}
\item[$\circ$] $\lambda\in\pi+2\pi\NN$,
\item[$\circ$] $(\mu,\nu)\in\cF$,
\item[$\circ$]
$(\iota_1,\iota_2)\in \begin{cases}	
\{(1,0),(0,1),(1,1)\}, &\text{if}\quad
\mu+{\rm i}\nu\notin\{{\rm i},{\rm e}^{\pi {\rm i}/3}\}, \\[0.5ex]
\{(1,0),(1,1)\}, &\text{if}\quad \mu+{\rm i}\nu={\rm i},\\[0.5ex]
\{(1,1)\} ,&\text{if}\quad \mu+{\rm i}\nu={\rm e}^{\pi {\rm i}/3};
\end{cases}$
\end{itemize}
\item[$4.$] $L^3_r(\lambda):=L_r\big(\lambda,\frac{1}{2},\frac{\sqrt{3}}{2},\xi_0,z_0\big)$, where $(\xi_0,z_0)=\big(\frac1r+{\rm i}\frac1{2r},-\frac{1}{8r^2}\big)$ if $r$ is~odd and $(\xi_0,z_0)=(0,0)$ if $r$ is even,~for
\begin{itemize}
\item[$\circ$]
$\lambda \in \frac {2\pi}3 +2\pi \ZZ;$
\end{itemize}
\item[$5.$] $L^{3,+}_r(\lambda):=L_r\big(\lambda,\frac{1}{2},\frac{\sqrt{3}}{2},\xi_0,z_0\big)$, where $3|r$ and~$(\xi_0,z_0)=\big({\rm i}\frac1{2r},-\frac{1}{24r^2}\big)$ if $r$ is odd
and $(\xi_0,z_0)=\big({-}\frac1r,-\frac{1}{6r^2}\big)$ if $r$ is even for
\begin{itemize}
\item[$\circ$]
$\lambda \in \frac {2\pi}3 +2\pi \ZZ;$
\end{itemize}
\item[$6.$] $L^4_r(\lambda):= L_r(\lambda,0,1,0,0)$ for
\begin{itemize}
\item[$\circ$] $\lambda \in \frac \pi2 +2\pi \ZZ;$
\end{itemize}
\item[$7.$] $L^{4,+}_r(\lambda):=L_r\big(\lambda,0,1,\frac1r,-\frac1{4r^2}\big)$, where $r$ is even, for
\begin{itemize}
\item[$\circ$]
$\lambda \in \frac \pi2 +2\pi \ZZ;$
\end{itemize}
\item[$8.$] $L^6_r(\lambda):=L_r\big(\lambda,\frac{1}{2},\frac{\sqrt{3}}{2},\xi_0,z_0\big)$, where $(\xi_0,z_0)=\big(\frac{1}{2r}, -\frac{1}{8r^2}\big)$ if $r$ is odd and $(\xi_0,z_0)=(0,0)$ if~$r$ is even, for
\begin{itemize}
\item[$\circ$]
$\lambda \in \frac \pi3 +2\pi \ZZ$
\end{itemize}
\end{enumerate}
with $r, r_0\in\NN_{>0}$.

All lattices $L^1_{r_0}(r,\lambda,\mu,\nu,\iota_1,\iota_2)$ are of type $L^1_{r_0}$. Analogously, the lattices $L^q_r(\dots)$ and $L^{q,+}(\dots)$, $q\not=1$, are of type $L^q_r$ and $L^{q,+}_r$, respectively, where the dots stand for arbitrary parameters used above.
\end{Theorem}

\begin{Definition}
Lattices in the list in Theorem~\ref{th:classUptoInner} are called standard lattices.
\end{Definition}

\begin{Remark}\label{z0}
In Theorem~\ref{th:classUptoInner}, $z_0$ is chosen such that $l_4^q=(0,0,q\lambda)$ for $q=\ord(\lambda)\in\{2,3,4,6\}$. Moreover, using Remark~\ref{rm4}$(iii)$ one can easily check that for every standard lattice with $q=\ord(\lambda)\in\{2,3,4,6\}$ the $z_0$ in Theorem~\ref{th:classUptoInner} in fact coincides with the corresponding $z_0$ given in Definition~\ref{de:inv}.
\end{Remark}
The remainder of this subsection deals with the proof of Theorems~\ref{MMM} and~\ref{th:classUptoInner}. The following lemma shows that the map defined in~\eqref{M0M1} is well-defined.
\begin{Lemma}\label{pr:shift}
Let $L$ be a normalised lattice, $\lambda:=\lambda(L)$, $r:=r(L)$. If~$u=1/(s_0r\lambda)$, then there exists an inner automorphism $F$ of~$\Osc_1$ such that $F_u (L)=F(L)$.
\end{Lemma}
\begin{proof}
Let $L$ be a normalised lattice. Choose adapted generators and define $s_0$ as in Definition~\ref{de:inv}. If~$s_0=1$, we~see immediately that $F_u(L)=L$.
Hence, we~only have to consider the situation, where $L$ is of type $L^1_{r_0}$, $L^4_r$ with $r\in4\NN+2$ and $L^{q,+}_r$, $q\in\{2,3,4\}$ with $r\in4\NN_{>0}$ for~$q=4$.

At first, assume that $L$ is of type $L^1_{r_0}$. We~choose $\eta=-\frac1r\big({b_1}\bar\alpha+{b_2}\bar\beta\big)$ with $b_1,b_2\in\ZZ$ such that $b_2rx_\delta-b_1ry_\delta-b_3r=r_0$ for some $b_3\in\ZZ$, which is possible since $r_0=\gcd(rx_\delta,ry_\delta,r)$.
Then $F_\eta(L)=F_{1/(s_0r\lambda)}(L)$.

Now, assume that $L$ is of type $L^4_{r}$ with $r\in 2+4\NN$.
Then $v+w=r(x_\delta+y_\delta+2z_\alpha)$ is even because of Proposition~\ref{TF} and equation~\eqref{klvw}, thus $r(x_\delta+y_\delta+2z_\alpha)=2m$ for some $m\in\ZZ$.
Take $\eta=-\frac{1}{2}\big(\bar\alpha+\bar\beta\big)$.
Since $F_\eta(\alpha)$ and $F_\eta(\beta)$ are contained in~$\la \alpha,\beta,\gamma\ra$ and{\samepage
\begin{gather*}
F_\eta(\delta)= \bigg((x_\delta-1)\bar\alpha+y_\delta\bar\beta, z_\delta+\frac12 x_\delta -\frac14,\lambda\bigg)
\\ \hphantom{F_\eta(\delta)}
{}=(-\bar\alpha,-z_\alpha,0)\cdot \bigg(0,\frac12 x_\delta+\frac12 y_\delta+z_\alpha,0\bigg)\cdot \bigg(x_\delta\bar\alpha+y_\delta\bar\beta,z_\delta-\frac14,\lambda\bigg)
=\alpha^{-1}\gamma^m\,F_{-\frac1{4\lambda}}(\delta),
\end{gather*}
we obtain~$F_\eta(L)=F_{-1/(4\lambda)}(L)=F_{1/(s_0r\lambda)}(L)$.}

Now, assume that $L$ is of type $L^{4,+}_{r}$ with $r\in 4\NN_{>0}$, i.e., $r=4r'$ for some $r'\in\NN_{>0}$.
Then $v+w=r(x_\delta+y_\delta+2z_\alpha)$ is odd, i.e., $r(x_\delta+y_\delta+2z_\alpha)-1=2m$ for some $m\in\ZZ$.
Take $\eta=-\frac{1}{2}\big(\bar\alpha+\bar\beta\big)$ as above,
which gives now
\begin{gather*}
F_\eta(\delta)= \bigg((x_\delta-1)\bar\alpha+y_\delta\bar\beta, z_\delta+\frac12 x_\delta -\frac14,\lambda\bigg) = \alpha^{-1}\gamma^{m-r'} F_\frac1{2r\lambda}(\delta).
\end{gather*}
Hence $F_\eta(L)=F_{1/(s_0r\lambda)}(L)$.

Now, assume that $L$ is of type $L^{2,+}_{r}$ with $r\in 2\NN_{>0}$.
Then $v=r(x_\delta-2z_\beta)$ is odd or~$w=r(y_\delta+2z_\alpha)$ is odd by Proposition~\ref{TF} and equation~\eqref{klvw}.
For $v+w$ odd, i.e., $v+w=2m-1$, take $\eta=\frac{1}{2}\big(\bar\alpha-\bar\beta\big)$.
Then $F_\eta(\alpha)$ and $F_\eta(\beta)$ are contained in $\la \alpha,\beta,\gamma\ra$ and
\begin{gather*}
F_\eta(\delta)=\big((x_\delta+1)\bar\alpha+(y_\delta-1)\bar\beta, z_\delta,\lambda\big)=\alpha\beta^{-1}\gamma^{-m+r'}F_{\frac1{2r\lambda}}(\delta),
\end{gather*}
where $r=2r'$. We~obtain~$F_\eta(L)=F_{1/(s_0r\lambda)}(L)$. If~$v+w$ is even, then $v$ and $w$ are odd, thus $w=2m-1$. We~put $\eta=\frac{1}{2}\bar\alpha$. Then $F_\eta(\alpha), F_\eta(\beta) \in\la \alpha,\beta,\gamma\ra$ and
\begin{gather*}
F_\eta(\delta)=\big((x_\delta+1)\bar\alpha+y_\delta\bar\beta, z_\delta,\lambda\big)=\alpha\gamma^{-m}F_{\frac1{2r\lambda}}(\delta),
\end{gather*}
hence $F_\eta(L)=F_{1/(s_0r\lambda)}(L)$.

Finally, assume that $L$ is of type $L^{3,+}_{r}$ with $r\in 3\NN_{>0}$. By~Proposition~\ref{TF} and equation~\eqref{klvw}, $2v-w+r/2\equiv \kappa\ {\rm mod}\ 3$ for $\kappa\in\{1,-1\}$. Since $3|r$, this implies that $2v-w-\kappa r/2=3m+\kappa$ for some $m\in\ZZ$.
Thus $r(2x_\delta-y_\delta -3z_\beta-\kappa/2)=3m+\kappa$ by~\eqref{abvw}.
Consider $\eta=-\frac{\kappa}{3}\big(\bar\alpha-\bar\beta\big)$.
Then $F_\eta(\alpha), F_\eta(\beta) \in\la \alpha,\beta,\gamma\ra$ and
\begin{gather*}
F_\eta(\delta)=\bigg(x_\delta\bar\alpha+(y_\delta+\kappa)\bar\beta, \frac\kappa6 x_\delta-\frac\kappa3 y_\delta-\frac16+z_\delta ,\lambda\bigg)
\\ \hphantom{F_\eta(\delta)}
{}= \beta^\kappa \bigg(x_\delta\bar\alpha+y_\delta\bar\beta, \frac\kappa{3r}(3m+\kappa)+z_\delta ,\lambda\bigg)
= \beta^\kappa\gamma^{\kappa m}F_{\frac 1{3r\lambda}}(\delta)
\end{gather*}
and we obtain~$F_\eta(L)=F_{1/(s_0r\lambda)}(L)$.
\end{proof}

\begin{proof}[Proof of Theorem~\ref{MMM}.]
Lemma~\ref{pr:shift} already shows that the map in (\ref{M0M1}) is well-defined. The injectivity follows straightforwardly using Remark~\ref{rm4}$(iv)$. Furthermore, the map is surjective. Indeed, if $L$ is normalised, then $F_u(L)$ with $u=-\lambda^{-1}\cdot\textbf{s}_L$ is normalised and unshifted and $\big([F_u(L)],\textbf{s}_L s_0 r\big)\in\cM_0\times\RR/\ZZ$ maps to $[L]$.

Now, let $L$ be an arbitrary lattice in~$\Osc_1$.
Since $[L\cap H,L\cap H]$ is a lattice in~$Z(H)$ we can choose $a>0$ such that $(0,a^2,0)$ generates $[L\cap H,L\cap H]$.
Then $F_S(L)$ with $S={a}^{-1}I_2$ is a~normalised lattice. This proves the surjectivity of the first map in Theorem~\ref{MMM}. Since $a>0$ of the isomorphism class of~$F_S(L)$, $S=aI_2$ with respect to inner automorphisms is uniquely determined by $[L\cap H,L\cap H]$ the injectivity follows as well.
\end{proof}

\begin{Lemma} \label{lm:generatorChoice}
Let $L$ be a normalised lattice. There exists a map $S\in\SO(2)$ such that $F_S(L)$ is generated by elements
\begin{gather}\label{eq:generators'}
\alpha=\bigg(\frac 1{\sqrt\nu},z_\alpha,0\bigg),\quad
\beta=\bigg({-}\frac{\mu}{\sqrt\nu}+{\rm i}\sqrt\nu ,z_\beta,0\bigg),\quad
\gamma=\bigg(0,\frac 1r,0\bigg),\quad \delta
\end{gather}
 for an appropriate $\mu+{\rm i}\nu\in\cF$. If~$q:=\ord(\lambda(L))=4$, then $\mu+{\rm i}\nu={\rm i}$ and if $q\in\{3,6\}$, then $\mu+{\rm i}\nu= {\rm e}^{\pi {\rm i}/3}$. If~$\delta\in L$ is chosen such that its projection to the $\RR$-factor of~$\Osc_1$ equals $\lambda(L)$, then $\alpha$, $\beta$, $\gamma$, $\delta$ is adapted.
\end{Lemma}

\begin{proof}
We can choose elements $\gamma=(0,1/r,0)\in L\cap Z(H)$ and $\alpha', \beta' \in L\cap H$ such that $[\alpha',\beta']=\gamma^r$. Let~$\bar \alpha'$ and $\bar\beta'$ be the projections to $H/Z(H)$ of~$\alpha'$ and $\beta'$, respectively.
Then the matrix $T':=\big(\bar \alpha',\bar\beta'\big)$ consisting of columns $\bar \alpha'$ and $\bar \beta'$ is in~$\SL(2,\RR)$. Let~$S$ be in~$\SO(2)$ and consider the inner automorphism $F_S\colon \Osc_1 \to \Osc_1$, $(\xi,z,t)\mapsto (S\xi,z,t)$. Then $\alpha$, $\beta$, $\gamma$ are generators of~$F_S(L)\cap H$ satisfying $[\alpha,\beta]=\gamma^r$ if and only if
\begin{gather*}
\big(\overline{\alpha},\overline{\beta}\big) = ST'M
\end{gather*}
for some $M\in\SL(2,\ZZ)$. Consequently, if $T\in \SL(2,\RR)$ represents the same element in
\begin{gather*}
\SO(2)\backslash \SL(2,\RR)/\SL(2,\ZZ)
\end{gather*} as $T'$, then there exist an inner automorphism $F_S$ and generators $\alpha,\beta,\gamma$ of~$F_S(L)\cap H$ such that $T=\big(\overline{\alpha},\overline{\beta}\big)$. We~consider the diffeomorphism
\begin{gather*}
\SO(2)\backslash \SL(2,\RR)/\SL(2,\ZZ) \longrightarrow \SL(2,\ZZ)\backslash \SL(2,\RR)/\SO(2)
\end{gather*}
induced by $A\mapsto A^{-1}$ and the diffeomorphism
\begin{gather*}
\SL(2,\RR)/\SO(2) \longrightarrow \textbf{H},\qquad A\cdot \SO(2)\longmapsto A\cdot {\rm i}.
\end{gather*}
The latter one has a right inverse defined by
\begin{gather}\label{Tmunu}
\textbf{H}\longrightarrow\SL(2,\RR)/\SO(2) ,\qquad
\mu+{\rm i}\nu \longmapsto T_{\mu,\nu}:=
\begin{pmatrix}
\sqrt\nu&\frac{\mu}{\sqrt\nu}
\\[1ex]
0&\frac 1{\sqrt\nu}
\end{pmatrix}\!.
\end{gather}
We define $\mu',\nu'\in\RR$ by $\mu'+{\rm i}\nu'=T'^{-1} \cdot {\rm i}$. Let~$\mu+{\rm i}\nu=M \cdot (\mu'+{\rm i}\nu')$, $M\in\SL(2,\ZZ)$, be an arbitrary point in the $\SL(2,\ZZ)$-orbit of~$\mu'+{\rm i}\nu'$. By~the above considerations, $T'$ and $T^{-1}_{\mu,\nu}$ represent the same point in~$\SO(2)\backslash \SL(2,\RR)/\SL(2,\ZZ)$ since
\begin{gather*}
T_{\mu,\nu}\cdot {\rm i} =\mu+{\rm i}\nu=M \cdot (\mu'+{\rm i}\nu')=MT'^{-1}\cdot {\rm i}.
\end{gather*}
In particular, we~can choose the element $\mu+{\rm i}\nu$ such that it belongs to $\cF$ and we put $T:=(T_{\mu,\nu})^{-1}$, which proves the existence of generators satisfying~\eqref{eq:generators'}.

The lattice of~$\RR^2\cong \CC$ spanned by $\bar \alpha = 1/\sqrt \nu$ and $\bar\beta=-\frac{\mu}{\sqrt\nu}+{\rm i}\sqrt\nu$ has to be invariant under~${\rm e}^{{\rm i}\lambda(L)}$. This gives $\mu+{\rm i}\nu={\rm i}$ for $q=4$ and $\mu+{\rm i}\nu= {\rm e}^{\pi {\rm i}/3}$ if $q\in\{3,6\}$.
\end{proof}

\begin{proof}[Proof of Theorem~\ref{th:classUptoInner}.]
Clearly, the standard lattices are normalised and unshifted. Now let $L$ be an unshifted normalised lattice and choose adapted generators as in Lemma~\ref{lm:generatorChoice}. As~usual, we~denote by $\bar \alpha$ and $\bar \beta$ the projections of~$\alpha$ and $\beta$ to $\RR^2$.
Take $\eta=-z_\beta\bar\alpha+z_\alpha\bar\beta$. Then $F_{\eta}(\alpha)=(\bar\alpha,0,0)$, $F_{\eta}(\beta)=\big(\bar\beta,0,0\big)$ and $F_{\eta}(\gamma)=\gamma$. Thus we may assume that we have generators
\begin{gather}\label{choice2}
\alpha=\!\bigg(\!\frac 1{\sqrt\nu},0,0\!\bigg),\quad
\beta=\!\bigg(\!{-}\frac{\mu}{\sqrt\nu}+{\rm i}\sqrt\nu,0,0\!\bigg),\quad
\gamma=\!\bigg(\!0,\frac 1r,0\!\bigg),\quad
\delta=\!\big(x_\delta\bar\alpha+y_\delta\bar\beta,z_\delta,\lambda\big),\!\!
\end{gather}
where $\mu+{\rm i}\nu\in\cF$, $\mu+{\rm i}\nu={\rm i}$ for $q=4$ and $\mu+{\rm i}\nu= {\rm e}^{\pi {\rm i}/3}$ for $q\in\{3,6\}$, where $\lambda=\lambda(L)$ and $q=\ord(\lambda(L))$.

Of course, inner automorphisms of~$\Osc_1$ do not change the type of the lattice and they do not change $\lambda$.

Let us first consider the case $q=1$. Let~$L$ be unshifted with generators as in~\eqref{choice2}. In~particular, $T:=(\bar \alpha,\bar\beta)=T_{\mu,\nu}^{-1}$. Define $v$, $w$ as in~\eqref{abvw}. Then $v=rx_\delta\in\ZZ$ and $w=ry_\delta\in\ZZ$ by Remark~\ref{rm4}$(ii)$. Hence $L=L_r\big(\lambda,\mu,\nu,\frac vr+{\rm i}\frac wr,z_\delta\big)$. Inner automorphisms of~$\Osc_1$ do not change the number $r$. Furthermore, under our assumption that $\mu+{\rm i}\nu$ belongs to $\cF$, also $\mu$ and $\lambda$ are uniquely determined. Therefore, we~consider the set
\begin{gather*}
{\mathcal L}:=\bigg\{L=L_r\bigg(\lambda,\mu,\nu,\frac {\iota_1}r+{\rm i}\frac {\iota_2}r,z\bigg)\,\bigg|\, \iota_1,\iota_2\in\ZZ;\, z\in\RR \text{ such that } \bs_L=0\bigg\}
\end{gather*}
for fixed $r>0$, $\lambda\in 2\pi\NN_{>0}$ and $\mu+{\rm i}\nu\in\cF$. Lattices in~$\mathcal L$ that differ only in the parameter $z$ are considered to be equivalent. Then the set ${\mathcal L}/_{\sim}$ of equivalence classes is in bijection with $\ZZ_r\times\ZZ_r$. We~denote by $[\iota_1,\iota_2]$ the equivalence class of the lattice $L_r(\lambda,\mu,\nu,\frac {\iota_1}r+{\rm i}\frac {\iota_2}r,z)\in {\mathcal L}$. Equivalent lattices in~${\mathcal L}$ are isomorphic via an inner automorphism of~$\Osc_1$, see Lemma~\ref{pr:shift}. Thus it remains to check which equivalence classes can be represented by lattices that are isomorphic via an inner automorphism. Let~$F$ be an inner automorphism of~$\Osc_1$. If~there exist lattices $L_1,L_2\in{\mathcal L}$ such that $F(L_1)=L_2$, then an easy calculation shows that $F$ belongs to the set
${\mathcal I}$ consisting of automorphisms $F_\eta F_S$ that satisfy the conditions
\begin{gather*}
M_S= \begin{pmatrix}m_1&m_2\\m_3&m_4\end{pmatrix}:=T^{-1}ST\in
\begin{cases}
\la S_4 \ra, & \text{if}\quad\mu+{\rm i}\nu={\rm i},\\
\la S_6 \ra, & \text{if}\quad \mu+{\rm i}\nu={\rm e}^{{\rm i}\pi/3},\\
\la -I_2 \ra, & \text{else},
\end{cases}
\end{gather*}
and
\begin{gather*}
T^{-1} \eta\in {\frac12} (m_1 m_2, m_3 m_4)^\top + \frac1r\ZZ\times\frac1r \ZZ.
\end{gather*}
Conversely, each element of~${\mathcal I}$ maps equivalence classes of~${\mathcal L}$ to equivalence classes of~${\mathcal L}$. More exactly, $F_\eta F_S([\iota_1,\iota_2])= [\iota_1',\iota_2']$ for $(\iota_1',\iota_2')^\top= M_S\cdot (\iota_1,\iota_2)^\top$, which proves the claim for $q=1$.

Let $L$ be an unshifted normalised lattice of type $L^2_r$. According to the above considerations, modulo an appropriate inner automorphism, $L$ is generated by elements given as in~\eqref{choice2}, where $\lambda\in\pi+2\pi\,\NN$. Let~$v$, $w$ be defined as in~\eqref{abvw}.
Note that $v$ and $w$ are even for $r\in2\NN_{>0}$ by~Proposition~\ref{TF} and~equation~\eqref{klvw}.
Then
\begin{gather*}
\delta= \bigg(\frac{v}{r}\bar\alpha+\frac{w}{r}\bar\beta,z_\delta,\lambda\bigg)
= \alpha^{\tilde rv}\beta^{\tilde rw}\bigg(\frac{1}{r}(v-r\tilde rv)\bar\alpha+\frac{1}{r}(w-r\tilde rw)\bar\beta,z_\delta-\frac12 vw\tilde r^2,\lambda\bigg).
\end{gather*}
Thus, $L$ is generated by $\alpha$, $\beta$, $\gamma$ and
\begin{gather*}
\delta':=\bigg(\bigg(\frac{1}{r}v-\tilde rv\bigg)\bar\alpha+\bigg(\frac{1}{r}w-\tilde rw\bigg)\bar\beta,z_\delta-\frac12 vw\tilde r^2,\lambda\bigg).
\end{gather*}
We define $\eta=\eta_1\bar \alpha +\eta_2\bar\beta$ by
\begin{gather*}
(\eta_1,\eta_2)^\top=-\frac{1}{2r}(v-r\tilde rv,\,w-r\tilde rw)^\top.
\end{gather*}
Since $v$ and $w$ are even for even $r$, $F_{\eta}$ preserves the subgroup of~$L\cap H=\la \alpha,\beta,\gamma\ra$. Furthermore, $F_{\eta}(\delta')=(0,z_\delta',\lambda)$ for an appropriate $z_\delta'$. Since with $L$ also $F_\eta(L)$ is unshifted and since $F_\eta(L)$ arises by a shift $F_u$ from the unshifted lattice $L^2_r(\lambda,\mu,\nu)=\big\langle\,\alpha,\,\beta,\,\gamma,\,(0,0,\lambda)\,\big\rangle$, there exists an~inner automorphism that maps $F_{ \eta}(L)$ to $L^2_r(\lambda,\mu,\nu)$, see Lemma~\ref{pr:shift}.

If $L$ is unshifted, normalised and of type $L^{2,+}_r$, then we again may assume that $L$ is generated by elements of the form~\eqref{choice2}, where $\lambda\in\pi/2+2\pi\cdot\ZZ$. In~this case $v$ is odd or $w$ is odd by~Proposition~\ref{TF} and~equation~\eqref{klvw}. Denote $\tilde v=\rem_2(v)$ and $\tilde w=\rem_2(w)$.
Here we define $\eta=\eta_1\bar \alpha +\eta_2\bar\beta$ by
\begin{gather*}
(\eta_1,\eta_2)^\top=-\frac{1}{2r}(v-\tilde v,\,w-\tilde w)^\top.
\end{gather*}
Since $v-\tilde v$, $w-\tilde w$ and $r$ are even, $F_{\eta}$ preserves the subgroup of~$L\cap H=\la \alpha,\beta,\gamma\ra$. Furthermore, $F_{\eta}(\delta)=(\tilde v\bar\alpha+\tilde w\bar\beta,z_\delta',\lambda)$ for an appropriate $z_\delta'$.
Now we again use Lemma~\ref{pr:shift} to see that there exists an inner automorphism that maps $F_\eta(L)$ to the unshifted lattice $L_r(\lambda,\mu,\nu,\tilde v/r+{\rm i}\tilde w/r,0)$.
For $\mu+{\rm i}\nu={\rm i}$, we~have $F_{S_4}\big(L^{2,+}_r(\lambda,\mu,\nu,1,0)\big)=L_r(\lambda,\mu,\nu,{\rm i}/r,0)$. Furthermore, for $\mu+{\rm i}\nu={\rm e}^{{\rm i}\pi/3}$, we~have $F_{S}(L_r(\lambda,\mu,\nu,1/r,0))=L^{2,+}_r(\lambda,\mu,\nu,1,1)$ and $F_{S}\big(L^{2,+}_r(\lambda,\mu,\nu,1,1)\big)=L_r(\lambda,\mu,\nu,{\rm i}/r,0)$, where $S=T_{\mu,\nu}^{-1}S_6T_{\mu,\nu}$ for $T_{\mu,\nu}$ as defined in~\eqref{Tmunu}.
This shows that $L$ is isomorphic to some standard lattice of type $L^{2,+}_r$ given in the list of the theorem. Conversely, it~is not hard to see that these standard lattices are not isomorphic to each other, which proves the assertion for type~$L^{2,+}_r$.

Let $L$ be an unshifted normalised lattice of type $L^4_r$. As~above we may assume that $L$ is generated by elements given as in~\eqref{choice2}, where $\lambda\in\pi/2+2\pi\cdot\ZZ$. Let~$v$, $w$ be defined as in~\eqref{abvw}. Then $v+w$ is even for $r\in2\NN_{>0}$ by Proposition~\ref{TF} and~equation~\eqref{klvw}. We~have
\begin{gather*}
\delta= \bigg(\frac{v}{r}\bar\alpha+\frac{w}{r}\bar\beta,z_\delta,\lambda\bigg)
= \alpha^{\tilde rv}\beta^{\tilde rw}\bigg(\frac{1}{r}(v-r\tilde rv)\bar\alpha+\frac{1}{r}(w-r\tilde rw)\bar\beta,z_\delta-\frac12 vw\tilde r^2,\lambda\bigg).
\end{gather*}
Thus, $L$ is generated by $\alpha$, $\beta$, $\gamma$ and
\begin{gather*}
\delta':=\bigg(\bigg(\frac{1}{r}v-\tilde rv\bigg)\bar\alpha+\bigg(\frac{1}{r}w-\tilde rw\bigg)\bar\beta,z_\delta-\frac12 vw\tilde r^2,\lambda\bigg).
\end{gather*}
We define $\eta=\eta_1\bar \alpha +\eta_2\bar\beta$ by
\begin{gather*}
r(S_4-I_2)(\eta_1,\eta_2)^\top=(v-r\tilde rv,\,w-r\tilde rw)^\top.
\end{gather*}
Since $v+w$ is even for even $r$, $F_{\eta}$ preserves the subgroup of~$L\cap H=\la \alpha,\beta,\gamma\ra$. Furthermore, $F_{\eta}(\delta')=(0,z_\delta',\lambda)$ for an appropriate $z_\delta'$. Since $F_\eta(L)$ is also unshifted and since $F_\eta(L)$ arises by a~shift $F_u$ from the unshifted lattice $L^4_r(\lambda)=\big\langle\alpha,\beta,\gamma,(0,0,\lambda)\big\rangle$, there exists an inner automorphism that maps $F_{ \eta}(L)$ to $L^4_r(\lambda)$,
see Lemma~\ref{pr:shift}.

If $L$ is unshifted, normalised and of type $L^{4,+}_r$, then we again may assume that $L$ is generated by elements of the form~\eqref{choice2}, where $\lambda\in\pi/2+2\pi \ZZ$. In~this case $v+w$ is odd. Here we define $\eta=\eta_1\bar \alpha +\eta_2\bar\beta$ by
\begin{gather*}
r(S_4-I_2)(\eta_1,\eta_2)^\top=(v-1,\,w)^\top.
\end{gather*}
Since $v+w$ is odd, $F_{\eta}$ preserves the subgroup of~$L\cap H=\la \alpha,\beta,\gamma\ra$. Furthermore, $F_{\eta}(\delta)=\big(1/r,z_\delta',\lambda\big)$ for an appropriate $z_\delta'$. Now we can argue as above that there exists an inner automorphism that maps $L$ to $L^{4,+}_r(\lambda)$,

The remaining cases $q=3$ and $q=6$ follow a similar strategy.
\end{proof}

\subsection[Classification of lattices up to all automorphisms of~$\Osc_1$]
{Classification of lattices up to all automorphisms of~$\boldsymbol{\Osc_1}$}

To complete this subsection, we~recall from~\cite{F} the classification of lattices in~$\Osc_1$ up to all automorphisms of~$\Osc_1$, i.e., we~consider lattices $L_1,L_2\subset \Osc_1$ as isomorphic if and only if there exists an automorphism $F$ of~$\Osc_1$ such that $F(L_1)=L_2$.

Clearly, every lattice of~$\Osc_1$ is isomorphic to some normalised unshifted lattice in the list of~Theorem~\ref{th:classUptoInner}. On the other hand, there are standard lattices of type $L^1_{r_0}$, $L^2_r$ and $L^{2,+}_r$ which are isomorphic to each other with respect to all automorphisms. Which ones are isomorphic was shown in~\cite{F}.

Before formulating the classification result, let us recall some notions concerning the group $\GL(2,\ZZ)$. We~can extend the action of~$\SL(2,\ZZ)$ on the Poincar\'e half plane $\textbf{H}$ to $\GL(2,\ZZ)$ by~$S\cdot z:=-\bar z$ for $S:=\diag(1,-1)\in\GL(2,\ZZ)$. The set
 $\cF_+$ defined in~\eqref{F+}
 is a fundamental domain of this action. Furthermore, for $M_1,\dots,M_k\in\GL(2,\ZZ)$, we~denote by $\la M_1,\dots,M_k \ra$ the subgroup of~$\GL(2,\ZZ)$ that is generated by $M_1,\dots,M_k$. In~particular, we~will consider subgroups of~$\GL(2,\ZZ)$ that are isomorphic to the dihedral groups $D_2$, $D_4$ and $D_6$.

\begin{Theorem}[\cite{F}]
A complete list of lattices of the oscillator group up to automorphisms of~$\Osc_1$ is given by the list of standard lattices in Theorem~$\ref{th:classUptoInner}$ where we replace the cases~$1$ to~$3$~by
\begin{enumerate}\itemsep=0pt
\item[$1.$] $L^1_{r_0}(r,\lambda,\mu,\nu,\iota_1,\iota_2):= L_r\big(\lambda,\mu,\nu,\frac{\iota_1}{r}+{\rm i}\frac{\iota_2}{r},z_0\big)$, where $z_0=0$ if $r\iota_1\iota_2$ even and $z_0=1/2$ else,~for
\begin{itemize}
\item[$\circ$] $r\in r_0\,\NN_{>0}$,
\item[$\circ$] $\lambda\in 2\pi\,\NN_{>0}$,
\item[$\circ$] $(\mu,\nu)\in\cF_+$,
\item[$\circ$] $\iota=(\iota_1,\iota_2)^\top\in\SSS\backslash\ZZ^2_r$ with $\gcd(\iota_1,\iota_2,r)=r_0$, where
\begin{gather*}
\SSS=\begin{cases}
\langle-I_2\rangle\simeq \ZZ_2, & \text{if}\quad \mu^2+\nu^2>1,\quad \mu\notin\big\{0,\frac{1}{2}\big\},
\\[0.5ex]
\big\langle \left(\begin{smallmatrix}1&0\\0&-1\end{smallmatrix}\right),-I_2\big\rangle \simeq D_2,& \text{if}\quad \mu^2+\nu^2>1,\quad \mu=0,
\\[0.5ex]
\big\langle \left(\begin{smallmatrix}1&-1\\0&-1\end{smallmatrix}\right),-I_2\big\rangle\simeq D_2, & \text{if}\quad \mu^2+\nu^2>1,\quad \mu=\frac{1}{2},
\\[0.5ex]
\big\langle \left(\begin{smallmatrix}0&1\\1&0\end{smallmatrix}\right),-I_2\big\rangle\simeq D_2, & \text{if}\quad \mu^2+\nu^2=1,\quad \mu\notin\big\{0,\frac{1}{2}\big\},
\\[0.5ex]
\big\langle \left(\begin{smallmatrix}1&0\\0&-1\end{smallmatrix}\right),\left(\begin{smallmatrix}0&-1\\1&0\end{smallmatrix}\right)\big\rangle\simeq D_4, & \text{if}\quad \mu=0,\quad \nu=1,
\\
\big\langle \left(\begin{smallmatrix}1&-1\\0&-1\end{smallmatrix}\right),\left(\begin{smallmatrix}1&-1\\1&0\end{smallmatrix}\right)\big\rangle\simeq D_6, & \text{if}\quad \mu=\frac{1}{2},\quad \nu=\frac{\sqrt{3}}{2},
\end{cases}
\end{gather*}
and $S\in\SSS$ acts on~$\iota\in\ZZ^2_r$
by
$S\cdot\iota=\det(S)S(\iota);$
\end{itemize}
\item[$2.$] $L^2_r(\lambda,\mu,\nu):=L_r(\lambda,\mu,\nu,0,0)$ for
\begin{itemize}
\item[$\circ$] $\lambda\in\pi+2\pi\NN$,
\item[$\circ$] $(\mu,\nu)\in\cF_+;$
\end{itemize}
\item[$3.$] $L^{2,+}_r(\lambda,\mu,\nu,\iota_1,\iota_2):= L_r\big(\lambda,\mu,\nu,\frac{\iota_1}{r}+{\rm i}\frac{\iota_2}{r},0\big)$, where $r$ is even, for
\begin{itemize}
\item[$\circ$] $\lambda\in\pi+2\pi\NN$,
\item[$\circ$] $(\mu,\nu)\in\cF_+$,
\item[$\circ$]
$(\iota_1,\iota_2)\in
\begin{cases}	
\{(1,0),(0,1),(1,1)\}, &\text{if}\quad
\mu^2+\nu^2>1,\quad \mu\neq\frac{1}{2 },
\\[0.5ex]
\{(1,0),(1,1)\}, &\text{if}\quad \mu^2+\nu^2>1,\quad \mu=\frac{1}{2},
\\[0.5ex]
\{(1,0),(1,1)\}, &\text{if}\quad \mu^2+\nu^2=1,\quad \mu\neq\frac{1}{2},\\
\{(1,1)\} ,&\text{if}\quad \mu=\frac{1}{2},\quad \nu=\frac{\sqrt{3}}{2}.
\end{cases}$
\end{itemize}
\end{enumerate}
\end{Theorem}
We remark that compared to~\cite{F}, we~use a slightly different system of representatives. For~instance, here $z_0$ is chosen such that the lattice is unshifted whereas $z_0=0$ in~\cite{F}.
\begin{Remark}\label{RMR}
There is the following generalisation of the classical four-dimensional oscillator group. Let~$H_n$ be the $(2n+1)$-dimensional Heisenberg group, which we identify with $\RR\times \CC^n$ (as a set). Let~us fix an element $\lambda\in\RR^n$. Then we can define the $(2n+2)$-dimensional oscillator group $\Osc_n(\lambda):=H_n\rtimes\RR$, where $t\in\RR$ acts on~$H_n$ by $t.(z,\xi)=(z,\exp(\diag({\rm i}t\lambda))(\xi))$. Medina and Revoy~\cite{MR} proved a criterion for the existence of lattices in~$\Osc_n(\lambda)$ in terms of~$\lambda$. In~\cite[p.~94]{MR}, they tried to classify lattices of~$\Osc_1$ (up to automorphisms of~$\Osc_1$). For every $r\in\NN_{>0}$, they found only a finite number of (isomorphism classes of) lattices $L$ such that $r(L)=r$, where~$r(L)$ was defined in~(\ref{Drlambda}), which is obviously wrong. Note that the map given on p.~92 of~\cite{MR} is not an automorphism of~$\Osc_n(\lambda)$. Therefore, in the proof of Theorem III, one cannot assume that $(0,0,t)$ belongs to $L$ without changing $L\cap H_n$. Thus, Theorem III is not correct, which is one reason for the wrong classification.
\end{Remark}

\section[The model $G$ of the oscillator group]{The model $\boldsymbol G$ of the oscillator group}
In the following, we~want to use a slightly different multiplication rule for the oscillator group, which will make our computations easier. We~use the well known fact that the Heisenberg group $H$ is isomorphic to the set $H(1)$ of elements $M(x,y,z)$ parametrised by $x,y,z\in \RR$ with group multiplication
\begin{gather*}
M(x,y,z)M(x',y',z')=M(x+x',y+y',z+z'+xy').
\end{gather*}
We define an action~${l}$ of~$\RR$ on~$H(1)$ by
\begin{gather*}
l(t)(M(x,y,z))=M\bigg(x \cos t\! -y \sin t, x \sin t\! + y\cos t, z \!+ \frac {xy}2 (\cos (2t)\! -1)\!+\frac {x^2\!-y^2}4\sin (2t)\!\bigg)
\end{gather*}
and consider the semi-direct product
\begin{gather}
G:=H(1)\rtimes_l \RR. \label{GM}
\end{gather}
{\samepage
The image of an element $t\in\RR$ under the identification of~$\RR$ with the second factor of~$G$ in~\eqref{GM} is denoted by $(t)$. It~is easy to check that
\begin{gather}
\phi\colon\quad \Osc_1\rightarrow G,\qquad (x+{\rm i}y,z,t)\mapsto M\bigg({-}y,x,z-\frac12 xy\bigg)(t)\label{Eiso}
\end{gather}
is an isomorphism.

}

The isomorphism $\phi$ maps $L_r(\lambda,\mu,\nu,\xi_0,z)$ to the lattice $\Gamma_r(\lambda,\mu,\nu,\xi_0,z)$ generated by
\begin{gather}
\gamma_1:= M\bigg({-}\sqrt\nu,-\frac{\mu}{\sqrt\nu},\frac{1}{2}\mu\bigg), \label{g1}
\\
\gamma_2:= M\bigg(0,\frac{1}{\sqrt\nu},0\bigg), \label{g2}
\\
\gamma_3:= M\bigg(0,0,\frac{1}{r}\bigg), \label{g3}
\\
\gamma_4:= M\bigg({-}\sqrt\nu y_0,\frac{1}{\sqrt\nu} x_0-\frac{\mu}{\sqrt\nu} y_0,z-\frac{1}{2}\big(x_0 y_0-\mu y_0^2\big)\bigg)\cdot(\lambda), \label{g4}
\end{gather}
where $\xi_0=x_0+{\rm i}y_0$.

The images under $\phi$ of the standard lattices in~$\Osc_1$ are
\begin{gather*}
\Gamma^1_{r_0}(r,\lambda,\mu,\nu,\iota_1,\iota_2) = \phi\big(L^1_{r_0}(r,\lambda,\mu,\nu,\iota_1,\iota_2)\big),\qquad\dots,\qquad
\Gamma^6_r(\lambda)=\phi\big(L^6_r(\lambda)\big).
\end{gather*}
They are called standard lattices in~$G$.

\begin{Definition} A lattice $L\subset \Osc_1$ is called straight if it is generated by a lattice in the Heisenberg group and an element of the centre of~$\Osc_1$. A lattice in~$G$ is called straight if its preimage in~$\Osc_1$ is straight. Similarly, a lattice in~$G$ is called unshifted or normalised if its preimage in~$\Osc_1$ has this property.
\end{Definition}
\begin{Definition} For a lattice $\Gamma\subset G$, we~put $\lambda(\Gamma):=\lambda\big(\phi^{-1}(\Gamma)\big)$. Furthermore, the type of~$\Gamma$ is defined to be the type of~$\phi^{-1}(\Gamma)$.
\end{Definition}
\begin{Remark}
An unshifted normalised lattice $\Gamma$ in~$G$ is straight if and only if it is isomorphic under inner automorphisms of~$G$ to a lattice $\Gamma_r(\lambda,\mu,\nu,0,0)$ for $\lambda\in2\pi\,\NN_{>0}$. Indeed, consider the preimage $L$ of~$\Gamma$ in~$\Osc_1$. After applying an inner automorphism we have~\eqref{choice2} and $\delta$ is an~element of the centre of~$\Osc_1$, thus $x_\delta=y_\delta=0$. Since $L$ is unshifted, $z_\delta\in\frac 1{s_0r}\ZZ$. Now we apply Lemma~\ref{pr:shift}.
\end{Remark}

\section{The right regular representation}

Let $L$ be a lattice in~$\Osc_1$.
Then $L$ acts on the left of~$\Osc_1$ and we can consider the quotient $L\backslash \Osc_1$.
Furthermore, $L$ acts by left translation on functions $\ph\colon \Osc_1\to\CC$. For $\gamma\in L$ this action is defined by
\begin{gather*}
\big(L_\gamma^* \ph\big)(g):=\ph(\gamma\cdot g).
\end{gather*}

Let $L^2(L\backslash \Osc_1)$ denote the Banach space completion of the normed space of all left $L$-invariant continuous functions $\varphi\colon \Osc_1\rightarrow \CC$ that are compactly supported mod $L$ with norm
\begin{gather*}
\lVert\varphi\rVert^2=\int_\bF|\varphi(g)|^2 \,{\rm d}(x,y,z,t),
\end{gather*}
where $\bF$ is a fundamental domain for the action of~$L$ on~$\Osc_1$. For integration we use the Lebesgue measure, which is left- and right-invariant with respect to multiplication in~$\Osc_1$.

The right regular representation~$\rho$ of~$\Osc_1$ on~$L^2(L\backslash \Osc_1)$ is a unitary representation given~by
\begin{gather*}
(\rho(g)(\varphi))(x)=\varphi(xg).
\end{gather*}
It is a classical result that $\big(\rho,L^2(L\backslash\Osc_1)\big)$ is a discrete direct sum of irreducible unitary representations of~$\Osc_1$ with finite multiplicities, see, e.g.,~\cite{Wo07}. We~already described the irreducible unitary representations of~$\Osc_1$ in Section~\ref{irrrep}. Our aim is to determine how often they occur in~$\big(\rho,L^2(L\backslash\Osc_1)\big)$ for a given lattice $L\subset \Osc_1$.
In the previous section we identified $\Osc_1$ with the group $G$ using the isomorphism $\phi\colon \Osc_1\rightarrow G$ defined by (\ref{Eiso}). Let~$L$ be a lattice in~$\Osc_1$ and let $\Gamma=\phi(L)$ denote the corresponding lattice in~$G$. We~identify the right regular representation~$\big(\rho, L^2(L\backslash\Osc_1)\big)$ of~$\Osc_1$ with the right regular representation~$\big(\rho_G, L^2(\Gamma\backslash G)\big)$ of~$G$ via the isomorphism $\phi^*\colon L^2(\Gamma\backslash G)\rightarrow L^2(L\backslash\Osc_1)$, which satisfies $\rho(g)\circ\phi^*=\phi^*\circ(\rho_G\circ\phi)(g)$ for all $g\in\Osc_1$. Then our problem now consists in the decomposition of~$\big(\rho_G, L^2(\Gamma\backslash G)\big)$ into irreducible subrepresentations of~$G$. It~would be natural to use the push-forwards of the irreducible representations of~$\Osc$ as models for the irreducible representations of~$G$. Let~us denote by $\big(\phi^{-1}\big)^*(\sigma,V)=\big(\sigma\circ\phi^{-1},V\big)$ such a push-forward of a representation~$(\sigma,V)$ of~$\Osc_1$. Then the irreducible unitary representations of~$G$ are $\big(\phi^{-1}\big)^*\cC_d$, $\big(\phi^{-1}\big)^*{\mathcal S}^\tau_a$ and $(\phi^{-1})^*{\mathcal F}_{c,d}$. However, in practice we will work with the slightly different representations ${\cC_d}_G$, ${{\mathcal S}^\tau_a}_G$ and ${{\mathcal F}_{c,d}}_G$, which are defined by the same formulas as $\cC_d$, ${\mathcal S}^\tau_a$ and ${\mathcal F}_{c,d}$ of~$\Osc_1$, but where now $X$, $Y$, $Z$, $T$ is the basis of the Lie algebra $\frak g$ of~$G$ that satisfies
\begin{gather*}
\exp(sX)=M(s,0,0),\quad \exp(sY)=M(0,s,0),\quad \exp(sZ)=M(0,0,s) ,\quad \exp(sT)=(s).
\end{gather*}
These representations are equivalent to $\big(\phi^{-1}\big)^*\cC_d$, $\big(\phi^{-1}\big)^*{\mathcal S}^\tau_a$ and $\big(\phi^{-1}\big)^*{\mathcal F}_{c,d}$, respectively. In~the following we simply write~$\rho$ instead of~$\rho_G$ and $\cC_d$, ${\mathcal S}^\tau_a$ and ${\mathcal F}_{c,d}$ instead of~${\cC_d}_G$, ${{\mathcal S}^\tau_a}_G$ and ${{\mathcal F}_{c,d}}_G$.

In the following sections we will mainly concentrate on the spectrum of quotients by standard lattices, only the final result will be formulated for arbitrary ones. This is justified by the following observation. An arbitrary lattice $L$ can be transformed into an unshifted and normalised one by an automorphism $F$ of~$\Osc_1$. Furthermore, the normalised and unshifted lattice $L'$ we get is isomorphic to a standard lattice $\tilde L$ under an inner isomorphism of~$\Osc_1$. Let~$\rho'$ and $\tilde \rho$ denote the corresponding right regular representation on~$L^2(L'\backslash \Osc_1)$ and $L^2\big(\tilde L\backslash \Osc_1\big)$, respectively. Then we have $\tilde\rho\cong \rho'$ and $F^*\tilde \rho\cong\rho$
and we can apply~\eqref{Fiso} to obtain the spectrum of~$L\backslash \Osc_1$ from that of~$\tilde L\backslash \Osc_1$.

We begin the study of~$\rho$ by calculating the action of the basis elements $X$, $Y$, $Z$, $T$ of~$\fg$. In~order to simplify the notation, we~often write $\ph(x,y,z,t)$ instead of~$\ph(M(x,y,z)\cdot(t))$ for $\ph\in L^2(\Gamma\backslash G)$.

\begin{Lemma}\label{Lrep}
The right regular representation of~$\fg_{\mathbb C}$ on~$L^2(\Gamma\backslash G)$ is given by
\begin{gather}
\rho_*(Z)=\partial_z \label{EZ},
\\
\rho_*(X+{\rm i}Y)={\rm e}^{-{\rm i}t}\partial_x +{\rm i}{\rm e}^{-{\rm i}t} \partial_y +{\rm i}x {\rm e}^{-{\rm i}t} \partial_z \label{E+},
\\
\rho_*(X-{\rm i}Y)={\rm e}^{{\rm i}t}\partial_x -{\rm i}{\rm e}^{{\rm i}t} \partial_y -{\rm i}x {\rm e}^{{\rm i}t} \partial_z \label{E-},
\\
\rho_*(T)=\partial_t. \label{ET}
\end{gather}
\end{Lemma}
\begin{proof}
We compute
\begin{align*}
(X\varphi)(x,y,z,t)&=\frac{\rm d}{{\rm d} s}\varphi(M(x,y,z)t\exp(sX))\big|_{s=0}
=\frac{\rm d}{{\rm d} s}\varphi(M(x,y,z)t M(s,0,0))\big|_{s=0}
\\
&=\frac{\rm d}{{\rm d} s}\varphi\bigg(M(x,y,z)M\bigg(s\cos t,s\sin t, \frac{1}{4}s^2\sin(2t)\bigg)t\bigg)\bigg|_{s=0}
\\
&=\frac{\rm d}{{\rm d} s}\varphi\bigg(x+s\cos t,y+s\sin t,z+ \frac{1}{4}s^2\sin(2t)+sx\sin t,t\bigg)\bigg|_{s=0}
\\
&=\cos t\,\partial_x \varphi(x,y,z,t) +\sin t\,\partial_y\varphi(x,y,z,t) +x\sin t\, \partial_z \varphi(x,y,z,t).
\end{align*}

Analogously,
\begin{gather*}
(Y\varphi)(x,y,z,t)
=-\sin t\,\partial_x\varphi(x,y,z,t) +\cos t\,\partial_y\varphi(x,y,z,t) +x\cos t\,\partial_z\varphi(x,y,z,t).
\end{gather*}
This gives (\ref{E+}) and (\ref{E-}).
Moreover
\begin{gather*}
(Z\varphi)(x,y,z,t)
=\frac{\rm d}{{\rm d} s}\varphi(x,y,z+s,t)\big|_{s=0}=\partial_z\varphi(x,y,z,t)
\end{gather*}
and
\begin{gather*}
(T\varphi)(x,y,z,t)
=\frac{\rm d}{{\rm d} s}\varphi(x,y,z,t+s)\big|_{s=0}=\partial_t\varphi(x,y,z,t).
\tag*{\qed}
\end{gather*}
\renewcommand{\qed}{}
\end{proof}

Denote by $\cH_0$ the sum of all irreducible subrepresentations of~$L^2(\Gamma\backslash G)$ for which $Z$ acts trivially and let $\cH_1$ be the orthogonal complement of~$\cH_0$ in~$L^2(\Gamma\backslash G)$, thus
\begin{gather*}
L^2(\Gamma\backslash G)=\cH_0 \oplus \cH_1.
\end{gather*}
By (\ref{EZ}), $\cH_0$ consists of those functions in~$L^2(\Gamma\backslash G)$ that do not depend on~$z$.

\section{Spectra of quotients by straight lattices}
We consider a normalised, unshifted and straight lattice $\Gamma:=\Gamma_r(\lambda,\mu,\nu,0,0)$, where $\lambda=2\pi \kappa$ for $\kappa\in\NN_{>0}$ and $\mu,\nu\in\RR$, $\nu>0$. This lattice is generated by $\gamma_1$, $\gamma_2$, $\gamma_3$ as defined in (\ref{g1})--(\ref{g3}) and by $\gamma_4=(2\pi\kappa)$.
\begin{Lemma}
Every function~$\ph\in L^2(\Gamma\backslash G)$ considered as a $\Gamma$-invariant function on~$G$ has a~Fourier expansion
\begin{gather}\label{EF}
\varphi(x,y,z,t)=\sum_{k,m,n\in\Z}\varphi_{k,m,n}(x){\rm e}^{2\pi {\rm i}\sqrt\nu ky}{\rm e}^{2\pi {\rm i}rmz}{\rm e}^{{\rm i} \frac n\kappa t},
\end{gather}
where the functions $\ph_{k,m,n}$ satisfy
\begin{gather}\label{eq:varphikmn}
\varphi_{k,m,n}(x)=\varphi_{k+rm,m,n}\big(x-\sqrt\nu\big) {\rm e}^{-\pi {\rm i}\mu (2k+rm)}
\end{gather}
for all $x\in\RR$ and, more generally,
\begin{gather}\label{eq:varphikmn2}
\varphi_{k,m,n}(x)=\varphi_{k+jrm,m,n}\big(x-j\sqrt\nu\big) {\rm e}^{-2\pi {\rm i}\mu jk} {\rm e}^{-\pi {\rm i}rmj^2\mu}
\end{gather}
for all $j\in\ZZ$.
Moreover,
\begin{gather}\label{EL2}
\|\varphi\|^2 =\frac{2\pi \kappa}{r\sqrt\nu}\sum_{k,m,n\in\Z}\int_0^{\sqrt\nu}|\varphi_{k,m,n}(x)|^2\,{\rm d}x.
\end{gather}
\end{Lemma}

\begin{proof}
We have
\begin{gather*}
\gamma_2^{k} \,\gamma_3^{m}\gamma_4^{n} =M\bigg(0,\frac{k}{\sqrt\nu},\frac{m}{r}\bigg)(2\pi \kappa n)
\end{gather*}
for all $k,m,n\in\ZZ$.
Furthermore,
\begin{align}
M\bigg(0,\frac{k}{\sqrt\nu},\frac{m}{r}\bigg)\,(2\pi \kappa n)\, M(x,y,z)\,(t)
&=M\bigg(0,\frac{k}{\sqrt\nu},\frac{m}{r}\bigg)M(x,y,z)(t+2\pi\kappa n)\nonumber
\\
&=M\bigg(x,y+\frac{k}{\sqrt\nu},z+\frac{m}{r}\bigg)(t+2\pi \kappa n)\label{eq:2PiInv2}.
\end{align}

By equation~\eqref{eq:2PiInv2}, a function~$\varphi\colon \Gamma\backslash G\rightarrow\RR$ considered as a $\Gamma$-invariant function on~$G$ satisfies
\begin{gather*}
\varphi(x,y,z,t)=\varphi\bigg(x,y+\frac{k}{\sqrt\nu},z+\frac{m}r,t+2\pi \kappa n\bigg).
\end{gather*}
In particular, $\varphi$ is periodic in~$y$, $z$ and $t$ with periodicity $\frac{1}{\sqrt\nu}$, $\frac{1}{r}$ and $2\pi\kappa$.
Hence $\varphi$ has a Fourier expansion of the form~\eqref{EF}.
Since $\varphi$ is $\Gamma$-invariant, we~have
\begin{gather*}
\ph(x,y,z,t)=\ph(\gamma_1\cdot M(x,y,z)(t))=\ph\bigg(x-\sqrt\nu,y-\frac\mu{\sqrt\nu}, z+\frac12\mu-\sqrt\nu y,t\bigg),
\end{gather*}
hence
\begin{align*}
\sum_{k,m,n}&\varphi_{k,m,n}(x){\rm e}^{2\pi {\rm i}\sqrt\nu ky}{\rm e}^{2\pi {\rm i}rmz}{\rm e}^{{\rm i} \frac n \kappa t}\\
&=\sum_{k,m,n}\varphi_{k,m,n}\big(x-\sqrt\nu\big){\rm e}^{-\pi {\rm i}\mu (2k-rm)}{\rm e}^{2\pi {\rm i}\sqrt\nu (k-rm)y} {\rm e}^{2\pi {\rm i}rmz} {\rm e}^{{\rm i} \frac n\kappa t},
\end{align*}
which gives~\eqref{eq:varphikmn}. Then~\eqref{eq:varphikmn2} follows by induction.
The set
\begin{gather*}
\bF:=\big\{M(x,y,z)\cdot(t) \mid x\in \big[0,\sqrt\nu\big],\, y\in\big[0,1/\sqrt\nu\big],\, z\in[0,1/r],\, t\in[0,2\pi\kappa]\big\}
\end{gather*}
is a fundamental domain of the $\Gamma$-action on~$G$. Thus
\begin{gather*}
\int_{\Gamma\backslash G} |\varphi|^2 \,{\rm d}(x,y,z,t) =\int_{\bF} |\varphi|^2 \,{\rm d}(x,y,z,t)=\frac{2\pi \kappa}{r\sqrt\nu}\sum_{k,m,n}\int_0^{\sqrt\nu}|\varphi_{k,m,n}(x)|^2\,{\rm d}x,
\end{gather*}
which proves~\eqref{EL2}.
\end{proof}

Recall that we put
\begin{gather}\label{Tmunu2}
T_{\mu,\nu}=\begin{pmatrix} \sqrt\nu&\frac{\mu}{\sqrt\nu}\vspace{1mm}\\0&\frac 1{\sqrt\nu} \end{pmatrix}\!.
\end{gather}

\begin{Proposition} \label{PH0H1}We have $L^2(\Gamma\backslash G)=\cH_0 \oplus \cH_1$. The first summand is equivalent to
\begin{gather}
\cH_0\cong\bigoplus_{n\in\Z} \cC_{n/\lambda} \oplus \bigoplus _{\substack{(l,k)\in\Z^2\\(l,k)\not=(0,0)}} \bigoplus_{K=0}^{\kappa-1} {\mathcal S}_{a(l,k)}^{\tau(K)}, \label{PH0}
\end{gather}
for $a(l,k)= \big(\nu k^2 +\frac1\nu(-\mu k+l)^2\big)^{\frac12}=\big\|T_{\mu,\nu}^{-1}\cdot(l,k)^\top\big\|$ and $\tau(K)=K/\kappa$.
The second one is equivalent to
\begin{gather}
\cH_1\cong\bigoplus_{m\in \Z_{\not=0}} |m|r\cdot \bigoplus_{n\in\Z} \cF_{rm,\,n/\lambda}. \label{PH1}
\end{gather}
\end{Proposition}

\begin{Remark}\label{R73}
Note that the representations ${\mathcal S}_{a(l,k)}^{\tau(K)}$, which appear in equation~\eqref{PH0} are not pairwise non-isomorphic since $a(l,k)=a\big(\hat l,\hat k\big)$ does not imply $(l,k)=\big(\hat l,\hat k\big)$. For instance,
${\mathcal S}_{a(k,l)}^{\tau(K)}\cong {\mathcal S}_{a(-k,-l)}^{\tau(K)}$.
\end{Remark}

\begin{proof}[Proof of Proposition~\ref{PH0H1}.]
Assume $\varphi\in \cH_0$. Using equations~\eqref{EF} and~\eqref{eq:varphikmn}, we~obtain
\begin{gather*}
\varphi(x,y,z,t)
=\sum_{k,n}\varphi_{k,n}(x){\rm e}^{2\pi {\rm i}\sqrt\nu ky}{\rm e}^{{\rm i} \frac{n}{\kappa}t},
\end{gather*}
where $\varphi_{k,n}\colon \RR\rightarrow\RR$ satisfies
\begin{gather*}
\varphi_{k,n}(x)=\varphi_{k,n}\big(x+\sqrt\nu\big) {\rm e}^{2\pi {\rm i}\mu k}.
\end{gather*}
The latter equation implies that $\varphi_{k,n}\sigma_k^{-1}$ is periodic with periodicity $\sqrt\nu$ for
\begin{gather*}
\sigma_k\colon\quad \RR\longrightarrow\CC,\qquad
\sigma_k(x)={\rm e}^{-2\pi {\rm i}\mu xk\frac{1}{\sqrt\nu}}.
\end{gather*}
Consequently,
\begin{gather*}
\cH_0=\cSpan\left\{\phi_{l,n}^k(x,y,z,t):=\sigma_k(x){\rm e}^{2\pi {\rm i} \frac l{\sqrt\nu}x} {\rm e}^{2\pi {\rm i}\sqrt\nu ky} {\rm e}^{{\rm i} \frac n\kappa t}\mid k,l,n\in\ZZ\right\}.
\end{gather*}
An easy computation shows that
\begin{gather*}
\phi_{l,n}^k(x,y,z,t)={\rm e}^{2\pi {\rm i}(x,y)T_{\mu,\nu}^{-1}(l,k)^\top} {\rm e}^{{\rm i}\,\frac n\kappa t}.
\end{gather*}
Using Lemma~\ref{Lrep}, we~compute
\begin{gather*}
\rho_*(X+{\rm i}Y)\big(\phi_{l,n}^k\big)= {2\pi {\rm i} \bigg(\frac 1{\sqrt\nu}(-\mu k+l)+{\rm i}\sqrt \nu k \bigg) \phi_{l,n-\kappa}^k}\,,
\\
\rho_*(X-{\rm i}Y)\big(\phi_{l,n}^k\big)= {2\pi {\rm i} \bigg(\frac 1{\sqrt\nu}(-\mu k+l)-{\rm i}\sqrt \nu k \bigg) \phi_{l,n+\kappa}^k}\,,
\\
\rho_*(T)\big(\phi_{l,n}^k\big)={\rm i}\,\frac n\kappa\, \phi_{l,n}^k.
\end{gather*}
Hence, for each $n\in\ZZ$, $\phi_{0,n}^0$ spans a representation of type $\cC_{n/2\pi\kappa}$. Fix $k$, $l$, where now at least one of these numbers does not vanish. Furthermore, fix $K\in \NN$, $0\le K<\kappa$ and put
\begin{gather*}
\phi_j:=\phi^k_{l,j\kappa+K}.
\end{gather*}
Lemma~\ref{Lrep} implies
\begin{gather*}
\rho_*(X+{\rm i}Y)(\phi_j)={2\pi {\rm i} \bigg(\frac 1{\sqrt\nu}(-\mu k+l)+{\rm i}\sqrt \nu k \bigg) \, \phi_{j-1}},
\\
\rho_*(X-{\rm i}Y)(\phi_j)={2\pi {\rm i} \bigg(\frac 1{\sqrt\nu}(-\mu k+l)-{\rm i}\sqrt \nu k \bigg) \, \phi_{j+1}},
\\
\rho_*(T)(\phi_j)={\rm i}\bigg(j+\frac K\kappa\bigg)\, \phi_j.
\end{gather*}
Consequently, the representation of~$G$ on~$\cSpan\{\phi_j\mid j\in\ZZ\}$ is equivalent to~${\mathcal S}_{a(k,l)}^{\tau(K)}$,
which pro\-ves~\eqref{PH0}.

Let us turn to $\cH_1$. For fixed $m\in\ZZ$ and $k\in\ZZ$, we~define
\begin{gather*}
{\mathcal V}_{k,m}:=\bigg\{ \ph\in L^2(\Gamma\backslash G)\,\bigg| \, \varphi(x,y,z,t)=\sum_{j,n\in\Z}\varphi_{k+jrm,m,n}(x){\rm e}^{2\pi {\rm i}\sqrt\nu (k+jrm)y}{\rm e}^{2\pi {\rm i}rmz}{\rm e}^{{\rm i} \frac{n}{\kappa}t}\bigg\}.
\end{gather*}
Obviously, ${\mathcal V}_{k,m}={\mathcal V}_{k+lmr,m}$ holds.
\begin{Lemma} \label{LL1m} We have
\begin{align}
\cH_1&=\bigoplus_{m\in \Z_{\not=0}} \cV_{0,m}\oplus\dots\oplus\cV_{|m|r-1,m}\label{sum}
\\
&\cong\bigoplus_{m\in \Z_{\not=0}} |m|r\cdot \big( L^2_{m}(\RR)\otimes L^2\big(S^1\big)\big), \label{eqk}
\end{align}
where we identify $S^1$ with $\RR/2\pi\kappa \ZZ$ and endow $L^2_{m}(\RR):=L^2(\RR)$ with an action of~$G$ determi\-ned~by
\begin{gather}
Z(\ph_1\otimes\ph_2)=2\pi {\rm i} rm\,\ph_1\otimes\ph_2, \label{EmZ}
\\
(X+{\rm i}Y)(\ph_1\otimes\ph_2)(x,t)= ( \ph_1'(x)-2\pi rm x \ph_1(x)) \big({\rm e}^{-{\rm i}t} \ph_2(t)\big), \label{Em+}
\\
(X-{\rm i}Y)(\ph_1\otimes\ph_2)(x,t)= ( \ph_1'(x)+2\pi rm x \ph_1(x)) \big({\rm e}^{{\rm i}t} \ph_2(t)\big),\label{Em-}
\\
T(\ph_1\otimes\ph_2)=\ph_1\otimes \ph_2'.\label{EmT}
\end{gather}
The equivalence of representations $L^2_{m}(\RR)\otimes L^2\big(S^1\big) \cong {\mathcal V}_{k,m}$ in~\eqref{eqk} is given by
\begin{align*}
\Phi_k\colon\ L^2_{m}(\RR)\otimes L^2\big(S^1\big) &\longrightarrow{\mathcal V}_{k,m},
\\
\ph_1\otimes\ph_2&\longmapsto \sum_{j\in \frac k{rm}+\Z}
\varphi_1\big(x+j\sqrt\nu\big){\rm e}^{\pi {\rm i}\mu rmj^2}{\rm e}^{2\pi {\rm i}rmj\sqrt\nu y}{\rm e}^{2\pi {\rm i}rmz}\ph_2(t).
\end{align*}
\end{Lemma}

\begin{proof}
Assume $\varphi\in \cH_1$. Then
\begin{gather*}
\varphi(x,y,z,t)=\sum_{m\in\Z_{\not=0}}\sum_{k,n\in\Z}\varphi_{k,m,n}(x){\rm e}^{2\pi {\rm i}\sqrt\nu ky}{\rm e}^{2\pi {\rm i}rmz}{\rm e}^{{\rm i}\,\frac{n}{\kappa}t},
\end{gather*}
where the functions $\varphi_{k,m,n}\colon \RR\rightarrow\RR$ satisfy (\ref{eq:varphikmn}). Thus $\varphi$ is uniquely determined by
$\varphi_{k,m,n}$ for $k\in\{0,\dots,|m|r-1\}$, $m\neq 0$ and $n\in\ZZ$. In~particular,~\eqref{sum} holds.

By equation~\eqref{eq:varphikmn2}, an element $\ph$ of~$\cV_{k,m}$ satisfies
\begin{align*}
\varphi(x,y,z,t)&=\sum_{j,n\in\Z}\varphi_{k+jrm,m,n}(x){\rm e}^{2\pi {\rm i}\sqrt\nu (k+jrm)y}{\rm e}^{2\pi {\rm i}rmz}{\rm e}^{{\rm i} \frac{n}{\kappa}t}
\\
&={\rm e}^{-\pi {\rm i}\mu rm\left(\frac{k}{rm}\right)^2}\!\!\!\!\!\!\sum_{j\in\frac k{rm}+\Z}\sum_{n\in\Z}\! \varphi_{k,m,n}\bigg(\!x\!-\!\frac{k\sqrt\nu}{rm}\!+\!j\sqrt\nu\bigg){\rm e}^{\pi {\rm i}\mu rmj^2}{\rm e}^{2\pi {\rm i}rmj\sqrt\nu y}{\rm e}^{2\pi {\rm i}rmz}{\rm e}^{{\rm i} \frac{n}{\kappa}t}.
\end{align*}

 Hence, $\Phi_k$ is an isomorphism of vector spaces.
By~\eqref{EL2}, $(r\sqrt\nu)^{1/2}\Phi_k$ is an isometry.
Moreover,~$\Phi_k$ is $G$-equivariant. Indeed,~\eqref{EZ} and~\eqref{ET} show that $\rho_*(Z)\circ\Phi_k=\Phi_k\circ Z$ and $\rho_*(T)\circ\Phi_k=\Phi_k\circ T$. Furthermore, by~\eqref{E+},
\begin{gather*}
\rho_*(X+{\rm i}Y)\bigg(\sum_{j\in \frac k{rm}+\Z}
\varphi_1\big(x+j\sqrt\nu\big){\rm e}^{\pi {\rm i}\mu rmj^2}{\rm e}^{2\pi {\rm i}rmj\sqrt\nu y}{\rm e}^{2\pi {\rm i}rmz}\ph_2(t)\bigg)
\\ \qquad
{}={\rm e}^{-{\rm i}t}\sum_{j\in \frac k{rm}+\Z}\big(\varphi_1'\big(x+j\sqrt\nu\big)+{\rm i} 2\pi {\rm i}rmj\sqrt\nu\varphi_1\big(x+j\sqrt\nu\big)
+{\rm i}x2\pi {\rm i}rm\varphi_1\big(x+j\sqrt\nu\big)\big)
\\ \qquad \qquad
{}\times{\rm e}^{\pi {\rm i}\mu rmj^2}\,{\rm e}^{2\pi {\rm i}rmj\sqrt\nu y}\,{\rm e}^{2\pi {\rm i}rmz}\ph_2(t)
\\ \qquad
{}= {\rm e}^{-{\rm i}t}\sum_{j\in \frac k{rm}+\Z}\big(\varphi_1'\big(x+j\sqrt\nu\big)-2\pi rm\big(j\sqrt\nu+x\big)\varphi_1\big(x+j\sqrt\nu\big)\big)
\\ \qquad \qquad
{}\times {\rm e}^{\pi {\rm i}\mu rmj^2}\,{\rm e}^{2\pi {\rm i}rmj\sqrt\nu y}\,{\rm e}^{2\pi {\rm i}rmz}\ph_2(t),
\end{gather*}
 which equals
 \begin{gather*}
 \Phi_k\big((X+{\rm i}Y)(\ph_1\otimes\ph_2)\big)=\Phi_k\big((\ph_1'(x)-2\pi rmx\,\ph_1(x)){\rm e}^{-{\rm i}t}\ph_2(t)\big).
 \end{gather*}
 Analogously, one proves
$\rho_*(X-{\rm i}Y)\circ\Phi_k=\Phi_k\circ (X-{\rm i}Y)$ using~\eqref{E-}.
\end{proof}

Let $H_n$ denote the Hermite polynomial of degree $n$. Hermite polynomials are defined recursively by
\begin{gather}\label{Hermite}
H_0=1,\qquad H_{n+1}(x)=2xH_n(x)-H'_{n}(x),\qquad n=0,1,2,\dots
\end{gather}
and satisfy the equation
\begin{gather}\label{Hermite2}
H_n'=2nH_{n-1},\qquad n=1,2,\dots.
\end{gather}
We fix $m\not= 0$ and consider
\begin{gather*}
\psi_{m,q}(x):=(-1)^q\frac{(2r|m|)^{1/4}}{\sqrt 2^q\sqrt{q!}}H_q\big(\sqrt{2\pi r|m|}x\big) {\rm e}^{-\pi r|m|x^2}.
\end{gather*}
It is well known that, for fixed $m\not=0$, the set $\psi_{m,q}$, $q=0,1,2,\dots$ constitutes a complete orthonormal system in~$L^2(\RR)$. We~define
\begin{gather*}
\psi_{m,q}^{n}(x,t):=(2\pi\kappa)^{-1/2} \psi_{m,q}(x)\, {\rm e}^{{\rm i}\,\frac{n}{\kappa}t}
\end{gather*}
for $n\in\ZZ$ and obtain a complete orthonormal system $\{\psi_{m,q}^{n}\mid n\in \ZZ,\, q\in\NN\}$ in~$L^2_{m}(\RR)\otimes L^2\big(S^1\big)$.
\begin{Lemma}\label{LL2m} For fixed $m,n\in\ZZ$, we~put ${\mathcal X}^{n}_m:=\cSpan\big\{\psi^{n-q\kappa}_{m,q}\mid q\in\NN\big\}$. Then
\begin{gather*}
L^2_m(\RR)\otimes L^2\big(S^1\big)=\bigoplus_{n\in\Z} {\mathcal X}_m^{n}\cong \bigoplus_{n\in\Z} \cF_{rm,\,n/\lambda}.
\end{gather*}
The ground state in~${\mathcal X}_m^{n}$ equals $\psi^n_{m,0}=(2\pi\kappa)^{-1/2}(2r|m|)^{1/4}\, {\rm e}^{-\pi r|m|x^2}{\rm e}^{{\rm i}\,\frac{n}{\kappa}t}$.
\end{Lemma}
\begin{proof}
{\sloppy
Obviously, $L^2_m(\RR)\otimes L^2\big(S^1\big)=\bigoplus_{n\in\Z} {\mathcal X}_m^{n}$ as a vector space. Furthermore,~\eqref{EmZ} and~\eqref{EmT} imply
\begin{gather*}
Z\psi_{m,q}^{n}=2\pi {\rm i}rm\psi_{m,q}^{n},\qquad
T\psi_{m,q}^{n}={\rm i}\frac{n}{\kappa}\psi_{m,q}^{n}. 
\end{gather*}}
Using~\eqref{Hermite} and~\eqref{Hermite2}, we~obtain
\begin{gather*}
\psi'_{m,0}(x)+2\pi r|m|x\,\psi_{m,0}(x)=0, 
\\
\psi'_{m,q}(x)+2\pi r|m|x\,\psi_{m,q}(x)=-2\sqrt{q}\sqrt{\pi r|m|}\,\psi_{m,q-1}(x),
\\
\psi'_{m,q}(x)-2\pi r|m|x\,\psi_{m,q}(x)=2\sqrt{q+1}\sqrt{\pi r|m|}\,\psi_{m,q+1}(x).
\end{gather*}
These equations together with~\eqref{Em+} and~\eqref{Em-} give
\begin{gather*}
(X-{\rm i}Y)\psi_{m,0}^{n}=0,
\\
(X-{\rm i}Y)\psi_{m,q}^{n}=-2\sqrt{q}\sqrt{\pi rm}\,\psi_{m,q-1}^{n+\kappa},
\\
(X+{\rm i}Y)\psi_{m,q}^{n}=2\sqrt{q+1}\sqrt{\pi rm}\,\psi_{m,q+1}^{n-\kappa}
\end{gather*}
for $m>0$ and
\begin{gather*}
(X+{\rm i}Y)\psi_{m,0}^{n}=0,
\\
(X+{\rm i}Y)\psi_{m,q}^{n}=-2\sqrt{q}\sqrt{\pi r|m|}\,\psi_{m,q-1}^{n-\kappa},
\\
(X-{\rm i}Y)\psi_{m,q}^{n}=2\sqrt{q+1}\sqrt{\pi r|m|}\,\psi_{m,q+1}^{n+\kappa}
\end{gather*}
 for $m<0$. This shows that, for fixed $n\in\ZZ$, the space ${\mathcal X}^{n}_m:=\cSpan\big\{\psi^{n-q\kappa}_{m,q}\mid q\in\NN\big\}$ is invariant under $G$ and equivalent to $\cF_{rm,\,n/\lambda}$ as a $G$-representation.
\end{proof}

Now equation~\eqref{PH1} follows from Lemmas~\ref{LL1m} and~\ref{LL2m}. This finishes the proof of~Pro\-po\-si\-tion~\ref{PH0H1}.
\end{proof}

\section{Spectra of quotients by standard lattices}\label{S8}
\subsection{Strategy} \label{S81}

We will see that each standard lattice $\Gamma$ contains an unshifted normalised straight lattice $\Gamma'$, which is generated by $\gamma_1$, $\gamma_2$, $\gamma_3$ and a power of~$\gamma_4$. In~particular, $\lambda':=\lambda(\Gamma')=2\pi\kappa'$. We~can identify $L^2(\Gamma\backslash G)$ with the space of functions in~$L^2(\Gamma'\backslash G)$ that are invariant under $\gamma_4$. By~Proposition~\ref{PH0H1}, the representation~$L^2(\Gamma'\backslash G)$ decomposes as a direct sum $\cH_0'\oplus\cH_1'$, where $\cH_0'$ and $\cH_1'$ are defined as in~\eqref{PH0} and~\eqref{PH1} but with $\lambda'$ and $\kappa'$ instead of~$\lambda$ and $\kappa$. The subspaces $\cH_0'$ and $\cH_1'$ are invariant under $\gamma_4$. Moreover, each isotypic component of~$\cH_0'$ and $\cH_1'$ is invariant by~$\gamma_4$. Hence we have to determine the invariants of the action of~$\gamma_4$ on each of these isotypic components.

How do these isotypic components look like? Of course, the subrepresentations $\cC_{n/\lambda'}$ of~$\cH'_0$ are pairwise non-equivalent. We~have
\begin{gather*}
\cC_{n/\lambda'}\cong\cSpan\big\{ \phi^0_{0,n}= {\rm e}^{{\rm i} \frac n{\kappa'} t}\big\}.
\end{gather*}
The situation for the second part of~$\cH_0'$ is more complicated. As~already mentioned, the summands ${\mathcal S}_{a(l,k)}^{K/\kappa'}$ and ${\mathcal S}_{a( l',k')}^{K/\kappa'}$ can be equal even if $(l,k)\not=(l',k')$. For instance, it~will turn out that $a(l,k)=a(l',k')$ if $(l,k)$ and $(l',k')$ belong to the same orbit of the $\ZZ_q$-action on~$\ZZ^2$ defined by~the matrix $S_q$, where $q=\ord(\lambda(\Gamma))$. Thus the corresponding representations are equal. Let~$O$ denote the $\ZZ_q$-orbit of a non-zero element of~$\ZZ^2$. Then
\begin{gather*}
\bigoplus _{(l,k)\in O} {\mathcal S}_{a(l,k)}^{K/\kappa'}\cong
\cSpan\big\{ \phi_{l,n}^k={\rm e}^{2\pi {\rm i}(x,y)T_{\mu,\nu}^{-1}(l,k)^\top} {\rm e}^{{\rm i} \frac n{\kappa'} t}\mid (l,k)\in O,\, n\in K+\kappa'\ZZ\big\}
\end{gather*}
for $K\in \NN$, $0\le K<\kappa'$. Let~us now turn to the isotypic components of~$\cH_1'$. These are the summands $|m|r\cdot \cF_{rm,\,n/\lambda'}$ for $m\not=0$ and $n\in\ZZ$. By~\eqref{A+}, each $\cF_{rm,\,n/\lambda'}$ is spanned by a~ground state $\psi_0$ and by $A_+^j\psi_0$ for $j\in\NN_{>0}$. Let~$m\not=0$ and $n\in\ZZ$ be fixed and consider the subrepresentation of~$\cH_1'$ that is equivalent to $|m|r\cdot \cF_{rm,\,n/\lambda'}$. By~Lemma~\ref{LL2m}, the space ${\mathcal W}_{m,n}$ of~ground states in this subrepresentation is spanned by
\begin{align*}
\theta_{k,n}&:=(2\pi\kappa')^{\frac12}(2r|m|)^{-\frac14}\,\Phi_k(\psi^n_{m,0})
=\Phi_k\big({\rm e}^{-\pi r |m|x^2}{\rm e}^{{\rm i}nt/\kappa'}\big)
\\
&= \sum_{j\in \frac k{rm}+\Z}
{\rm e}^{-\pi r |m|(x+j\sqrt\nu)^2}{\rm e}^{\pi {\rm i}\mu rmj^2}{\rm e}^{2\pi {\rm i}rmj\sqrt\nu y}{\rm e}^{2\pi {\rm i}rmz}{\rm e}^{{\rm i}nt/\kappa'}
\end{align*}
for $k=0,\dots,r|m|-1$. Thus
\begin{gather*}
|m|r\cdot \cF_{rm,\,n/\lambda'}\cong {\mathcal W}_{m,n}\oplus A_+({\mathcal W}_{m,n})\oplus A_+^2({\mathcal W}_{m,n})\oplus\cdots
\end{gather*}
as a vector space.
Since $\Phi_{k+lmr}=\Phi_k$, we~may write
\begin{gather*}
{\mathcal W}_{m,n}=\cSpan\big\{\theta_{k,n}\mid k\in \ZZ_{r|m|}\big\}.
\end{gather*}
The action of~$\gamma_4$ commutes with the regular representation, hence the subspace of~$\gamma_4$-invariants in~$|m|r\cdot \cF_{rm,\,n/\lambda'}$ is isomorphic to
${\cW_0}\oplus A_+({\cW_0})\oplus A_+^2({\mathcal W}_0)\oplus\cdots$, where $\cW_0$ is the space of~$\gamma_4$-invariant elements in~$\cW_{m,n}$.

\subsection[The subrepresentation~$\cH_0$]{The subrepresentation~$\boldsymbol{\cH_0}$}
Recall that we defined integer valued matrices $S_q$ by~\eqref{Sq}. Since $(S_q)^q=I_2$ holds, $S_q$ defines a left action of~$\ZZ_q$ on~$\ZZ^2$ by $(l,k)^\top\mapsto S_q\cdot(l,k)^\top$. Each orbit of this action contains exactly $q$ elements except that one of~$(0,0)\in\ZZ^2$. We~denote the orbit space by $\ZZ^2/\ZZ_q$.

We will use again the matrix $T_{\mu,\nu}$, see~\eqref{Tmunu2}. Let~$\Gamma$ be a standard lattice and let $\lambda$, $\mu$, $\nu$ be the corresponding parameters. We~put $q=\ord(\lambda)$. Then
\begin{gather}\label{sq}
S_q=T_{\mu,\nu}{\rm e}^{{\rm i}\lambda}T_{\mu,\nu}^{-1}.
\end{gather}
In particular, the function~$(l,k)^\top\mapsto \big\|T_{\mu,\nu}^{-1}\cdot(l,k)^\top\big\|$ is constant on orbits of the $\ZZ_q$-action.

\begin{Proposition} \label{pr:H0}
Suppose that $\Gamma$ is a standard lattice and let $\lambda$, $\mu$, $\nu$ be the corresponding parameters from the list in Theorem~$\ref{th:classUptoInner}$. Assume that $q=\ord(\lambda)$ equals $2$, $3$, $4$ or~$6$ and define $\kappa$ by $\lambda=\lambda_0+2\pi\kappa$, where $\lambda_0\in\{\pi,\pi/2,\pi/3,2\pi/3\}$, thus $\kappa\in\NN$ if $\ord(\lambda)=2$ and $\kappa\in\ZZ$ otherwise.
Then the representation~$\cH_0$ is equivalent~to
\begin{gather*}
\cH_0\cong\bigoplus_{n\in \Z} \cC_{n/\lambda} \oplus \bigoplus _{\substack{(l,k)\in\Z^2/\Z_q\\(l,k)\not=(0,0)}} \bigoplus_{K=0}^{|1+q\kappa|-1} {\mathcal S}_{a(l,k)}^{\tau(K)},
\end{gather*}
where $\tau(K)=K/|1+q\kappa|$ and $a(l,k)=\big(\nu k^2 +\frac1\nu(-\mu k+l)^2\big)^{\frac12}=\big\|T_{\mu,\nu}^{-1}\cdot(l,k)^\top\big\|$.
\end{Proposition}
\begin{proof} By Remark~\ref{z0}, we~have $l_4^q=(0,0,q\lambda)$. This implies that the sublattice of~$\Gamma$ spanned by $\gamma_1$, $\gamma_2$, $\gamma_3$ and $\gamma_4^q=(q\lambda)$ is straight. With $\kappa':=|1+q\kappa|$ we obtain~$\gamma_4^q=(\pm 2\pi\kappa')$. We~proceed as explained in Section~\ref{S81}. Since
\begin{gather*}
L_{\gamma_4}^*(\phi^0_{0,n})={\rm e}^{\pm2\pi {\rm i} \frac nq} \phi^0_{0,n},
\end{gather*}
the representation~$\cC_{n/\lambda'}$ consists of~$\gamma_4$-invariant elements if and only if $q|n$.

We want to show that $\gamma_4$ leaves invariant the subrepresentation~$\cH_O$ that is isomorphic to $\bigoplus _{(l,k)\in O} {\mathcal S}_{a(l,k)}^{K/\kappa'}$ for a fixed $\ZZ_q$-orbit $O\not=\{(0,0)\}$ and we want to determine the space of~$\gamma_4$-invariant elements in~$\cH_O$. Since $\gamma_4=M'\cdot(\lambda)$ for
\begin{gather*}
M'=M\bigg({-}\sqrt\nu y_0,\,\frac1{\sqrt\nu}(x_0-\mu y_0),\,\frac12 y_0(x_0-\mu y_0)\bigg)
\end{gather*}
and
\begin{gather*}
\phi_{l,n}^k={\rm e}^{2\pi {\rm i}(x,y)T_{\mu,\nu}^{-1}( l, k)^\top} {\rm e}^{{\rm i}\,\frac n{\kappa'} t}
\end{gather*}
for $(l,k)\in \ZZ^2\setminus \{(0,0)\}$ and $n\in \ZZ$, we~obtain~$L^*_{M'}\big(\phi^k_{l,n}\big)={\rm e}^{2\pi {\rm i}(kx_0-ly_0)}\phi^k_{l,n}$, thus
\begin{align*}
L_{\gamma_4}^*\big(\phi^k_{l,n}\big)(x,y,z,t) &= L_{(\lambda)}^*\big(L_{M'}^*\big(\phi^k_{l,n}\big)\big)(x,y,z,t)
=L_{(\lambda)}^*\big({\rm e}^{2\pi {\rm i}(kx_0-ly_0)}\phi^k_{l,n}\big)(x,y,z,t)
\\
&={\rm e}^{2\pi {\rm i}(kx_0-ly_0)} {\rm e}^{2\pi {\rm i}(x\cos\lambda-y\sin\lambda,x\sin\lambda+y\cos \lambda)T_{\mu,\nu}^{-1}(l,k)^\top}{\rm e}^{{\rm i} \frac n{\kappa'} (t+\lambda)}
\\
&={\rm e}^{2\pi {\rm i}(kx_0-ly_0)}{\rm e}^{2\pi {\rm i}(x,y)T_{\mu,\nu}^{-1}\left(\hat l,\hat k\right)^\top}{\rm e}^{{\rm i} \frac n{\kappa'} (t+\lambda)},
\end{align*}
where $\big(\hat l,\hat k\big)^\top=S^{-1}_q(l,k)^T$ by (\ref{sq}) for $q=\ord(\lambda)$. This gives
\begin{gather*}
L_{\gamma_4}^*\big(\phi^k_{l,n}\big)={\rm e}^{2\pi {\rm i}(kx_0-ly_0)}{\rm e}^{{\rm i}\,\frac n{\kappa'} \lambda}\phi_{\hat l,n}^{\hat k}
\end{gather*}
which shows that $L^*_{\gamma_4}(\cH_O)=\cH_O$. Furthermore, if $O$ is the orbit of~$(l,k)$, then the space of~$\gamma_4$-invariant functions in~$\cH_O$ is spanned by
\begin{gather*}
\hat \phi_{l,n}^k:=\sum_{j=0}^{q-1}\big(L^*_{\gamma_4}\big)^j\phi_{l,n}^k,\qquad
n\in K+\kappa'\ZZ.
\end{gather*}
In particular, the subrepresentation of~$\gamma_4$-invariant elements in~$\cH_O$ is equivalent to ${\mathcal S}_{a(l,k)}^{K/\kappa'}$.
\end{proof}

Instead of standard lattices $\Gamma^1_{r_0}(r,\lambda,\mu,\nu,\iota_1,\iota_2)= \Gamma_r\big(\lambda,\mu,\nu,\frac{\iota_1}{r}+{\rm i}\frac{\iota_2}{r},z_0\big)$, for which always $(\mu,\nu)$ is in~$\cF$, we~consider, more generally, lattices $\Gamma_r\big(\lambda,\mu,\nu,\frac{\iota_1}{r}+{\rm i}\frac{\iota_2}{r},z_0\big)$ of type $\Gamma^1_{r_0}$ for arbitrary $(\mu,\nu)\in\textbf{H}$.
\begin{Proposition} \label{2piH0}
Suppose $\Gamma=\Gamma_r(\lambda,\mu,\nu,\xi_0,z_0)$ with $\lambda=2\pi\kappa$, $\kappa\in\NN_{>0}$, $(\mu,\nu)\in\textbf{H}$, $\xi_0=\frac1r(\iota_1+{\rm i}\iota_2)$, $z_0=\frac{1}{2}\rem_2(\iota_1\iota_2r)$, $\iota_1,\iota_2\in\ZZ$.
The representation~$\cH_0$ is equivalent to
\begin{gather*}
\cH_0\cong\bigoplus_{n\in \Z} \cC_{n/\lambda} \oplus \bigoplus _{\substack{(l,k)\in\Z^2,\\(l,k)\not=(0,0)}} \bigoplus_{j=0}^{\kappa-1} {\mathcal S}_{a(l,k)}^{\tau(K_j(l,k))},
\end{gather*}
{\sloppy
where $a(l,k)=\big\|T_{\mu,\nu}^{-1}\cdot(l,k)^\top\big\|$ is defined as before,
$\tau(K):=\frac{K}{\kappa}\in\RR/\ZZ$ and $K_j(l,k):=$ \mbox{$j-\frac1{r}(\iota_1 k-\iota_2l)$}.

}
\end{Proposition}
\begin{proof} We define $s_0:=r/\gcd(\iota_1,\iota_2,r)$.
We have
\begin{gather*}
\gamma_4=M\bigg({-}\sqrt\nu\frac{\iota_2}{r},\,\frac{1}{r\sqrt\nu}(\iota_1-\mu\iota_2), \, z_0-\frac{1}{2r^2}\big(\iota_1\iota_2-\mu\iota_2^2\big)\bigg)(2\pi\kappa),
\end{gather*}
thus
\begin{gather*}
\gamma_4^l=M\bigg({-}l\sqrt\nu\frac{\iota_2}{r},\,\frac{l}{r\sqrt\nu}(\iota_1-\mu\iota_2),\, lz_0-\frac{l^2}{2r^2}\big(\iota_1\iota_2-\mu\iota_2^2\big)\bigg)(2l\pi\kappa).
\end{gather*}
In particular,
\begin{gather*}
\gamma_4^{s_0}=\gamma_1^{n_2}\gamma_2^{n_1}\gamma_3^{n_3}\,(2\pi\kappa'),
\end{gather*}
where $\kappa'=s_0\kappa$, $n_j=s_0\iota_j/r$ for $j=1,2$ and a suitable $n_3\in\ZZ$. In~particular, $\Gamma':=\la \gamma_1, \gamma_2, \gamma_3, \gamma_4^{s_0}\ra$ is a straight lattice. We~compute
\begin{gather*}
L_{\gamma_4}^* \phi^k_{l,n}= {\rm e}^{2\pi {\rm i}\left(x-\frac{\sqrt\nu}r\iota_2,y+\frac 1{r\sqrt\nu}(\iota_1-\mu \iota_2)\right)T_{\mu,\nu}^{-1}( l,k)^\top}{\rm e}^{{\rm i} \frac n{\kappa'} ( t+\lambda)}
= {\rm e}^{2\pi {\rm i} \frac 1r(\iota_1k-\iota_2l)} {\rm e}^{2\pi {\rm i} \frac n{s_0}} \phi^k_{l,n}.
\end{gather*}
Thus each $\phi^k_{l,n}$ is an eigenfunction of~$L_{\gamma_4}^*$. The corresponding eigenvalue equals one if and only~if
\begin{gather}\label{cond}
\frac1r(\iota_1k-\iota_2l)+\frac n{s_0}\in\ZZ.
\end{gather}
Now we consider $n=j\kappa'+K$, where $K$ runs through a complete set of representatives modulo $\kappa'=s_0\kappa$. Then~\eqref{cond} is equivalent to
\begin{gather*}
 K\in -\frac{s_0}r(\iota_1k-\iota_2l)+s_0\ZZ,
 \end{gather*}
thus to $K\in\big\{js_0-\frac{s_0}r(\iota_1k-\iota_2l)\mid j=0,\dots,\kappa-1\big\}$. Since we are working in the straight lattice with $\kappa'=s_0\kappa$, the parameter $\tau$ equals $\frac K{\kappa'}=\frac 1\kappa \, \frac K{s_0}$ and the assertion follows.
\end{proof}
\subsection[The subrepresentation~$\cH_1$]{The subrepresentation~$\boldsymbol{\cH_1}$}
\subsubsection[$\ord(\lambda)=1$]{$\boldsymbol{\ord(\lambda)=1}$}

\begin{Proposition}\label{pr:2piH1} Let $\Gamma=\Gamma_{r_0}^1(r,\lambda,\mu,\nu,\iota_1,\iota_2)$ be a standard lattice with $\lambda\in2\pi\NN_{>0}$. The representation~$\cH_1$ is equivalent to
\begin{gather*}
\cH_1\cong\bigoplus_{m\in\Z _{\not=0}}\bigoplus_{n\in \Z} r_0|m|\cdot {\mathcal F}_{rm,\frac{r_0n}{r\lambda}}.
\end{gather*}
\end{Proposition}
\begin{proof}
Consider, more generally, $\Gamma=\Gamma_r(\lambda,\mu,\nu,\xi_0,z_0)$, where $\lambda=2\pi\kappa$, $\kappa\in\NN_{>0}$, $\xi_0=\frac 1r(\iota_1+{\rm i}\iota_2)$ and $z_0=\frac12\rem_2(r\iota_1\iota_2)$ with $\gcd(\iota_1,\iota_2,r)=r_0$. We~define
\begin{gather*}
s_2:=\gcd(\iota_2,r),\qquad s:=s_2/r_0.
\end{gather*}
Recall from the proof of Proposition~\ref{2piH0} that
\begin{gather*}
\gamma_4=M\bigg({-}\sqrt\nu\frac{\iota_2}{r},\frac{1}{r\sqrt\nu}(\iota_1-\mu\iota_2), z_0-\frac{1}{2r^2}\big(\iota_1\iota_2-\mu\iota_2^2\big)\bigg)(2\pi\kappa),
\end{gather*}
and
\begin{gather*}
\gamma_4^{s_0}=\gamma_1^{n_2}\gamma_2^{n_1}\gamma_3^{n_3}\,(2\pi\kappa'),
\end{gather*}
where $\kappa'=s_0\kappa$, $n_j=s_0\iota_j/r$ for $j=1,2$ and a suitable $n_3\in\ZZ$. In~particular, $\Gamma':=\la \gamma_1, \gamma_2, \gamma_3,\gamma_4^{s_0}\ra$ is a straight lattice.

Let us first assume that $m>0$.
The action of~$\gamma_4$ on functions is given by
\begin{gather*}
L^*_{\gamma_4} \ph(x,y,z,t)=
\ph\bigg(x\!-\! \frac{\sqrt\nu}r\iota_2,\,y\!+\!\frac1{r\sqrt\nu}(\iota_1-\mu\iota_2),\,z\!+\!z_0- \frac{\sqrt\nu}r\iota_2y\!-\!\frac1{2r^2}\big(\iota_1\iota_2\!-\!\mu\iota_2^2\big),\,t\!+\!\lambda\bigg).
\end{gather*}
In particular, this yields
\begin{gather*}
L_{\gamma_4}^*\theta_{k,n}={\rm e}^{\pi {\rm i}(2k-m\iota_2)\frac{\iota_1}{r}}{\rm e}^{2\pi {\rm i} r m z_0}{\rm e}^{2\pi {\rm i} \frac{n}{s_0}} \theta_{k-m\iota_2,n},
\end{gather*}
for ground states $\theta_{k,n}$, which inductively gives
\begin{gather}\label{ELl}
\big(L_{\gamma_4}^*\big)^{l}\theta_{k,n}={\rm e}^{\pi {\rm i}(2kl-ml^2\iota_2)\frac{\iota_1}{r}}{\rm e}^{2\pi {\rm i} rml z_0}{\rm e}^{2\pi {\rm i} \frac{n}{s_0}l} \theta_{k-lm\iota_2,n}.
\end{gather}
If non-trivial $\gamma_4$-invariant ground states exist, then also the space $\hat\cW\subset {\mathcal W}_{m,n}$ of invariants of~$\big(L_{\gamma_4}^*\big)^{r/s_2}$ is non-trivial. By~\eqref{ELl}, all $\theta_{k,n}$ are eigenfunctions of~$\big(L_{\gamma_4}^*\big)^{r/s_2}$. More exactly,
\begin{gather*}
\big(L_{\gamma_4}^*\big)^{r/s_2}\theta_{k,n}={\rm e}^{\pi {\rm i}\left(2k\frac {\iota_1}{s_2}-m\frac r{s_2}\frac{\iota_2}{s_2}\iota_1\right)}{\rm e}^{2\pi {\rm i} mz_0\frac{r^2}{r_0s}}{\rm e}^{2\pi {\rm i} \frac{n}s} \theta_{k,n}={\rm e}^{\pi {\rm i}(a+b)} \theta_{k,n},
\end{gather*}
where
\begin{gather*}
a= -m \iota_1 \frac r{s_2} \frac{\iota_2}{s_2},\qquad
b= \frac1s \bigg(2k \frac{\iota_1}{r_0} +2n+2mz_0\frac{r^2}{r_0}\bigg).
\end{gather*}
Consequently, $\gamma_4$-invariant ground states can only exist if $a+b\in 2\ZZ$ for some $k\in\ZZ$. By~the definition of~$b$, this condition is equivalent to the existence of an integer $l$ such that
\begin{gather}\label{as}
 as +2mz_0\frac{r^2}{r_0}+2\bigg(k\frac {\iota_1}{r_0} +n\bigg)=2ls.
\end{gather}
Note that $a$ is an integer. Moreover, $as+2mz_0\frac{r^2}{r_0}=m\big({-}\frac{\iota_1}{r_0} \frac{\iota_2}{s_2} r+2z_0\frac{r^2}{r_0}\big)$ is even. Indeed, if $r$, $\iota_1$, $\iota_2$ are odd, then $z_0=1/2$ and both summands in the brackets are odd. Otherwise, $z_0=0$ and $\frac{\iota_1}{r_0} \frac{\iota_2}{s_2} r ={\iota_1} \frac{\iota_2}{s_2} \frac r{r_0}=\frac{\iota_1}{r_0} \iota_2 \frac r{s_2}$ is even. Hence,~\eqref{as} is equivalent to
\begin{gather}\label{Eab}
 k\frac {\iota_1}{r_0}\equiv -n-\frac12 as-mz_0\frac{r^2}{r_0} \ {\rm mod}\ s.
\end{gather}
Let $K\subset\ZZ_{rm}$ denote the set of solutions $k$ of~\eqref{Eab}. Then $\hat\cW=\cSpan\{\theta_{k,n}\mid k\in K\}$. Since $\iota_1/r_0$ and $s$ are relatively prime, equation~\eqref{Eab} has exactly one solution~$k$ mod $s$. Hence $K$ contains $rm/s$ elements. The space of~$\gamma_4$-invariant ground states is contained in~$\hat\cW$. It~is spanned by
\begin{gather*}
\hat \theta_{k,n}:=\sum_{j=0}^{\frac r{s_2}-1}\big(L_{\gamma_4}^*\big)^{j}\theta_{k,n}
\end{gather*}
for $k\in K$. The summands on the right hand side satisfy $\big(L_{\gamma_4}^*\big)^{j}\theta_{k,n}\in\Span\{\theta_{k-jm\iota_2,n}\}$. It follows that they are linearly independent since the order of~$m\iota_2$ in~$\ZZ_{rm}$ equals
\begin{gather*}
rm/\gcd(rm,m\iota_2)=r/s_2.
\end{gather*}
Hence the dimension of the space of~$\gamma_4$-invariant ground states equals $\# K \cdot \frac {s_2}r=m r_0$. This proves the assertion for $m>0$.

The automorphism $F:=\phi F_{S}\phi^{-1}\in \Aut(G)$ for $S=\diag(1,-1)$ maps $\Gamma_r(\lambda,\mu,\nu,\xi_0,z_0)$ to $\Gamma_r(\lambda,-\mu,\nu,-\bar\xi_0,z_0)$. Using~\eqref{Fiso}, we~see that $F^*(\cF_{c,d})=\cF_{-c,-d}$. Moreover, as we have seen, the dimension~$\bm(m,n)=\gcd(\iota_1,\iota_2,r) m$ for $m>0$ of the space of~$\gamma_4-$invariant ground states does only depend on~$r_0=\gcd(\iota_1,\iota_2,r)=\gcd(-\iota_1,\iota_2,r)$. Hence $\bm(m,n)=\bm(-m,-n)=r_0|m|$, which proves the assertion also for $m<0$.
\end{proof}

\subsubsection[$\ord(\lambda)=2$]{$\boldsymbol{\ord(\lambda)=2}$}

\begin{Proposition} \label{P1}
Let $\Gamma$ be one of the lattices $\Gamma_r^2(\lambda,\mu,\nu)$ or $\Gamma_r^{2,+}(\lambda,\mu,\nu,\iota_1,\iota_2)$ with $\lambda\in\pi+2\pi\NN$, where in the latter case we assume that $r$ is even.
Then $\cH_1$ is equivalent to
\begin{gather*}
\cH_1\cong\bigoplus_{m\in \Z\backslash\{0\}} \bigoplus_{n\in\Z} \bm(m,n)\cdot \cF_{rm,\,\frac n{2\lambda}}.
\end{gather*}
For $\Gamma=\Gamma_r^2(\lambda,\mu,\nu)$, the multiplicities are given by
\begin{gather*}
\bm(m,n) =\begin{cases}
\big[\frac{|m|r}2\big]+1,& \text{if\quad $n$ is even},
\\[.5ex]
\big[\frac{|m|r-1}2\big],& \text{if\quad $n$ is odd},
\end{cases}
\end{gather*}
and, for $\Gamma=\Gamma_r^{2,+}(\lambda,\mu,\nu,\iota_1,\iota_2)$, by
\begin{gather*}
\bm(m,n) =
\begin{cases}
 \frac12 r|m|, &\text{if\quad $m$ is odd},
 \\[0.5ex]
 \frac12 r|m| +1,&\text{if\quad $m$ and $n$ are even},
 \\[0.5ex]
 \frac12 r|m| -1, &\text{if\quad $m$ is even and $n$ is odd}.
 \end{cases}
 \end{gather*}
\end{Proposition}

\begin{proof}
To calculate all cases simultaneously, we~consider $\Gamma_r(\lambda,\mu,\nu,\xi_0,z_0)$ with $\lambda=\pi+2\pi\kappa$, $\xi_0=(\iota_1+{\rm i}\iota_2)/r$ and $z_0=0$.
The lattice $\Gamma':=\langle\gamma_1,\gamma_2,\gamma_3,(2\pi\kappa')\rangle$ with $\kappa'=1+2\kappa$ is a straight sublattice of~$\Gamma$, since $\gamma_4^2=(2\pi \kappa')$.

Let us first assume that $m>0$. Then the action of~$\gamma_4$ on~$L^2(\Gamma'\backslash G)$ is given by
\begin{gather*}
L^*_{\gamma_4} \ph(x,y,z,t)=\ph\bigg(\!{-}x-\sqrt\nu \frac{\iota_2}{r},-y+(\iota_1-\iota_2\mu )\frac 1{r\sqrt\nu },z-\frac12(\iota_1-\iota_2\mu)\frac{\iota_2}{r^2}+\sqrt\nu \frac{\iota_2}{r} y,t+\lambda\!\bigg).
\end{gather*}
Thus, the action of~$\gamma_4$ on~${\mathcal W}_{m,n}$ is given by
\begin{gather*}
L^*_{\gamma_4}\theta_{k,n}=(-1)^n{\rm e}^{2\pi {\rm i}\,\frac{k}{r}\iota_1}{\rm e}^{-\pi {\rm i} m\frac{\iota_1\iota_2}{r}}\theta_{\iota_2m-k,n}.
\end{gather*}

We define
\begin{gather*}
S_{\iota_1,\iota_2}(\theta_{k,n})={\rm e}^{2\pi\frac{k}{r}\iota_1}{\rm e}^{-\pi {\rm i} m\frac{\iota_1\iota_2}{r}}\theta_{\iota_2m-k,n}.
\end{gather*}

Now, the dimension of the space of~$\gamma_4$-invariant functions in~${\mathcal W}_{m,n}$ equals the multiplicity of the eigenvalue $(-1)^n$ of~$S_{\iota_1,\iota_2}$.
Moreover, the multiplicity $m_{\pm }$ of the eigenvalue $\pm 1$ of~$S_{\iota_1,\iota_2}$ equals
\begin{gather*}
m_{\pm }=\frac{1}{2}(rm\pm\tr S_{\iota_1,\iota_2}).
\end{gather*}

Thus, it~remains to compute
\begin{gather*}
\tr S_{\iota_1,\iota_2}=
\begin{cases}
 0, &\text{if\quad $\iota_2m$ is odd and $rm$ is even},
 \\
 1, &\text{if\quad $\iota_2m$ is even and $rm$ is odd},
 \\
 {\rm e}^{\pi {\rm i} m\iota_1}, &\text{if\quad $\iota_2m$ is odd and $rm$ is odd},
 \\
 1+{\rm e}^{\pi {\rm i} m\iota_1}, &\text{if\quad $\iota_2m$ is even and $rm$ is even}.
\end{cases}
\end{gather*}
Taking $\iota_1+{\rm i}\iota_2=0$ yields the assertion for $\Gamma^2_r(\lambda,\mu,\nu)$ and taking $r$ even and $\iota_1+{\rm i}\iota_2\in\{1/r,{\rm i}/r,1/r+{\rm i}/r\}$ yields the assertion for $\Gamma^{2,+}_r(\lambda,\mu,\nu,\xi_0)$.

The automorphism $F:=\phi F_{S}\phi^{-1}\in \Aut(G)$ for $S=\diag(1,-1)$ maps $\Gamma_r(\lambda,\mu,\nu,\xi_0,z_0)$ to $\Gamma_r(\lambda,-\mu,\nu,\bar\xi_0,z_0)$. Using~\eqref{Fiso}, we~see that $F^*(\cF_{c,d})=\cF_{-c,-d}$. Moreover, as we have seen, the multiplicities $\bm(m,n)$ for $m>0$ do not depend on~$\mu+{\rm i}\nu$. Hence $\bm(m,n)=\bm(-m,-n)$, which proves the assertion also for $m<0$.
\end{proof}

\subsubsection[$\ord(\lambda)=4$]{$\boldsymbol{\ord(\lambda)=4}$}

\begin{Proposition} \label{P12} Let $\Gamma$ be one of the lattices $\Gamma_r^4(\lambda)$ or $\Gamma_r^{4,+}(\lambda)$ with $\lambda\in\frac\pi2+2\pi\ZZ$, where in the latter case we assume that $r$ is even. Then $\cH_1$ is equivalent to
\begin{gather*}
\cH_1\cong\bigoplus_{m\in \Z\backslash\{0\}} \bigoplus_{n\in\Z} \bm(m,n)\cdot \cF_{rm,\,\frac n{4\lambda}}.
\end{gather*}
For $\Gamma=\Gamma_r^4(\lambda)$, the multiplicities are given by
\begin{gather*}
\bm(m,n) =\begin{cases}
\big[\frac14(r|m|-\rem_4(\sgn(\lambda m)\cdot n))\big]+1,& \text{if\quad $n$ is even},
\\[.5ex]
\big[\frac14(r|m|-\rem_4(\sgn(\lambda m)\cdot n))+\frac12\big],& \text{if\quad $n$ is odd},
\end{cases}
\end{gather*}
and, for $\Gamma=\Gamma_r^{4,+}(\lambda)$, by
\begin{gather*}
\bm(m,n) =\begin{cases}
\frac14 r|m|,& \text{if\quad $m$ is odd and $r\equiv 0\ {\rm mod}\ 4$},
\\[.5ex]
\big[\frac14(r|m|-\rem_4(\sgn(\lambda m)\cdot n))\big]+1,& \text{if\quad $m+n$ is even and}
\\
& (m \mbox{ is even or } r\equiv 2\ {\rm mod}\ 4),
\\[.5ex]
\big[\frac14(r|m|-\rem_4(\sgn(\lambda m)\cdot n))+\frac12\big],& \text{if\quad $m+n$ is odd and}
\\
& (m \mbox{ is even or } r\equiv 2\ {\rm mod}\ 4).
\end{cases}
\end{gather*}
\end{Proposition}

\begin{proof}
We want to study both cases simultaneously and put $a:=0$ if $\Gamma=\Gamma_r^4(\lambda)$ and $a=1$ if~$\Gamma=\Gamma_r^{4,+}(\lambda)$. Since here $\mu=0$ and $\nu=1$, we~have
\begin{gather*}
\gamma_4= M\bigg(0, \frac ar,-\frac{a}{4r^2}\bigg)\cdot(\lambda).
\end{gather*}
We define an $\RR$-linear map $S\colon \CC\to\CC$ by $S(\xi)={\rm i}\bar\xi$. The automorphism $F:=\phi F_{S}\phi^{-1}$ of~$G$ maps $\Gamma$ to $\Gamma$. Indeed, $F(\gamma_1)=\gamma_2$, $F(\gamma_2)=\gamma_1$, $F(\gamma_3)=\gamma_3^{-1}$, $F(\gamma_4)=\gamma_4^{-1}$. Using~\eqref{Fiso}, we~see that $F^*(\cF_{c,d})=\cF_{-c,-d}$. Hence $\bm(m,n)=\bm(-m,-n)$. Therefore, we~will assume that $m>0$ in the following computations.

We define $\kappa\in\ZZ$ by $\lambda=\frac\pi2+2\pi\kappa$ and we put $\kappa'=|4\kappa+1|$.
The lattice $\Gamma'=\la \gamma_1,\gamma_2,\gamma_3,\gamma_4^4\ra=\la \gamma_1,\gamma_2,\gamma_3,(2\pi \kappa')\ra$ is a sublattice of~$\la \gamma_1,\gamma_2,\gamma_3,\big(\frac\pi2\big)\ra$. In~particular, $\big(\frac\pi2\big)$ acts on~${\mathcal W}_{m,n}$.
\begin{Lemma}
For $m>0$, the action of~$(\frac{\pi}{2})$ on~${\mathcal W}_{m,n}$ is given by
\begin{gather*}
L^*_{\left(\frac{\pi}{2}\right)}\theta_{k,n}= \frac{1}{\sqrt{rm}}{\rm e}^{{\rm i}n \frac{\pi}{2\kappa'}}\sum_{\bar k\in\Z_{rm}}{\rm e}^{-2\pi {\rm i}\, \frac{k\bar k}{rm}}\theta_{\bar k,n}.
\end{gather*}
\end{Lemma}

\begin{proof}
We have
\begin{gather*}
 \bigg(\frac{\pi}{2}\bigg)M(x,y,z)(t)=M(-y,x,z-xy)\bigg(t+\frac{\pi}{2}\bigg),
 \end{gather*}
thus
\begin{gather*}
L^*_{\left(\frac{\pi}{2}\right)}\theta_{k,n}={\rm e}^{2\pi {\rm i}rmz}{\rm e}^{{\rm i}n\frac{t}{\kappa'}}{\rm e}^{{\rm i}n\frac{\pi}{2\kappa'}}
\sum_{j\in\frac{k}{rm}+\Z}{\rm e}^{-\pi rm(-y+j)^2}{\rm e}^{2\pi {\rm i}rmjx}{\rm e}^{-2\pi {\rm i}rmxy}.
\end{gather*}
This implies
\begin{align*}
\big\la L^*_{\left(\frac{\pi}{2}\right)}\theta_{k,n}, {\rm e}^{2\pi {\rm i}py}\big \ra & =
\int_0^1 L^*_{\left(\frac{\pi}{2}\right)}\theta_{k,n}{\rm e}^{-2\pi {\rm i}py} {\rm d}y
\\
&={\rm e}^{2\pi {\rm i}rmz}{\rm e}^{{\rm i}n\frac{t}{\kappa'}}{\rm e}^{{\rm i}n \frac{\pi}{2\kappa'}}
{\rm e}^{-\pi \frac{k^2}{rm}}{\rm e}^{2\pi {\rm i}xk}\int_{\R} {\rm e}^{-\pi rm\left(y^2+2y\left({\rm i}x-\frac k{rm}+{\rm i}\frac p{rm}\right)\right)}{\rm d}y
 \\
&={\rm e}^{2\pi {\rm i}rmz}{\rm e}^{{\rm i}n\frac{t}{\kappa'}}{\rm e}^{{\rm i}n\frac{\pi}{2\kappa'}}{\rm e}^{-\pi \frac{k^2}{rm}}{\rm e}^{2\pi {\rm i}xk} \\
&\quad \times\int_{\R} {\rm e}^{-\pi rm\left(y+{\rm i}x-\frac{k}{rm}+{\rm i}\frac{p}{rm}\right)^2}
{\rm e}^{\pi rm\left({\rm i}x-\frac{k}{rm}+{\rm i}\frac{p}{rm}\right)^2}{\rm d}y
\\
&=\frac{1}{\sqrt{rm}}{\rm e}^{2\pi {\rm i}rmz}{\rm e}^{{\rm i}n\frac{t}{\kappa'}}{\rm e}^{{\rm i}n\frac{\pi}{2\kappa'}}
{\rm e}^{-2\pi {\rm i}\,\frac{kp}{rm}}{\rm e}^{-\pi rm\left(x+\frac{p}{rm}\right)^2}.
\end{align*}

Hence,
\begin{align*}
L^*_{\left(\frac{\pi}{2}\right)} \theta_{k,n}&
=\sum_{p\in\Z} \big\la L^*_{\left(\frac{\pi}{2}\right)} \theta_{k,n}, {\rm e}^{2\pi {\rm i}py}\big\ra \,{\rm e}^{2\pi {\rm i}py}
\\
&=\frac1{\sqrt{rm}}{\rm e}^{{\rm i}n\frac t{\kappa'}}{\rm e}^{{\rm i}n\frac{\pi}{2\kappa'}}{\rm e}^{2\pi {\rm i}rmz}\sum_{\bar k\in\Z_{rm}}\sum_{J\in\Z} {\rm e}^{-2\pi {\rm i}\,\frac{k\bar k}{rm}}{\rm e}^{-\pi rm\left(x+\frac{\bar k}{rm}+J\right)^2}{\rm e}^{2\pi {\rm i} rm \left(\frac{\bar k}{rm}+J\right)y}
\\
&= \frac{1}{\sqrt{rm}}{\rm e}^{{\rm i}n\frac{\pi}{2\kappa'}}\sum_{\bar k\in\Z_{rm}}{\rm e}^{-2\pi {\rm i}\,\frac{k\bar k}{rm}}\theta_{\bar k,n},
\end{align*}
where we put $p=\bar k+J\cdot rm$. This gives the assertion.
\end{proof}

Recall that we assume $m>0$. We~define
\begin{gather*}
\gamma= M\bigg(0,\frac{a}{r},-\frac{a}{4r^2}\bigg)\bigg(\frac{\pi}{2}\bigg),\qquad a=0,1.
\end{gather*}
Then $\gamma^4=(2\pi)$. We~already proved that $\big(\frac \pi2\big)$ acts on~$\cW_{m,n}$. Thus $\gamma$ acts on~$\cW_{m,n}$ if $a=0$. This is also true for $a=1$ since then $\la \gamma_1,\gamma_2,\gamma_3,(2\pi \kappa')\ra$ is also a sublattice of~$\la \gamma_1,\gamma_2,\gamma_3,\gamma\ra$.
This action can be computed from the one in the case $a=0$ using
\begin{gather*}
L^*_{M\left(0,\frac1{r},-\frac1{4r^2}\right)}\theta_{k,n}={\rm e}^{-\frac{\pi {\rm i}}{2}\frac{m}{r}}{\rm e}^{2\pi {\rm i}\,\frac{k}{r}}\theta_{k,n}.
\end{gather*}
We obtain
\begin{gather*}
L^*_{\gamma}\theta_{k,n}={\rm e}^{{\rm i}n\frac{\pi}{2\kappa'}}S_a(\theta_{k,n}),
\end{gather*}
where
\begin{gather*}
S_a(\theta_{k,n}):=\frac{1}{\sqrt{rm}}{\rm e}^{-a\frac{\pi {\rm i}}{2}\frac{m}{r}}{\rm e}^{2a\pi {\rm i}\,\frac{k}{r}}\sum_{\bar k\in\Z_{rm}}{\rm e}^{-2\pi {\rm i}\,\frac{k\bar k}{rm}}\theta_{\bar k,n}.\end{gather*}

Using equation~\eqref{qgs} we compute
\begin{gather*}
\tr S_a=\frac{1}{\sqrt{rm}}{\rm e}^{-a\frac{\pi {\rm i}}{2}\frac{m}{r}}\sum_{k\in\Z_{rm}}{\rm e}^{\frac{\pi {\rm i}}{rm}(-2k^2+2amk)}
= \frac{1}{\sqrt{2}}{\rm e}^{-\frac{\pi {\rm i}}{4}}\big(1+{\rm e}^{\pi {\rm i}\,\frac{rm}{2}}{\rm e}^{a\pi {\rm i} m}\big).
\end{gather*}

Hence for $a=0$ we get
\begin{gather*}
\tr S_0=\begin{cases}
1-{\rm i},&\text{if}\quad rm=0\ {\rm mod}\ 4,
\\
1,&\text{if}\quad rm=1\ {\rm mod}\ 4,
\\
0,&\text{if}\quad rm=2\ {\rm mod}\ 4,
\\
-{\rm i},&\text{if}\quad rm=3\ {\rm mod}\ 4,
\end{cases}
\end{gather*}
and for $r$ even and $a=1$ we get
\begin{gather*}
\tr S_1=\begin{cases}
1-{\rm i},&\text{if\quad $\frac{rm}{2}+m$ is even},
\\[0ex]
0,&\text{if\quad $\frac{rm}{2}+m$ is odd}.
\end{cases}
\end{gather*}

Moreover, one easily checks
\begin{gather*}
S_a^2(\theta_{k,n})={\rm e}^{-a\pi {\rm i}\,\frac{m}{r}}{\rm e}^{2a\pi {\rm i}\,\frac{k}{r}}\theta_{am-k,n}.
\end{gather*}

Thus for $a=0$ we get
\begin{gather*}
\tr S_0^2=\begin{cases}
1,&\text{if\quad $rm$ is odd},
\\
2,&\text{if\quad $rm$ is even},
\end{cases}
\end{gather*}
and for $r$ even and $a=1$ we have
\begin{gather*}
\tr S_1^2=\begin{cases}
0,&\text{if\quad $m$ is odd},
\\
2,&\text{if\quad $m$ is even}.
\end{cases}
\end{gather*}

\begin{Lemma}\label{LD4}
Let $D$ be a diagonalisable linear automorphism of order $4$ of a vector space of dimension~$d$. Let~$t, t_1, t_2\in \RR$ be defined by
\begin{gather*}
\tr D= t_1+{\rm i} t_2,\qquad \tr D^2= t.
\end{gather*}
Then the multiplicities $m_1$, $m_{\rm i}$, $m_{-1}$ and $m_{-{\rm i}}$ of the eigenvalues $1$, ${\rm i}$, $-1$ and $-{\rm i}$ of~$D$ are equal~to
\begin{gather*}
m_{\pm1}= \frac1{4}(d\pm 2t_1+t),\quad
m_{\pm {\rm i}}= \frac1{4}(d\pm 2t_2-t).
\end{gather*}
\end{Lemma}
Using Lemma~\ref{LD4} for $D=S_0$, we~obtain
\begin{gather*}
m_1=\bigg[\frac{rm}{4}\bigg]+1, \qquad
m_{\pm {\rm i}}=\bigg[\frac{rm\mp 1}{4}\bigg], \qquad
m_{-1}=\bigg[\frac{rm-2}{4}\bigg]+1.
\end{gather*}
Similary, we~obtain for even $r$ and $D=S_1$
\begin{gather*}
m_1=\begin{cases}
\frac{rm}{4}+1,&\text{if\quad $m$ is even},
\\
\frac{rm}{4}+\frac{1}{2},&\text{if}\quad rm\equiv 2\ {\rm mod}\ 4,
\\
\frac{rm}{4},&\text{if\quad $m$ is odd\quad and\quad $r\equiv 0\ {\rm mod}\ 4$},
\end{cases}
\\[1ex]
m_{-1}=\begin{cases}
\frac{rm}{4},&\text{if}\quad rm\equiv 0\ {\rm mod}\ 4,
\\
\frac{rm}{4}-\frac{1}{2},&\text{if}\quad rm\equiv 2\ {\rm mod}\ 4,
\end{cases}
\\[1ex]
m_{\pm {\rm i}}=\frac{rm}2-m_{\pm 1}.
\end{gather*}

Recall that $a=0$ if $\Gamma=\Gamma_r^4(\lambda)$ and $a=1$ if $\Gamma=\Gamma_r^{4,+}(\lambda)$. In~both cases, we~have $\gamma_4=\gamma\cdot(2\pi\kappa)$. Thus
\begin{gather*}
L^*_{\gamma_4} \theta_{k,n}=L^*_{(2\pi\kappa)} L^*_{\gamma}\theta_{k,n}={\rm e}^{{\rm i}n\frac{2\pi\kappa}{\kappa'}}{\rm e}^{{\rm i}n\frac{\pi}{2\kappa'}}S_a(\theta_{k,n})={\rm e}^{{\rm i}\sgn(\lambda)n\frac{\pi}{2}}S_a(\theta_{k,n}).
\end{gather*}

In particular, the dimension of the space of~$\gamma_4$-invariant functions in~${\mathcal W}_{m,n}$ equals the multiplicity of the eigenvalue ${\rm e}^{-{\rm i}n\frac{\pi}{2}}$ of~$S_a$ for $\lambda>0$, and the multiplicity of the eigenvalue ${\rm e}^{{\rm i}n\frac{\pi}{2}}$ of~$S_a$ for $\lambda<0$. In~particular, we~have that $\bm(m,n)$ for $\lambda$ equals $\bm(m,-n)$ for $-\lambda$. Now, the previous computations for $S_a$ give the assertion for $m>0$. Furthermore, using $\textbf{m}(m,n)=\bm(-m,-n)$, we~obtain the assertion also for $m<0$.
\end{proof}

\subsubsection[$\ord(\lambda)\in\{3,6\}$]{$\boldsymbol{\ord(\lambda)\in\{3,6\}}$}

\begin{Proposition} \label{P13} If $\Gamma=\Gamma^6_r(\lambda)$ with $\lambda\in\frac{\pi}3+2\pi\ZZ$, then the representation~$\cH_1$ is equivalent to
\begin{gather*}
\cH_1\cong\bigoplus_{m\in \Z\backslash\{0\}} \bigoplus_{n\in\Z} \bm(m,n)\cdot \cF_{rm,\,\frac n{6\lambda}},
\end{gather*}
where
\begin{gather*}
\bm(m,n) =\begin{cases}
\big[\frac16\big({r|m|-\rem_6(\sgn(\lambda m)\cdot n)}\big)\big]+1,& \text{if\quad $rm+n$ is even},
\\[.5ex]
\big[\frac16\big(r|m|-\rem_6(\sgn(\lambda m)\cdot n)+3\big)\big],& \text{if\quad $rm+n$ is odd}.
\end{cases}
\end{gather*}
\end{Proposition}

\begin{Proposition} \label{Pord3} Let $\Gamma$ be one of the lattices $\Gamma_r^3(\lambda)$ or $\Gamma_r^{3,+}(\lambda)$ with $\lambda\in\frac{2\pi}3+2\pi\ZZ$, where in the latter case we assume that $r\equiv 0\ {\rm mod}\ 3$. Then $\cH_1$ is equivalent to
\begin{gather*}
\cH_1\cong\bigoplus_{m\in \Z\backslash\{0\}} \bigoplus_{n\in\Z} \bm(m,n)\cdot \cF_{rm,\,\frac n{3\lambda}}.
\end{gather*}
For $\Gamma=\Gamma_r^3(\lambda)$, the multiplicities are given by
\begin{gather*}
\bm(m,n) =\begin{cases}
\big[\frac13 \big({r|m|-\rem_3(\sgn(\lambda m)\cdot n)\big)}\big]+1,
& \text{if\quad $rm \equiv \sgn(\lambda)\cdot n$ {\rm mod}}\ 3,
\\[1ex]
\big[\frac13\big(r|m|-\rem_3(\sgn(\lambda m)\cdot n)+1\big)\big],
& \text{if\quad $rm\not\equiv \sgn(\lambda)\cdot n$ {\rm mod} $3$},
\end{cases}
\end{gather*}
and for $\Gamma=\Gamma_r^{3,+}(\lambda)$ by
\begin{gather*}
\bm(m,n) =
\begin{cases}
\frac{r|m|}{3},&\text{if\quad $m\not\equiv 0$ {\rm mod} $3$},
\\[1ex]
\big[\frac13\big(r|m|-\rem_3(\sgn(\lambda m)\cdot n)\big)\big]+1,
& \text{if\quad $m \equiv 0$ {\rm mod 3} and $n\equiv 0$ \rm{mod 3}},
\\[1ex]
\big[\frac13\big(r|m|-\rem_3(\sgn(\lambda m)\cdot n)+1\big)\big],
&\text{if\quad $m \equiv 0$ {\rm mod 3} and $n\not \equiv 0$ \rm{mod 3}}.
\end{cases}
\end{gather*}
\end{Proposition}

\begin{proof}[Proof of Propositions~\ref{P13} and~\ref{Pord3}.] Recall that $\mu=1/2$, $\nu=\sqrt 3/2$.
Since we want to study all cases simultaneously, we~put $\gamma= M\big(0,\frac{b}{2r\sqrt\nu},-\frac{b}{8r^2}\big)\big(\frac{\pi}{3}\big)$,
where $b=0$ if $r$ is even and $b=1$ if~$r$ is odd.
Then $\gamma^6=(2\pi)$. We~put $q:=\ord(\lambda)\in\{3,6\}$.

We define an $\RR$-linear map $S\colon \CC\to\CC$ by $S(\xi)={\rm e}^{-\pi {\rm i}/3}\bar \xi$ if $q=3$ and $S(\xi)={\rm e}^{2\pi {\rm i}/3}\bar\xi$ if $q=6$. The automorphism $F:=\phi F_S\phi^{-1}$ of~$G$ maps $\Gamma$ to $\Gamma$. Indeed,
\begin{gather*}
F(\gamma_1)=\gamma_2^{-1},\qquad F(\gamma_2)=\gamma_1^{-1},\qquad F(\gamma_3)=\gamma_3^{-1},\qquad F(\gamma_4)=\gamma_4^{-1},
\end{gather*}
if $q=3$ and
\begin{gather*}
F(\gamma_1)=\gamma_2,\qquad F(\gamma_2)=\gamma_1,\qquad F(\gamma_3)=\gamma_3^{-1},\qquad F(\gamma_4)=\gamma_4^{-1},
\end{gather*}
if $q=6$.
Using~\eqref{Fiso}, we~see that $F^*(\cF_{c,d})=\cF_{-c,-d}$. Hence $\bm(m,n)=\bm(-m,-n)$. Therefore, we~will assume that $m>0$ in the following computations.

We define $\kappa\in\ZZ$ by $\lambda=\pi/3+2\pi\kappa$ and $\lambda=2\pi/3+2\pi\kappa$, respectively. We~set $\kappa':=|1+q\kappa|$. Then $\Gamma'=\la \gamma_1,\gamma_2,\gamma_3,\gamma_4^q\ra=\la \gamma_1,\gamma_2,\gamma_3,(2\pi\kappa')\ra$. As~usual, we~consider the ground states $\theta_{k,n} \in \cW_{m,n}\subset L^2(\Gamma'\backslash G)$.
\begin{Lemma}
We have
\begin{gather*}
L^*_{\gamma} \theta_{k,n} = {\rm e}^{{\rm i}n\frac{\pi}{3\kappa'}} D_b(\theta_{k,n}),
\end{gather*}
where
\begin{gather*}
D_b\colon\quad {\mathcal W}_{m,n}\longrightarrow {\mathcal W}_{m,n}, \qquad D_b(\theta_{k,n})=\frac1{\sqrt{rm}} {\rm e}^{-\frac {\pi {\rm i}}{12}}{\rm e}^{\frac{b\pi {\rm i}}r\left(k-\frac m4\right)}\sum_{\bar k\in \Z_{rm}} {\rm e}^{\frac{\pi {\rm i}}{rm} (k^2-2k\bar k) }\theta_{\bar k,n}.
\end{gather*}
\end{Lemma}

\begin{proof}
Recall that $\mu=\frac{1}{2}$ and $\nu=\frac{\sqrt{3}}{2}$. It~is easy to check that the action of~$\gamma$ is given by
\begin{gather}\label{Ew3}
L^*_{\gamma} \ph(x,y,z,t)=\ph\bigg(\!x \mu -y \nu, x \nu + y\mu+\frac b{2r{\sqrt\nu}}, z -\nu^2 {xy} +\frac\nu4 \big(x^2\!-y^2\big)\!-\frac b{8r^2},t+\frac{\pi}{3}\bigg).
\end{gather}

By~\eqref{Ew3},
\begin{align*}
L^*_{\gamma} \theta_{k,n}&={\rm e}^{{\rm i}nt/\kappa'}{\rm e}^{{\rm i}n\frac{\pi}{3\kappa'}}{\rm e}^{2\pi {\rm i}rm\left(-\nu^2xy+\frac\nu4(x^2-y^2)-\frac b{8r^2}\right)}{\rm e}^{2\pi {\rm i}rmz}
\\
&\phantom{=}\times \sum_{j\in \frac k{rm}+\Z}{\rm e}^{\pi {\rm i} \mu rmj^2}{\rm e}^{-\pi rm(\mu x-\nu y+j\sqrt\nu)^2}{\rm e}^{2\pi {\rm i} rmj\sqrt\nu \left(\nu x+\mu y+\frac b{2r{\sqrt\nu}}\right)}.
\end{align*}
This yields
\begin{align*}
\big\la L^*_{\gamma} \theta_{k,n}, {\rm e}^{2\pi {\rm i}p\sqrt\nu y}\big\ra &= \int_0^{1/\sqrt\nu}L^*_{\gamma} \theta_{k,n} \, {\rm e}^{-2\pi {\rm i}p\sqrt\nu y}{\rm d}y
\\
&={\rm e}^{{\rm i}n\frac{\pi}{3\kappa'}}{\rm e}^{{\rm i}nt/\kappa'}{\rm e}^{2\pi {\rm i}rmz} {\rm e}^{\frac{b\pi {\rm i}}r\left(k-\frac m4\right)} {\rm e}^{-\frac{\pi rm}2\left(\mu-{\rm i}\nu\right)x^2} {\rm e}^{\frac{\pi {\rm i}}{rm}(k^2-2pk)}
\\
&\phantom{=}\times\int_{\R}{\rm e}^{-\pi rm\nu\left((\nu+{\rm i}\mu)y^2+2\big((-\mu+{\rm i}\nu)x+\frac{{\rm i}p}{rm\sqrt\nu}\big)y\right)}{\rm d}y.
\end{align*}
We put
\begin{gather*}
\alpha:={\rm e}^{\frac{\pi {\rm i}}{12}},\qquad
\beta:=\alpha^{-1}\bigg((-\mu+{\rm i}\nu)x+\frac{{\rm i}p}{rm\sqrt\nu}\bigg).
\end{gather*}
Then $\alpha^2=\nu +{\rm i}\mu$ and
\begin{gather*}
\la L^*_{\gamma} \theta_{k,n}, {\rm e}^{2\pi {\rm i}p\sqrt\nu y}\ra
\\ \qquad
{}= {\rm e}^{{\rm i}n\frac{\pi}{3\kappa'}}{\rm e}^{{\rm i}nt/\kappa'}{\rm e}^{2\pi {\rm i}rmz} {\rm e}^{\frac{b\pi {\rm i}}r\left(k-\frac m4\right)} {\rm e}^{-\frac{\pi rm}2\left(\mu-{\rm i}\nu\right)x^2} {\rm e}^{\frac{\pi {\rm i}}{rm}(k^2-2pk)}
\int_{\R}{\rm e}^{-\pi rm\nu (\alpha y+\beta)^2}{\rm d}y \, {\rm e}^{\pi rm\nu\beta^2}
\\ \qquad
= \frac1{\sqrt{rm\nu}} {\rm e}^{-\frac{\pi {\rm i}}{12}} {\rm e}^{{\rm i}n\frac{\pi}{3\kappa'}}{\rm e}^{{\rm i}n\frac t{\kappa'}}{\rm e}^{2\pi {\rm i}rmz} {\rm e}^{\frac{b\pi {\rm i}}r\left(k-\frac m4\right)} {\rm e}^{-\pi rm\big(x+\frac{p\sqrt\nu}{rm}\big)^2}{\rm e}^{\frac{\pi {\rm i}}{rm}\big(k^2-2kp+\frac{p^2}2\big)}.
\end{gather*}
Hence,
\begin{align*}
L^*_{\gamma} \theta_{k,n}&=\sqrt\nu \sum_{p\in\Z} \big\la L^*_{\gamma} \theta_{k,n}, {\rm e}^{2\pi {\rm i}p\sqrt\nu y}\big\ra \,{\rm e}^{2\pi {\rm i}p\sqrt\nu y}
\\
&=\frac1{\sqrt{rm}}{\rm e}^{{\rm i}nt/\kappa'}{\rm e}^{{\rm i}n\frac{\pi}{3\kappa'}}{\rm e}^{-\frac {\pi {\rm i}}{12}}\, {\rm e}^{2\pi {\rm i}rmz}{\rm e}^{\frac{b\pi {\rm i}}r\left(k-\frac m4\right)}{\rm e}^{\pi {\rm i}\,\frac{k^2}{rm}}
\\
&\phantom{=}\times \sum_{\bar k\in\Z_{rm}}\sum_{J\in\Z} {\rm e}^{-2\frac{\pi {\rm i}}{rm} k\bar k}{\rm e}^{\pi {\rm i}\mu rm\big(\frac{\bar k}{rm}+J\big)^2}{\rm e}^{-\pi rm\left(x+\left(\frac{\bar k}{rm}+J\right)\sqrt\nu\right)^2}{\rm e}^{2\pi {\rm i} rm (\frac{\bar k}{rm}+J) \sqrt\nu y},
\end{align*}
where we put $p=\bar k+J\cdot rm$. We~conclude
\begin{gather*}
L^*_{\gamma} \theta_{k,n} = \frac1{\sqrt{rm}}{\rm e}^{{\rm i}n\frac{\pi}{3\kappa'}}{\rm e}^{-\frac {\pi {\rm i}}{12}} {\rm e}^{\frac{b\pi {\rm i}}r\left(k-\frac m4\right)}
\sum_{\bar k\in \Z_{rm}} {\rm e}^{\frac{\pi {\rm i}}{rm} (k^2-2k\bar k)}\theta_{\bar k,n} = {\rm e}^{{\rm i}n\frac{\pi}{3\kappa'}} D_b(\theta_{k,n}). \tag*{\qed}
\end{gather*}
\renewcommand{\qed}{}
\end{proof}

Consider $\Gamma=\Gamma^6_r(\lambda)$ for $\lambda \in \frac\pi 3+2\pi\kappa$, $\kappa\in\ZZ$. Here $\kappa'=|6\kappa+1|$.
Then $\gamma_4$ is constructed from $\xi_0=\frac{b}{2r}$ and $z_0=-\frac{b}{8r^2}$, where $b=0$ if $r$ is even and $b=1$ if $r$ is odd. Hence $\gamma_4=\gamma\cdot(2\pi\kappa)$.

Then
\begin{gather*}
L^*_{\gamma_4} \theta_{k,n}=L^*_{(2\pi\kappa)} L^*_{\gamma}\theta_{k,n}={\rm e}^{{\rm i}n\frac{2\pi\kappa}{\kappa'}}{\rm e}^{{\rm i}n\frac{\pi}{3\kappa'}}D_b(\theta_{k,n})={\rm e}^{{\rm i}\sgn(\lambda)n\frac{\pi}{3}}D_b(\theta_{k,n}).
\end{gather*}

In particular, the dimension of the space of~$\gamma_4$-invariant functions in~${\mathcal W}_{m,n}$ is equal to the multiplicity of the eigenvalue ${\rm e}^{-{\rm i}n\pi/3}$ of~$D_b$ if $\lambda>0$ and to the multiplicity of the eigenvalue ${\rm e}^{{\rm i}n\pi/3}$ of~$D_b$ if $\lambda<0$.

Now, we~consider $\Gamma=\Gamma^3_r(\lambda)$ with $\lambda \in \frac{2\pi} 3+2\pi\kappa$, $\kappa\in\ZZ$. Then $\gamma_4$ is determined by $\xi_0=\frac{b}{r}+{\rm i}\frac{b}{2r}$, $z_0=-\frac{b}{8r^2}$, where again~$b=0$ if $r$ is even and $b=1$ if $r$ is odd. We~have $\gamma_4=\gamma^2\cdot(2\pi\kappa)$.
Here $\kappa':=|1+3\kappa|$, thus
\begin{gather*}
L^*_{\gamma_4} \theta_{k,n}=L^*_{(2\pi\kappa)} L^*_{\gamma^2}\theta_{k,n}={\rm e}^{{\rm i}n\frac{2\pi\kappa}{\kappa'}}{\rm e}^{{\rm i}n\frac{2\pi}{3\kappa'}}D_b^2(\theta_{k,n})={\rm e}^{{\rm i}\sgn(\lambda)n\frac{2}{3}\pi}D_b^2(\theta_{k,n}).
\end{gather*}

Consequently, the dimension of the space of~$\gamma_4$-invariant functions in~${\mathcal W}_{m,n}$ is equal to the multiplicity of the eigenvalue ${\rm e}^{-{\rm i}n\frac{2}{3}\pi}$ of~$D_b^2$ if $\lambda>0$ and to the multiplicity of the eigenvalue ${\rm e}^{{\rm i}n\frac{2}{3}\pi}$ of~$D^2_b$ if $\lambda<0$.

At last, we~consider $\Gamma=\Gamma^{3,+}_r(\lambda)$ with $\lambda \in \frac{2\pi} 3+2\pi\kappa$, $\kappa\in\ZZ$, and $r\in3\ZZ$. Again~$\kappa':=|1+3\kappa|$.
Then $\gamma_4$ is determined by
\begin{gather*}
\xi_0=\begin{cases}
-\frac{1}{r},&\text{if\quad $r$ is even},
\\[.5ex]
{\rm i}\frac{1}{2r},&\text{if\quad $r$ is odd},
\end{cases}\qquad
z_0=\begin{cases}
-\frac{1}{6r^2},&\text{if\quad $r$ is even},
\\[.5ex]
-\frac{1}{24r^2},&\text{if\quad $r$ is odd}.
\end{cases}
\end{gather*}
We put $m_b:=M\big(0,-\frac{1}{r\sqrt\nu},\frac{3b-1}{6r^2}\big)$, where again~$b=0$ if $r$ is even and $b=1$ if $r$ is odd. Then $\gamma_4=m_b\gamma^2(2\pi\kappa)$.
The action of~$m_b$ on~${\mathcal W}_{m,n}$ is given by
\begin{gather*}
\theta_{k,n}\longmapsto {\rm e}^{\frac{\pi}{3}{\rm i}\frac{m}{r}(3b-1)}{\rm e}^{-2\pi {\rm i}\,\frac{k}{r}}\theta_{k,n}.
\end{gather*}

Thus
\begin{align*}
L_{\gamma_{4}}^*(\theta_{k,n})=&L^*_{(2\pi\kappa)}L^*_{\gamma^2}L^*_{m_b}(\theta_{k,n})={\rm e}^{{\rm i}\sgn(\lambda)n\frac{2}{3}\pi}{\rm e}^{\frac{\pi}{3}{\rm i}\frac{m}{r}(3b-1)}{\rm e}^{-2\pi {\rm i}\,\frac{k}{r}}D^2_b(\theta_{k,n}).
\end{align*}
Hence the dimension of the space of~$\gamma_4$-invariant functions in~${\mathcal W}_{m,n}$ is equal to the multiplicity of the eigenvalue ${\rm e}^{-{\rm i}n\frac{2}{3}\pi}$ of
\begin{gather*}
\tilde A_{1,b}:={\rm e}^{\frac{\pi}{3}{\rm i}\frac{m}{r}(3b-1)}{\rm e}^{-2\pi {\rm i}\,\frac{k}{r}}D^2_b
\end{gather*}
if $\lambda>0$ and to the multiplicity of the eigenvalue ${\rm e}^{{\rm i}n\frac{2}{3}\pi}$ of~$\tilde A_{1,b}$ for $\lambda<0$.

Thus to prove Propositions~\ref{P13} and~\ref{Pord3} it remains to compute the multiplicity of the eigenvalues of~$D_b$, $D^2_b$ and $\tilde A_{1,b}$. Since $\bm(m,n)$ for $\lambda$ equals $\bm(m,-n)$ for $-\lambda$, it~again suffices to consider $m>0$. We~will use the following lemma.

\begin{Lemma}\label{LD6}
Let $D$ be a diagonalisable linear automorphism of order $6$ of a vector space of dimension~$d$. Let~$m_j$ denote the multiplicity of the eigenvalue ${\rm e}^{\frac{j \pi {\rm i}}3}$ for $j=-2,-1,\dots,3$. Furthermore, let $t, t_1,\dots,t_4\in \RR$ be defined by
\begin{gather*}
\tr D= t_1+{\rm i}\nu t_2,\qquad \tr D^2= t_3+{\rm i}\nu t_4,\qquad \tr D^3=t.
\end{gather*}
Then
\begin{gather*}
m_0 = \frac13 \bigg(t_1+t_3+\frac t2+\frac d2 \bigg),\\[0.4ex]
m_{\pm1} = \pm\frac 14(t_2+t_4) +\frac16 (t_1-t_3-t+d),\\[0.4ex]
m_{\pm2} = \pm\frac 14(t_2-t_4) +\frac16 (-t_1-t_3+t+d),\\[0.4ex]
m_3 = \frac13 \bigg({-}t_1+t_3-\frac t2+\frac d2 \bigg).
\end{gather*}
\end{Lemma}
Using equation~\eqref{qgs} we compute
\begin{gather*}
\tr D_b= \frac1{\sqrt{rm}} {\rm e}^{-\frac {\pi {\rm i}}{12}}\sum_{k\in\Z_{rm}} {\rm e}^{\frac{b\pi {\rm i}}r\left(k-\frac m4\right)} {\rm e}^{-\frac{\pi {\rm i}}{rm}k^2}
= \frac1{\sqrt{rm}} {\rm e}^{-\frac {\pi {\rm i}}{12}} {\rm e}^{-\frac{b\pi {\rm i}}{4r}m}S(-1,bm,rm)
=\mu-\nu {\rm i}.
\end{gather*}

Moreover, using~\eqref{qgs}, one easily checks
\begin{gather*}
D_b^2(\theta_{k,n})=\frac1{\sqrt{rm}} {\rm e}^{\frac {\pi {\rm i}}{12}}{\rm e}^{\frac{a\pi {\rm i}}r\left(2k-\frac {3m}4\right)}\sum_{\bar k\in \Z_{rm}} {\rm e}^{\frac{a\pi {\rm i}}r \bar k } {\rm e}^{-\frac{\pi {\rm i}}{rm} (\bar k^2+2k\bar k) }\theta_{\bar k,n}.
\end{gather*}

Instead of~$\tr D_b^2$, we~compute the trace of the more general map
\begin{gather*}
A_{a,b}\colon\ \theta_{k,n}\longmapsto {\rm e}^{-2\pi {\rm i} a\frac kr} D_b^2(\theta_{k,n})
\end{gather*}
and obtain
\begin{align}
\tr A_{a,b}&=\frac1{\sqrt{rm}} {\rm e}^{\frac {\pi {\rm i}}{12}}\sum_{k\in\Z_{rm}} {\rm e}^{2\pi {\rm i}(b-a)\frac{k}r} {\rm e}^{-\frac {3}4b\pi {\rm i}\,\frac mr} {\rm e}^{\frac{b\pi {\rm i}}rk} {\rm e}^{-\frac{\pi {\rm i}}{rm}\, 3k^2}\nonumber
\\
&=\frac1{\sqrt{rm}} {\rm e}^{\frac {\pi {\rm i}}{12}} {\rm e}^{-\frac{3b\pi {\rm i}}{4r}m}S(-3,(3b-2a)m,rm).
\label{trA}
\end{align}
For $a=0$, this yields
\begin{gather*}
\tr D_b^2 =\tr A_{0,b}=
\begin{cases}
$\phantom{$-$}$3/2-{\rm i}\nu, &\mbox{if}\quad rm\equiv 0\ {\rm mod}\ 3,
\\
-1/2-{\rm i}\nu, &\mbox{if}\quad rm\equiv 1\ {\rm mod}\ 3,
\\
$\phantom{$-$}$1/2+{\rm i}\nu, &\mbox{if}\quad rm\equiv 2\ {\rm mod}\ 3.
\end{cases}
\end{gather*}
We also have
\begin{gather*}
D_b^3(\theta_{k,n})={\rm e}^{\frac{b\pi {\rm i}}r(2k-m)} \theta_{bm-k,n},
\end{gather*}
hence
\begin{gather*}
t=\tr D_b^3=\begin{cases}
$\phantom{$-$}2$, &\text{if\quad $rm$ is even},
\\
-1,&\text{if\quad $rm$ is odd}.
\end{cases}
\end{gather*}
Furthermore, in the notation of Lemma~\ref{LD6},
\begin{gather*}
d=rm,\qquad t_1=1/2,\qquad t_2=-1,
\\
t_3= \begin{cases}
$\phantom{$-$}$3/2, &\mbox{if}\quad rm\equiv 0\ {\rm mod}\ 3,
\\
-1/2, &\mbox{if}\quad rm\equiv 1\ {\rm mod}\ 3,
\\
$\phantom{$-$}$1/2, &\mbox{if}\quad rm\equiv 2\ {\rm mod}\ 3,
\end{cases}\qquad
t_4=\begin{cases}
-1, &\mbox{if}\quad rm\equiv 0,1\ {\rm mod}\ 3,
\\
1, &\mbox{if}\quad rm\equiv 2\ {\rm mod}\ 3,
\end{cases}
\end{gather*}
and Lemma~\ref{LD6} gives the assertion for $\Gamma=\Gamma^6_r(\lambda)$.

Similarly, we~can compute the multiplicity $m_0$ and $m_\pm$ of the eigenvalue 1 and ${\rm e}^{\pm\frac{2}{3}\pi {\rm i}}$ of~$D_b^2$ by using the following lemma and obtain the assertion for $\Gamma=\Gamma^3_r(\lambda)$.
\begin{Lemma}\label{LD3}
Let $A$ be a diagonalisable linear automorphism of order $3$ of a vector space of~dimension~$d$. Furthermore, let $t_1,t_2\in \RR$ be defined by $\tr A= t_1+{\rm i}\nu t_2$.
Then the multiplicities $m_0$ and $m_\pm$ of the eigenvalues~$1$ and ${\rm e}^{\pm\frac{2}{3}\pi {\rm i}}$ of~$A$ are equal to
\begin{gather*}
m_0 = \frac{1}{3}(2t_1+d),\qquad
m_\pm = \frac{1}{3}(d-t_1)\pm\frac{1}{2}t_2.
\end{gather*}
\end{Lemma}

At last, by using equations~\eqref{trA} and~\eqref{qgs} we compute
\begin{gather*}
\tr \tilde A_{1,b}= \tr \big({\rm e}^{\frac{\pi {\rm i}m}{3r}(3b-1)} A_{1,b}\big)
= \frac1{\sqrt 3}{\rm e}^{-\frac{\pi {\rm i}}6}\big(1+2\Re {\rm e}^{\frac{2\pi {\rm i}}3 m}\big)
\\ \hphantom{\tr \tilde A_{1,b}}
{}=\begin{cases}
0, &\text{if}\quad m\equiv 1,2\ {\rm mod}\ 3,
\\[1ex]
{3}/{2}-\nu {\rm i}, &\text{if}\quad m\equiv 0\ {\rm mod}\ 3.
\end{cases}
\end{gather*}
Applying Lemma~\ref{LD3} to $\tilde A_{1,b}$ we obtain the assertion for $\Gamma=\Gamma^{3,+}_r(\lambda)$.
\end{proof}

\begin{Remark}\label{RBr}
In the introduction, we~mentioned Brezin's book~\cite{Br} on harmonic analysis on solvmanifolds, which contains an informal discussion of approaches to the decomposition of the right regular representation. In~Chapter 1, Brezin studies a series of four-dimensional examples. These examples are general semi-direct products $S_\sigma=H\rtimes _\sigma\RR$ of the Heisenberg group $H$ by the real line defined by a one-parameter subgroup $\sigma$ of~$\SL(2,\RR)\subsetneq \Aut(H)$, where $\sigma(1)\in\SL(2,\ZZ)$. He fixes a lattice $\Gamma$ in~$H$ which is generated by the set $\ZZ^2$ in the complement $\RR^2$ of the centre~$Z(H)$. Then he considers lattices $\Lambda_\sigma$ that are generated by $\Gamma$ and the subset $\ZZ$ of the real line. This requires an additional condition for $\sigma(1)$, which is called \emph{suitability} by Brezin. Let~us now restrict these considerations to the case where $S_\sigma$ is isomorphic to the oscillator group and let $\Lambda_\sigma\subset S_\sigma$ be a lattice in such a group. Then it is easy to see that $\Lambda_\sigma$ must be of type $L_1^1$, $L_1^{2}$ or~$L_1^{4}$. Moreover, it~satisfies a further condition corresponding to $\iota_1=\iota_2=0$ in our classification of standard lattices. Compared to the general case, this special choice of lattices simplifies the decomposition of the right regular representation.

Brezin uses a decomposition of the space of~$L^2$-functions on the quotient into two subspaces~$Y(0)$ and $\sum_{n\not=0}Y(n)$, which correspond to our spaces $\cH_0$ and $\cH_1$. However, his strategy to treat these subspaces is somewhat different from ours. He prefers an inductive approach. He identifies functions in~$Y(0)$ with functions on the three-dimensional quotient of~$\RR^2\rtimes_\sigma\RR$ by a~lattice generated by $\ZZ^2\subset\RR^2$ and $\ZZ\subset\RR$. For functions in~$\sum_{n\not=0}Y(n)$, he considers their restriction to the Heisenberg group $H$. This yields a representation of~$H\rtimes_\sigma \ZZ$. This representation is decomposed into irreducible ones. Then all extensions of these irreducible representations of~$H\rtimes_\sigma \ZZ$ to representations of~$S_\sigma$ are determined and those that actually occur in the right regular representation are identified.

In Brezin's book, the emphasis is more on a general discussion of approaches than on explicit results.
However, the case of lattices of type $L_1^{4}$ is studied in more detail and some computations of multiplicities are done. Unfortunately, no final result is formulated, but we think that, basically, these computations coincide with ours.
\end{Remark}

\subsection{The spectrum of the wave operator} \label{rmwave}
The (Lorentzian) Laplace--Beltrami operator, i.e., the wave operator, equals minus the Casimir operator of the right regular representation. On each irreducible subrepresentation~$\sigma$ it acts as multiplication by a scalar $x_\sigma$. This scalar is known for all irreducible unitary representations of~$\Osc_1$, see Section~\ref{irrrep}. Consequently, once the decomposition of~$\big(\rho,L^2(L\backslash\Osc_1)\big)$ into irreducible summands is known, the spectrum of the Laplace--Beltrami operator can be computed explicitly. More exactly, the spectrum is equal to the closure of the set of all numbers $x_\sigma\in\RR$, where $\sigma$ is an irreducible subrepresentation of~$\rho$. Let~us do this, for example, for straight lattices.

\begin{Proposition} Let $L\subset\Osc_1$ be a normalised, unshifted and straight lattice, that is, $L=L_r(2\pi\kappa,\mu,\nu,0,0)$ for $\kappa\in\NN_{>0}$ and $\mu,\nu\in\RR$, $\nu>0$. We~define
\begin{gather*}
\cA(\mu,\nu):=\bigg\{4\pi \bigg(\nu
k^2+\frac1\nu(-\mu k+l)^2\bigg)\, \bigg|\, (l,k)\in\ZZ^2\setminus\{(0,0)\}\bigg\}.
\end{gather*}
Then the spectrum of the wave operator of~$L\backslash \Osc_1$ equals
$4\pi\frac r\kappa\,\ZZ\cup \cA(\mu,\nu)$ if $\kappa$ is even and $2\pi\frac r\kappa\,\ZZ\cup \cA(\mu,\nu)$ if $\kappa$ is odd. In~particular, the spectrum is discrete.
\end{Proposition}

\begin{proof}The assertion follows from Proposition~\ref{PH0H1} and the formulas for $\Delta_{\cC_d}$, $\Delta_{{\mathcal S}_a^\tau}$, and $\Delta_{\cF_{c,d}}$ in~Section~\ref{irrrep}.
\end{proof}

\begin{Corollary}On each quotient of~$\Osc_1$ by a standard lattice, the spectrum of the wave operator is discrete.
\end{Corollary}
\begin{proof}
As we have seen, every quotient by a standard lattice is finitely covered by a quotient by a straight lattice.
\end{proof}

There exist lattices in~$\Osc_1$ such that the spectrum of the wave operator on the quotient space is not discrete.

\begin{Example}
 Let $L=L_r(2\pi\kappa,\mu,\nu,0,0)$ be a straight lattice, where $\kappa\in\NN_{>0}$ and $\mu,\nu\in\RR$, $\nu>0$. We~choose a real number $u$ such that $\tilde u:=2\pi \kappa r u$ is irrational and consider $L':=F_u(L)$, where $F_u$ is the automorphism defined in (\ref{EFu}). Then the right regular representation with respect to $L'$ equals $\big(F_u^{-1}\big)^*\big(\rho, L^2(L\backslash\Osc_1)\big)$. By~Proposition~\ref{PH0H1}, equation~(\ref{Fiso}), and the formula for $\Delta_{\cF_{c,d}}$ in Section~\ref{irrrep}, we~obtain that
\begin{gather*}
\cB_+(\tilde u):=\bigg\{ x_\sigma\,\bigg|\, \sigma=\cF_{c,d} \mbox{ for } c=rm,\, d=\frac n{2\pi\kappa}-urm,\, m\in\NN_{>0},\, n\in\ZZ\bigg\}
\\[.5ex] \hphantom{\cB_+(\tilde u)\,}
{}=\bigg\{ 4\pi\frac r\kappa\bigg (n+\frac \kappa 2-\tilde u m\bigg)m\,\bigg|\, m\in\NN_{>0},\, n\in\ZZ\bigg\}
\end{gather*}
is contained in the spectrum of the wave operator. Let~us assume that $\kappa$ is even. Then $n':=n+\frac \kappa 2$ is an integer. By~Dirichlet's approximation theorem, the set
$\{(n'-\tilde u m)m\,|\,m\in\NN_{>0}$, $n'\in\ZZ\}$ contains infinitely many numbers in~$[-1,1]$ if $\tilde u$ is irrational. Hence, $\cB_+(\tilde u)$ contains an accumulation point in~$\RR$.
\end{Example}

\subsection{Generalised quadratic Gauss sums}
We used several times the following reciprocity formula for generalised quadratic Gauss sums~\cite{Si60}, see also~\cite{BE}. Let~$a$, $b$, $c$ be integers with $ac\not=0$ such that $ac+b$ is even. Then
\begin{gather}\label{qgs}
S(a,b,c):=\sum_{r=0}^{|c|-1} {\rm e}^{\pi {\rm i} (a r^2+br)/c}=|c/a|^{\frac12}{\rm e}^{\pi {\rm i}(|ac|-b^2)/(4ac)}\sum_{r=0}^{|a|-1} {\rm e}^{-\pi {\rm i}(cr^2+br)/a}.
\end{gather}

\section{Summary}\label{S9}
\looseness=1
Suppose we are given a lattice $L$ in~$\Osc_1=H\rtimes \RR$ and we want to determine the spectrum of~$L\backslash \Osc_1$. First, we~reduce the problem to the situation where $L$ is normalised and unshifted.

Let $L$ be arbitrary. Then $L\cap H$ is isomorphic to a discrete Heisenberg group $H^r_1(\ZZ)$ for a uniquely determined natural number $r$. We~put $r(L):=r$. The projection of~$L\subset\Osc_1=H\rtimes\RR$ to the second factor is generated by a uniquely determined element $\lambda$ of~$\pi\NN_{>0}\cup \{\lambda_0+2\pi\ZZ\mid \lambda_0=\pi/3,\pi/2,2\pi/3\}$. We~set $\lambda(L):=\lambda$. We~choose adapted generators $\alpha$, $\beta$, $\gamma$, $\delta$ (see Definition~\ref{def:adapted}) and determine the type of~$L$, i.e., the isomorphism class of~$L$ considered as a discrete oscillator group, using Proposition~\ref{TF}.
The commutator group $[L\cap H, L\cap H]$ is generated by an element $(0,h^2,0)$ for some $h\in\RR_{>0}$. We~put $S:=h^{-1}\cdot I_2$. Then $L':=F_S(L)$ is a normalised lattice. Now we compute the number $\bs_{L'}$ for $L'$ as in Definition~\ref{de:inv}, where we can use the adapted generators $F_S(\alpha)$, $F_S(\beta)$, $F_S(\gamma)$, $F_S(\delta)$. We~put $u=-\lambda^{-1}\cdot\bs_{L'}$ and $F=F_u\circ F_S$.
Then $L''=F_u(L')=F(L)$ is normalised and unshifted.
The right regular representation~$\big(\rho,L^2(L\backslash\Osc_1)\big)$ of~$\Osc_1$ is isomorphic to the pullback $F^*\big(\rho'',L^2(L''\backslash\Osc_1)\big)$ of the right regular representation~$\rho''$ on~$L^2(L''\backslash\Osc_1)$. In~particular, the irreducible subrepresentations of~$L^2(L\backslash\Osc_1)$ are the pullbacks of the irreducible subrepresentations of~$L^2(L''\backslash\Osc_1)$. Hence, the spectrum of~$L\backslash \Osc_1$ can be computed from that of~$L''$ using equation~(\ref{Fiso}).

Consequently, we~may assume that the lattice $L\subset\Osc_1$ is normalised and unshifted. If~not already done, we~now determine the type of~$L$. This can be done as explained in the general case choosing adapted generators and using Proposition~\ref{TF}. Alternatively, we~can use Corollary~\ref{co:iso}, which does not require adapted generators.
Since also for the following main result adapted generators are not necessarily needed, we~relax this condition: The subgroup $L\cap Z(H)$ is generated by $\gamma:=(0,1/r,0)$ for $r=r(L)\in\NN_{>0}$. Choose $\alpha,\beta\in L\cap H$, such that $[\alpha,\beta]=\gamma^r$. Moreover, choose the remaining generator $\delta$ of~$L$ such that its projection~$\lambda$ to the $\RR$-factor equals~$\lambda(L)$.
For instance, adapted generators $\alpha$, $\beta$, $\gamma$, $\delta$ satisfy these conditions.

Let $\bar \alpha$, $\bar\beta$ and $\bar\delta$ be the projections of~$\alpha$, $\beta$ and $\delta$ to $H/Z(H)\cong \CC\cong\RR^2$. We~define the matrix $T:=(\bar\alpha,\bar\beta)$ with columns $\bar \alpha$ and $\bar \beta$.
Put $q=\ord(\lambda)$.
For $q\in\{2,3,4,6\}$, the group $\ZZ_q=\big\la T^{-1} {\rm e}^{{\rm i}\lambda} T\big\ra$ acts on the left of~$\ZZ^2\setminus \big\{(0,0)^\top\big\}$. The orbit space of this action is denoted by~${\mathcal O}_q$.
Obviously, for $l,k\in\ZZ$, the number $\big\|T\cdot (l,k)^\top\big\|$ depends only on the orbit of~$(l,k)^\top$.

If $\lambda\in 2\pi\,\NN_{>0}$, we~define in addition a pair of integers $(\iota_1,\iota_2)$ by $\bar\delta=\frac{\iota_1}r\bar\alpha+\frac{\iota_2}r\bar\beta$.

\begin{Theorem}\label{theo}
Let $L$ be a normalised and unshifted lattice in~$\Osc_1$. The right regular representation~$\rho$ of~$\Osc_1$ on~$L^2(L\backslash\Osc_1)$ decomposes as a sum $\cH_0\oplus\cH_1$ into two summands, which we will describe in the following.
\begin{enumerate}\itemsep=0pt
\item[$1.$] If $q:=\ord(\lambda)\in\{2,3,4,6\}$, then the representation~$\cH_0$ is equivalent~to
\begin{gather}\label{EH0q>1}
\cH_0\cong\bigoplus_{n\in \Z} \cC_{n/\lambda} \oplus \bigoplus
_{(l,k)\in\cO_q}\bigoplus_{K=0}^{|1+q\kappa|-1} {\mathcal S}_{a(l,k)}^{\tau(K)},
\end{gather}
where $\lambda=2\pi/q+2\pi\kappa$, $\tau(K)=K/|1+q\kappa|\in\RR/\ZZ$
and $a(l,k)= \big\|T\cdot (l,k)^\top\big\|$.
\\
If $q:=\ord(\lambda)=1$, then $\cH_0$ is equivalent to
\begin{gather}\label{EH0q=1}
\cH_0\cong\bigoplus_{n\in \Z} \cC_{n/\lambda} \oplus \bigoplus _{\substack{(l,k)\in\Z^2,\\(l,k)\not=(0,0)}} \bigoplus_{j=0}^{\kappa} {\mathcal S}_{a(l,k)}^{\tau(K_j(l,k))},
\end{gather}
where $\kappa$, $\tau(K)$ and $a(l,k)$ are defined as before and $K_j(l,k):=j-\frac1{r}(\iota_1 k-\iota_2l)$.
\item[$2.$]
For $q:=\ord(\lambda)\in\{2,3,4,6\}$ denote by $q_0\in\{2,3\}$ the smallest prime factor of~$q$. Then the representation~$\cH_1$ is equivalent~to
\begin{gather*}
\cH_1\cong \bigoplus_{m\in \Z_{\not=0}} \bigoplus_{n\in\Z} \bm(m,n)\cdot \cF_{rm,\,\frac n{q\lambda}},
\end{gather*}
where $\bm(m,n)$ depends on the type of~$L$ in the following way.
\begin{enumerate}\itemsep=0pt
\item[$(a)$] If $L$ is of type $L_r^q$, then
\begin{gather*}
\bm(m,n) =\begin{cases}
\big[\frac1q \big({r|m|-\rem_q(\sgn(\lambda m)\cdot n)}\big)\big]+1,
& \text{if\quad $n \equiv \overline{rm}$ {\rm mod}\ $q_0$},
\\[.5ex]
\big[\frac1q\big( r|m|-\rem_q(\sgn(\lambda m)\cdot n)-2\big)\big]+1,
& \text{if\quad $n\not\equiv \overline{rm}$ {\rm mod} $q_0$},
\end{cases}
\end{gather*}
where $\overline{rm}=\sgn(\lambda)rm$ if $3\mid q$ and $\overline{rm}=0$ otherwise.
\item[$(b)$]
 If $q_0\mid r$ and $L$ is of type $L_r^{q,+}$, then
\begin{gather*}
\bm(m,n) =
\begin{cases}
\frac{r|m|}{q},&\text{if\quad $q_0\nmid m$ and $q\mid r$},
\\[1ex]
\big[\frac1q\big(r|m|-\rem_q(\sgn(\lambda m)\cdot n)\big)\big]+1,
& \text{if\quad $(q_0\mid m$ or $q\nmid r)$}
\\
&\mbox{and}\quad q_0\mid m+n,
\\[1ex]
\big[\frac1q\big(r|m|-\rem_q(\sgn(\lambda m)\cdot n)-2\big)\big]+1,
&\text{if\quad $(q_0\mid m$ or $q\nmid r)$}\\ &\mbox{and}\quad q_0\nmid m+n.
\end{cases}
\end{gather*}
\end{enumerate}
If $\ord(\lambda)=1$ and $L$ is of type $L^1_{r_0}$, then $\cH_1$ is equivalent to
\begin{gather*}
\cH_1\cong\bigoplus_{m\in\Z _{\not=0}}\bigoplus_{n\in \Z} r_0|m|\cdot {\mathcal F}_{rm,\frac {r_0n}{r\lambda}}.
\end{gather*}
\end{enumerate}
\end{Theorem}

\begin{Remark}\label{fR}\quad
\begin{enumerate}\itemsep=0pt
\item If $\alpha$, $\beta$, $\gamma$, $\delta$ are adapted generators, then the group $\ZZ_q=\big\la T^{-1} {\rm e}^{{\rm i}\lambda} T\big\ra$ is generated by the matrix $S_q$ defined in (\ref{Sq}).
\item We know from Theorem~\ref{th:classUptoInner} and Proposition~\ref{pr:H0} that
for $q\in\{3,4,6\}$ the subrepresentation~$\cH_0$ depends only on~$\lambda$. At first glance, this is not reflected by Formula~(\ref{EH0q>1}), which contains besides $\lambda$ also $T$. However, if $\alpha$, $\beta$, $\gamma$, $\delta$ are adapted generators (which always can be chosen), then
$\big\| T\cdot(l,k)^\top \big\|=\big\|T_{\mu,\nu}^{-1}(l,k)^\top\big\|$ holds for all $l,k\in\ZZ$, where $T_{\mu,\nu}$ is defined as in (\ref{Tmunu2}) and $\mu+{\rm i}\nu={\rm i}$ for $q=4$ and $\mu+{\rm i}\nu={\rm e}^{{\rm i}\pi/3}$ for $q\in\{3,6\}$.
\item As already noticed in Remark~\ref{R73}, the representations ${\mathcal S}_{a(l,k)}^{\tau(K)}$ appearing in~\eqref{EH0q>1} can be isomorphic for different $(l,k)$ since possibly $a(l,k)=a(l',k')$ even if $(l,k)$ and $(l',k')$ are not on the same $\ZZ_q$-orbit. Analogously, summands in~\eqref{EH0q=1} can be isomorphic. This problem already appeared in the examples in~\cite{Br}, where also the following discussion for order $q=4$ can be found. If~$q=4$, then $a(l,k)=l^2+k^2$. The pairs $(l,k)$ and $(l',k')$ are on the same $\ZZ_q$-orbit if and only if $l+{\rm i}k=z(l'+{\rm i}k')$ for some $z\in\{\pm1,\pm {\rm i}\}$, i.e., if and only $l+{\rm i}k$ and $l'+{\rm i}k'$ generate the same ideal in~$\ZZ[{\rm i}]$. Hence, for a given $a\in\NN_{>0}$, the representation~${\mathcal S}_a^{\tau(K)}$ appears with multiplicity $m(a)$, where $m(a)$ is the number of ideals of norm $a$ in~$\ZZ[{\rm i}]$. Let~$a=2^{k_0}p_1^{k_1}\cdots p_\mu^{k_\mu} p_{\mu+1}^{k_{\mu+1}} \cdots p_\nu^{k_\nu}$ be the prime factorisation of~$a$, where $p_1,\dots, p_\mu$ are the prime factors that are congruent to 1 modulo 4 and $p_{\mu+1},\dots,p_\nu$ are the remaining odd prime factors. Then $a$ is the norm of an ideal in~$\ZZ[{\rm i}]$ if and only if $k_{\mu+1},\dots,k_\nu$ are even. In~this case, $m(a)=\prod_{j=1}^\mu(k_j+1)$. Classical density results, see, e.g., \cite[Section~40]{H}, imply that the number $M(a)=\sum_{j=1}^a m(j)$ of ideals of norm at most~$a$ in~$\ZZ[{\rm i}]$ satisfies the asymptotics $\lim\limits_{a\to \infty}M(a)/a=\pi/4$.
\end{enumerate}
\end{Remark}

\begin{proof}[Proof of Theorem~\ref{theo}.]
Let $L$ be a normalised unshifted lattice in~$\Osc_1$. Denote by $\Gamma=\phi(L)$ the corresponding lattice in~$G$.
There is an inner automorphism $\varphi$ of~$G$ such that $\varphi(\Gamma)$ is a~standard lattice. Since $\ph$ does not change the type of the lattice and
since $\lambda(L)=\lambda(\ph(\Gamma))$,
we~obtain the assertion for $\cH_1$ by applying Propositions~\ref{pr:2piH1},~\ref{P1},~\ref{P12},~\ref{P13} and~\ref{Pord3}.

Now, we~want to prove the assertion for $\cH_0$. Assume that $q=\ord(\Gamma)=1$. There is an~inner automorphism $\varphi$ of~$G$ such that $\varphi(\Gamma)=\Gamma_r(\lambda,\mu,\nu,\xi_0,z_0)$ with $\lambda=\lambda(\Gamma)\in 2\pi\,\NN_{>0}$, $\mu+{\rm i}\nu=T^{-1}\cdot {\rm i}\in\textbf{H}$, $\xi_0=\frac{\iota_1}{r}+{\rm i}\frac{\iota_2}{r}$ and $z_0=\frac{1}{2}\rem_2(\iota_1\iota_2r)$ with $\iota_1,\iota_2\in\ZZ$ defined as above.
Indeed, choose $S\in\SO(2,\RR)$ such that $ST=T_{\mu,\nu}^{-1}$. Then $F:=\phi F_{S}\phi^{-1}$ maps $\Gamma$ to the normalised unshifted lattice generated by
\begin{gather*}
\bigg(\frac 1{\sqrt\nu},z_\alpha,0\bigg),\quad
\bigg({-}\frac{\mu}{\sqrt\nu}+{\rm i}\sqrt\nu,z_\beta,0\bigg),\quad
\bigg(0,\frac{1}{r},0\bigg),\quad
\bigg(\iota_1\frac 1{\sqrt\nu}+\iota_2\bigg({-}\frac{\mu}{\sqrt\nu}+{\rm i}\sqrt\nu\bigg),z_\delta,\lambda\bigg)
\end{gather*}
for suitable $z_\alpha,z_\beta,z_\delta\in\RR$, which is isomorphic to $\Gamma_r(\lambda,\mu,\nu,\xi_0,z_0)$ under inner automorphisms $\big($take $\phi F_\eta\phi^{-1}$ with $\eta=-z_\beta\frac 1{\sqrt\nu}+z_\alpha\big({-}\frac{\mu}{\sqrt\nu}+{\rm i}\sqrt\nu\big)$ and Lemma~\ref{pr:shift}$\big)$. By~Proposition~\ref{2piH0}, we~ob\-tain that $\cH_0$ for $\Gamma$ is equivalent to a representation given by formula~\eqref{EH0q=1} but with $a(l,k)=\big\|T_{\mu,\nu}^{-1}(l,k)^\top\big\|$.
Since $\big\|T_{\mu,\nu}^{-1}(l,k)^\top\big\|=\big\|T(l,k)^\top\big\|$, the assertion for $q=1$ follows.

Now, assume that $q$ is in~$\{2,3,4,6\}$. As~before there is an inner automorphism $\ph$ such that~$\varphi(\Gamma)$ is a standard lattice. Let~$\mu$ and $\nu$ be the corresponding parameters for $\ph(\Gamma)$.
Hence, by Proposition~\ref{pr:H0}, $\cH_0$ is equivalent to
\begin{gather*}
\cH_0=\bigoplus_{n\in \Z} \cC_{n/\lambda} \oplus \bigoplus _{\substack{(l',k')\in\Z^2/\Z_q\\(l',k')\not=(0,0)}} \bigoplus_{K=0}^{|1+q\kappa|-1} {\mathcal S}_{a'(l',k')}^{\tau(K)},
\end{gather*}
 where $a'(l',k')=\big\|T_{\mu,\nu}^{-1}\cdot(l',k')^\top\big\|$. It~remains to show that this representation is indeed equal to the one in~\eqref{EH0q>1}. We~know that there exist matrices $M\in\SL(2,\ZZ)$ and $S\in\SO(2,\RR)$ such that $STM=T_{\mu,\nu}^{-1}$.
Since $S_q=T_{\mu,\nu}{\rm e}^{-{\rm i}\lambda}T_{\mu,\nu}^{-1}$ holds, the map
\begin{gather*}
M\colon\quad \Z^2\longrightarrow \ZZ^2,\qquad (l',k')^\top\longmapsto (l,k)^\top:=M\cdot (l',k')^\top
\end{gather*} descends to a map of orbit spaces $M\colon \big(\ZZ^2\setminus\{(0,0)\}\big)/\ZZ_q\rightarrow \cO_q$.
Moreover, we~have $a'(l',k')=\big\|T_{\mu,\nu}^{-1}(l',k')^\top\big\|=\big\|STM(l',k')^\top\big\|=\big\|T(l,k)^\top\big\|=a(l,k)$ and the assertion for $q\in\{2,3,4,6\}$ follows.
\end{proof}

\LastPageEnding

\end{document}